\documentclass[reqno, a4paper,12pt]{amsart}
\usepackage[authoryear]{natbib}

\usepackage{amsmath}
\renewcommand{\l}{l}
\usepackage{amssymb}
\usepackage{amsthm} 
\usepackage{mathrsfs}
\usepackage{mathabx} 
\usepackage{dsfont}
\usepackage{bm}
\usepackage{bbm}
\usepackage[bmargin=4cm,tmargin=4cm,rmargin=2.7cm,lmargin=2.6cm]{geometry}
\newtheorem{theorem}{Theorem}[section]
\newtheorem{lemma}[theorem]{Lemma}

\newtheorem{definition}[theorem]{Definition}
\newtheorem{corollary}[theorem]{Corollary}
\newtheorem{remark}[theorem]{Remark}
\newtheorem{proposition}[theorem]{Proposition}
\newtheorem{question}[theorem]{Question}

\newcommand{\vertiii}[1]{{\left\vert\kern-0.25ex\left\vert\kern-0.25ex\left\vert #1 
    \right\vert\kern-0.25ex\right\vert\kern-0.25ex\right\vert}}

\DeclareBoldMathCommand\balpha{\alpha}
\DeclareSymbolFont{largesymbolsA}{U}{txexa}{m}{n}
\SetSymbolFont{largesymbolsA}{bold}{U}{txexa}{bx}{n}
\DeclareFontSubstitution{U}{txexa}{m}{n}
\DeclareMathSymbol{\varprod}{\mathop}{largesymbolsA}{16}

\title{From Parseval to Grothendieck and a little beyond,  part I}

\author{Ron Blei}

\email{Ron Blei, blei@math.uconn.edu}

\address{Department of Mathematics, University of Connecticut, Storrs, CT 06269}

\date{\today}

\begin{document}

\maketitle

\numberwithin{equation}{section}

\section{\bf{Foreward}} \label{s-1}
\ \emph{Le th\'{e}or\`{e}me fondamental de la th\'{e}orie metrique des produits tensoriels}
had first appeared, so dubbed, in Alexander Grothendieck's landmark work \citep{Grothendieck:1956} that at the time went largely unnoticed.  Fifteen years later in  another landmark work \citep{Lindenstrauss:1968},  \emph{the fundamental theorem}, its essence distilled, was reformulated  
as \emph{the Grothendieck inequality}, and, since then, has been applied, extended, and reinterpreted in diverse contexts and settings  \citep{pisier2012grothendieck}.

Our work on and around the Grothendieck inequality -- a fundamental statement about Euclidean spaces -- began four or so decades ago with the observation that Littlewood's  
mixed-norm inequality and the Khintchin $(L^1_R \hookrightarrow L^2)$-inequality are, in a precise sense, equivalent.  Noting that Littlewood's inequality is a special case of Grothendieck's, and therefore so is Khintchin's, we considered a question that seemed natural to us:  can the Grothendieck inequality be derived from the Khintchin inequality?  The answer was \emph{yes}, and the key an 'upgrade' of the Khintchin inequality \citep{Blei:1977uq}.

An 'upgrade' here means this.  Each of the aforementioned inequalities is equivalent to existence of  \emph{dual maps}:
 $(l^2 \hookrightarrow L^{\infty})$-\emph{interpolants} (the Khintchin inequality),  $(l^{\infty,2} \hookrightarrow M)$-\emph{interpolants} (Littlewood's inequality), and  $(l^2 \hookrightarrow L^{\infty})$-\emph{representations} (the Grothendieck inequality), each in effect expressing a property of Euclidean space.  An 'upgrade' in our context refers to existence of dual maps with an additional property, not \emph{a priori} guaranteed by the corresponding inequality. 
The task at hand -- a derivation of the Grothendieck inequality from Khintchin's -- required control over $(l^2 \hookrightarrow L^{\infty})$-interpolants that was not provided by the Khintchin inequality \emph{proper}. The derivation generalized, and indeed opened paths to other 'Grothendieck-type' descriptions of Euclidean as well as non-Euclidean spaces. 

Our starting point is a 'dual' view of the Grothendieck inequality:  the existence of representations of Euclidean vectors by bounded functions (via $(l^2 \hookrightarrow L^{\infty})$-\emph{representations}), analogous to the classical representations by square-integrable functions  (via \emph{unitary maps}), wherein Parseval's identity looms large and central. Indeed, the Grothendieck inequality in its dual form manifests a Parseval-like formula, synthesized -- in our proof of it -- from the usual Parseval formula in a harmonic analysis setting.  In the first part of the monograph (Ch.2 - Ch.6), we explain the underlying motivation, review classical and neo-classical results, and describe upgrades of the Grothendieck inequality, as well as extensions of it to $l^p$-sequence spaces.  In the second part (Ch.7 - Ch.11), building on ideas detailed in the first, we derive Grothendieck-type representations of mixed-norm spaces, including spaces of operators.

The monograph is intended for readers with interest in analysis, and specifically with interest in the \emph{Grothendieck inequality} and topics around it. Since its appearance, the inequality has indeed gained in stature over the years, impacting en route functional analysis, harmonic analysis, probability theory, theoretical physics, and recently also theoretical computer science; see \citet{pisier2012grothendieck}.  This book, starting from first principles, tells a tale about a subject that -- we believe -- is still in its 'mid-life,' still open-ended with a range of unresolved issues.  

We assume familiarity with rudiments of functional and harmonic analysis usually covered in introductory courses.  And otherwise we will review and explain, as we move along, relevant concepts and results that are either studied in advanced courses, or found in the research literature at large. 
\tableofcontents

\newpage

\section{\bf{Classical and neoclassical notions}}\label{s0}
\subsection{Unitary maps} \ 
Let $A$ be a set, and let $l^2(A)$ denote the Euclidean space whose 'coordinate axes' are indexed by $A$.  That is, consider the space of complex-valued, square-summable functions on $A$, 
\begin{equation}
 l^2(A) \  \overset{\text{def}}{=} \ \big \{{\bf{x}} \in \mathbb{C}^A: \sum_{\alpha \in A} |{\bf{x}}(\alpha)|^2 \ < \infty \big \},
\end{equation}
equipped with the usual inner product, 
\begin{equation}
\langle {\bf{x}}, {\bf{y}} \rangle \  \overset{\text{def}}{=} \ \sum_{\alpha \in A} {\bf{x}}(\alpha) \overline{{\bf{y}}(\alpha)}, \ \ \ \ {\bf{x}} \in l^2(A), \  {\bf{y}} \in l^2(A),
\end{equation}
and Euclidean norm,
\begin{equation}
\|{\bf{x}}\|_2  \  \overset{\text{def}}{=}  \ \sqrt{\langle {\bf{x}},{\bf{x}} \rangle} \ = \ \bigg(\sum_{\alpha \in A} |{\bf{x}}(\alpha)|^2\bigg)^{\frac{1}{2}}, \ \ \ \  {\bf{x}} \in l^2(A).
\end{equation}
We view $A$ merely as an indexing set, with no assumptions about its cardinality, or structures therein. Summing nonnegative scalars $y_{\alpha}$ over $\alpha \in A$ has the usual meaning,
\begin{equation}
\sum_{\alpha \in A} y_{\alpha} \  \overset{\text{def}}{=} \sup_{\text{finite } F \subset  A} \ \sum_{\alpha \in F}y_{\alpha}.
\end{equation}\\
 If $(y_{\alpha}: \alpha \in A) \in \mathbb{C}^A$ is absolutely summable over $A$ (i.e.,  \ $\sum_{\alpha \in A} |y_{\alpha}| < \infty$), then its support $\big\{\alpha: y_{\alpha}   \neq 0 \big\} = A^{\prime}$  \ is at most countably infinite, and the sum \  $\sum_{\alpha \in A} y_{\alpha}$ is well-defined:   for any sequence of finite sets $E_j$ increasing to $A^{\prime}$,  
 \begin{equation}
E_j \subset E_{j+1}  \subset A^{\prime}, \ \ \  j \in \mathbb{N}, \ \ \ \ \ \  \bigcup_{j \in \mathbb{N}}E_j = A^{\prime}, 
 \end{equation}
 the limit of the partial sums
\begin{equation}
\lim_{j \rightarrow \infty} \sum_{\alpha \in E_j} y_{\alpha} \ \overset{\text{def}}{=} \ \sum_{\alpha \in A} y_{\alpha}\end{equation}\\
exists, and does not depend on the choice of $(E_j)_j$.
 
Let $(\Omega, \mu)$ be a general measure space, and consider   
 \begin{equation} \label{L2}
 L^2(\Omega, \mu) = \big \{ \mathbb{C}\text{-valued measurable functions $f$ on}  \ \Omega : \int_{\Omega} |f|^2 d\mu < \infty \big \}, 
 \end{equation}
 wherein elements are listed modulo the equivalence $$f \equiv g  \ \Leftrightarrow \ f = g \ \  \text{a.e.} (\mu), \ \ \ f \in L^2(\Omega,\mu), \ g \in L^2(\Omega,\mu).$$    
 The resulting space of equivalence class representatives -- denoted also by $L^2(\Omega, \mu)$ -- is a Hilbert space, with  
  inner product
  \begin{equation}
 \langle f, g \rangle_{L^2} \ \overset{\text{def}}{=} \ \int_{\Omega} f \ \overline{g}  \ d\mu, \ \ \ \  \ f \in L^2(\Omega,\mu), \ g \in L^2(\Omega,\mu), 
 \end{equation} 
and $L^2$-norm 
 \begin{equation}
 \|f\|_{L^2} \ \overset{\text{def}}{=}  \ \sqrt{\langle f, f \rangle_{L^2} } = \bigg (\int_{\Omega}|f|^2 d\mu\bigg)^{\frac{1}{2}}, \ \ \ \  \ f \in L^2(\Omega,\mu).
 \end{equation}

Assuming the dimension of  $L^2(\Omega, \mu)$  is at least the cardinality of $A$,  we take an orthonormal family 
$\{f_{\alpha}\}_{\alpha \in A} \subset L^2(\Omega, \mu)$, and consider the linear injection
 \begin{equation} \label{unitaryV1}
U \ (= U_{\{f_{\alpha}\}}):  l^2(A) \rightarrow L^2(\Omega, \mu)
\end{equation}
defined by
 \begin{equation} \label{unitaryV}
 U {\bf{x}} =  \sum_{\alpha \in A} {\bf{x}} (\alpha) f_{\alpha}, \ \ \ {\bf{x}} \in l^2(A).
 \end{equation}
 For  ${\bf{x}}$ and ${\bf{y}}$ in $l^2(A),$
  \begin{equation} \label{parseval} 
  \langle {\bf{x}}, {\bf{y}} \rangle \  
  =  \ \int _{\Omega}{U\bf{x}} \ \overline{U{\bf{y}}} \ d\mu \ = \ \langle {U\bf{x}}, U{\bf{y}} \rangle_{L^2},  \\\\
 \end{equation}\\
 and (therefore)
 \begin{equation} \label{parseval1}
 \|{\bf{x}}\|_2 = \|U{\bf{x}}\|_{L^2}.
 \end{equation}
 
\begin{definition} \label{unit}
Let  $H_1$ and $H_2$ be Hilbert spaces with respective inner products $\langle \boldsymbol{\cdot}, \boldsymbol{\cdot} \rangle_{H_1}$ and  $\langle  \boldsymbol{\cdot}, \boldsymbol{\cdot} \rangle_{H_2}.$ A map   \ $V:   H_1  \rightarrow  H_2$  is said to be \emph{unitary}, and $H_1$ \emph{unitarily equivalent} to $V[H_1]$, if     
\begin{equation} \label{presip}
\langle {\bf{x}}, {\bf{y}} \rangle_{H_1} = \langle {V(\bf{x}}), V({\bf{y}}) \rangle_{H_2}, \ \ \ \ {\bf{x}} \in H_1, \ {\bf{y}} \in H_2.
\end{equation}
\end{definition}

\noindent
Every unitary map is necessarily \emph{injective, linear}, and \emph{bounded}, and in particular norm-continuous.  In our case,  $U$ in \eqref{unitaryV} is unitary, and $l^2(A)$ unitarily equivalent to  $$L^2_{ \{f_{\alpha}\}} \  \overset{\text{def}}{=}  \ L^2\text{-closure of the linear span of} \  \{f_{\alpha}\}_{\alpha \in A}.$$

\begin{remark}[a classical example; e.g., {\citet[Ch. I]{Katznelson:1976}}]  \label{classical} \ 
Let \  $$A   =   \mathbb{Z} \  (\text{the integer group}), \ \ \ \ \Omega   = [0,2\pi) \     {=}  \  \mathbb{T}  \  (\text{the circle group}),$$   $$\mu   =  dt/2\pi \ \ (\text{normalized Lebesgue measure}),$$   and  $\{e_j \}_{j\in \mathbb{Z}} \subset L^2(\mathbb{T},dt/2\pi)$,  
\begin{equation} \label{classical1}
e_j(t) \ = \ \exp({\mathfrak{i}jt}), \ \ \ \ j \in \mathbb{Z}, \ \ t \in \mathbb{T}, \  \ \mathfrak{i} \overset{\text{def}}{=} \sqrt{-1}.
\end{equation}
Then  \ $U_{\{e_j\}}: l^2(\mathbb{Z})  \ \rightarrow  \ L^2(\mathbb{T},dt)$, 
\begin{equation} \label{unitary2}
 U_{\{e_j\}}{\bf{x}} \ = \ \sum_{j \in \mathbb{Z}} {\bf{x}}(j) \ e_j, \ \ \ \ \ {\bf{x}} \in l^2(\mathbb{Z}),
\end{equation}
is the inverse  of the Fourier transform map
 \begin{equation}
\mathcal{F}: L^2(\mathbb{T},dt/2\pi) \  \rightarrow \  l^2(\mathbb{Z}), 
\end{equation}
\begin{equation}
\big(\mathcal{F} g\big)(j) \ \overset{\text{def}}{=} \ \widehat{g}(j) \ = \ \frac{1}{2\pi} \int_{t \in \mathbb{T}} g(t) \ e_j(-t) \ dt, \ \ \ \ \ g \in L^2(\mathbb{T},dt/2\pi),  \ \ j \in \mathbb{Z},
\end{equation}\\
whereby\\
\begin{equation} \label{parseval3}
\frac{1}{2\pi}\int_{t \in \mathbb{T}} f(t) \  \overline{g(t)} \ dt \ = \ \sum_{j\in \mathbb{Z}} \widehat{f}(j) \ \overline{\widehat{g}(j)}, \ \ \ \ f \in L^2(\mathbb{T},dt/2\pi), \ \ g \in L^2(\mathbb{T},dt/2\pi).
\end{equation}\\
 \end{remark}
\begin{remark}[Parseval] \label{Par111} \ In a framework of harmonic analysis, if \ $\Omega$ is a compact abelian group with normalized Haar measure $\mu$,  and  $\{f_{\alpha}\}_{\alpha \in A}$ is a set of characters of \ $\Omega$, then  the assertion in  \eqref{parseval}  is generally referred to as  \emph{Parseval's formula}  (e.g., \eqref{parseval3} above). As such, Parseval's formula is a special case of the more general Plancherel theorem \cite[\S 1.6]{rudin2011fourier}: if  $\Omega$ is a locally compact abelian group with dual group $\Gamma$, then there exist Haar measures $\mu$ on $\Omega$ and $\nu$ on $\Gamma$, with the property that the transform
\begin{equation} \label{transform}
\widehat{f}(\gamma) \ \overset{def}{=} \ \int_{G} f \ \overline{\gamma} \ d\mu, \ \ \ \ \ f \in L^1(\Omega,\mu), \ \ \gamma \in \Gamma,
\end{equation}
is uniquely extendible to a unitary map from  $L^2(\Omega,\mu)$ onto $L^2(\Gamma,\nu)$, i.e.,\\
 \begin{equation} \label{plancherel}
 \int_{\omega \in \Omega} f(\omega) \ \overline{g}(\omega) \ \mu(d\omega) \ = \ \int_{\gamma \in \Gamma} \widehat{f}(\gamma) \ \overline{\widehat{g}}(\gamma) \ \nu(d\gamma), \ \ \ \ \ f \in L^2(\Omega,\mu), \ \ g \in L^2(\Omega,\mu). \ \ \ \ \ \ \ \ \ 
 \end{equation}\\  
(E.g., see \citet[Chapter VI]{Katznelson:1976} for the classical  $(\Omega, \mu) = (\mathbb{R}, dt)$ and $(\Gamma, \nu) = (\mathbb{R},dt/2\pi)$.) The statement in \eqref{plancherel} -- a non-trivial generalization of the analogous assertion for compact $\Omega$ -- is sometime referred to also as a \emph{Parseval formula}.  In the parlance of harmonic analysis, a \emph{Parseval formula} could also refer to relations between maps and their transforms in 'non-Hilbertian' settings  (e.g., \citet[Theorem I.7.1]{Katznelson:1976}.  In the case of a general measure space $(\Omega,\mu)$, we refer to  \eqref{parseval} as  \emph{Parseval's identity}.
\end{remark} 

 \begin{remark} [notation]  \label{notation2} \  Given an orthonormal system $\{f_{\alpha}\}_{\alpha \in A} \subset L^2(\Omega,\mu)$,  we let \  $U_{\{f_{\alpha}\}}$ denote the map from $\mathbb{C}^A$ to the class of series spanned by $\{f_{\alpha}\}_{\alpha \in A}$,  formally defined by
\begin{equation} \label{U}
{\bf{x}} \ \mapsto \ U_{\{f_{\alpha}\}}{\bf{x}} \ \overset{\text{def}}{=} \ \sum_{\alpha \in A} {\bf{x}}(\alpha) f_{\alpha}, \ \ \ \ {\bf{x}} \in \mathbb{C}^A.
\end{equation} 
When restricted to ${\bf{x}} \in l^2(A)$, the series  \ $\sum_{\alpha \in A} {\bf{x}}(\alpha) f_{\alpha}$ \  converges in $L^2(\Omega,\mu)$, and Parseval's identity holds.

Henceforth,  $(\Omega,\mu)$ will be a probability space, i.e., a measure space with $\mu(\Omega) = 1$.          
\end{remark}

\subsection{($l^2 \hookrightarrow L^p$)-representations} \label{narrative}
The effect of an orthonormal  $\{f_{\alpha}\}_{\alpha \in A} \subset L^2(\Omega,\mu)$  on the corresponding  $U_{\{f_{\alpha}\}}$ can be calibrated by tail-probabilities 
\begin{equation} \label{tail}
\mu \big (|U_{\{f_{\alpha}\}}{\bf{x}}| \geq t \big ), \ \ \ {\bf{x}} \in l^2(A), \ \ t > 0
\end{equation}
(e.g., \citet[\S3.4]{blei2011measurements}).  In the case  of \emph{complete} systems $\{f_{\alpha}\}_{\alpha \in A}$, the 'square power' estimates 
 \begin{equation} \label{tail1}
\mu \big (|U_{\{f_{\alpha}\}}{\bf{x}}| \geq t \big ) \ \leq \frac{\|{\bf{x}}\|_2^2}{t^2}, \ \ \ \ \ {\bf{x}} \in l^2(A),
\end{equation}
are optimal:  for infinite $A$, if $\{f_{\alpha}\}_{\alpha \in A}$ is an orthonormal basis of $L^2(\Omega,\mu)$, then
\begin{equation}
\sup \bigg \{\limsup_{t \rightarrow \infty} {t^q}\mu \big(|U_{\{f_{\alpha}\}}{\bf{x}}| > t 
\big) : {\bf{x}} \in \mathcal{B}_{l^2(A)} \bigg \} = \infty, \ \ \ \ \ \ q > 2.
\end{equation}
(Throughout, $\mathcal{B}_Y  \overset{\text{def}}{=}  \big\{ y \in Y: \|y\|_Y \leq 1 \big\}$  will denote the closed unit ball in a normed linear space $Y.$)  Moving away from  'completeness,' as systems become 'sparser,'  tail-probabilities get 'smaller' and become sub-Gaussian in the limit.
To illustrate the limiting case in a generic setting, we take   
\begin{equation} \label{generic}
\Omega  \ =  \ \Omega_A  \ \overset{\text{def}}{=} \ \{-1,1\}^A \  \ \  ( \ = \ \text{$\{-1,1\}$-valued functions on $A$})
\end{equation}
endowed with the product topology, i.e., weakest topology with respect to which the coordinate functions \  $r_{\alpha}: \Omega_A \rightarrow \{-1,1\},$  
\begin{equation} \label{coordinate}
r_{\alpha}(\omega) = \omega(\alpha), \ \ \ \omega \in \Omega_A, \ \ \ \alpha \in A,
\end{equation}
are continuous, and then let $\mu$  be the uniform probability measure $\mathbb{P}_A$ on the resulting Borel field (the product probability measure heuristically associated with independent tosses of a fair coin), i.e., 
\begin{equation}
 \mathbb{P}_A \ \overset{\text{def}}{=} \ \bigtimes_{\alpha \in A}\mathbb{P}_0 \ \ \ \ (\text{product measure}),
\end{equation}\\
 where $\mathbb{P}_0$ is the uniform probability measure on $\{-1,1\}$,  
\begin{equation}
\mathbb{P}_0(\{-1\}) = \mathbb{P}_0(\{1\}) = 1/2.
\end{equation}
Equipped with coordinate-wise multiplication (point-wise product of $\{-1,1\}$-valued functions on $A$), $\Omega_A$ becomes a compact abelian group with Haar measure $\mathbb{P}_A$.    
The dual group  \ $\widehat{\Omega}_A$  \  is the \emph{Walsh system}
{\begin{equation}\label{Character}
W_A = \bigcup_{k=0}^{\infty} W_{A,k},
\end{equation}
where \  $W_{A,0} = \{r_0\},$ \  $r_0 \equiv 1$ \ on \ $\Omega_A$,  and 
\begin{equation} \label{Walsh0}
W_{A,k} = \bigg\{\prod_{\alpha \in F}r_{\alpha}: F\subset A, \ \#F = k\bigg\} \ \    \ \text{(\emph{Walsh characters} of order $k$),} \ \ \  k \in \mathbb{N}.
\end{equation}
We distinguish (for reasons that will later become apparent) between 
 characters of odd and even order, 
\begin{equation} \label{odd}
W_{A,odd} \ \overset{\text{def}}{=} \ \bigcup_{k=0}^{\infty} W_{A,2k+1}, \ \ \ \ W_{A,even} \ \overset{\text{def}}{=} \ \bigcup_{k=0}^{\infty} W_{A,2k}.
\end{equation}

At one end, we have the \emph{Walsh system} $W_A$, an orthonormal basis for  $L^2(\Omega_A,\mathbb{P}_A)$, and at the other end, we have the  \emph{Rademacher system}
\begin{equation}
R_A \  \overset{\text{def}}{=} \ \{r_{\alpha}: \alpha \in A\} = W_{A,1},
\end{equation}
an 'optimally' sparse set of independent characters of $\Omega_A$ (e.g., \citet[Ch. II \S 1]{Blei:2001}).  Taking
\begin{equation} \label{represent0}
{\bf{x}} \rightarrow U_{R_A}{\bf{x}} \ = \ \sum_{\alpha \in A} {\bf{x}}(\alpha)r_{\alpha}, \ \ \ \ \ \ \ {\bf{x}} \in l^2(A),
\end{equation}
we obtain from the $(L^2_R \hookrightarrow L^p$)-\emph{Khintchin inequalities}  (Remark \ref{expsq} below) sub-Gaussian tail estimates
\begin{equation}\label{Khint}
\mathbb{P}_A\big(|U_{R_A}{\bf{x}}| > t \big) \leq e^{-\frac{t^2}{2}}, \ \ \ \ {\bf{x}} \in \mathcal{B}_{l^2(A)}, \ \ \ t > 0,
\end{equation}\\
which are sharp:
for infinite  $A$,
\begin{equation}\label{optim}
\sup \bigg \{\limsup_{t \rightarrow \infty} e^{t^q}\mathbb{P}_A\big(|U_{R_A}{\bf{x}}| > t \big) : {\bf{x}} \in \mathcal{B}_{l^2(A)} \bigg \} = \infty, \ \ \ \ \ \ q > 2.
\end{equation}\\
Generally,  under 'unitarity,' sub-Gaussian tails are always optimal:  for infinite  $A,$ if $V: l^2(A) \rightarrow L^2(\Omega,\mu)$ is unitary, then
\begin{equation}\label{optim1}
\sup \bigg \{\limsup_{t \rightarrow \infty} e^{t^q}\mu \big(|V{\bf{x}}| > t \big) : {\bf{x}} \in \mathcal{B}_{l^2(A)} \bigg \} = \infty, \ \ \ \ \ \ q > 2.
\end{equation} 
(This was pointed out to me by Gilles Pisier, citing results in \citet{Gordon1975}.)      
\begin{question} \label{question}
Could 'smaller' tails be achieved with maps 'slightly less' than unitary?  
\end{question}

\begin{definition}  \label{Linfrepresent} \ Let  $A$ be a set, and $(\Omega,\mu)$ a probability space. An injection 
\begin{equation}
\Phi: \  l^2(A)  \ \rightarrow  \ L^p(\Omega,\mu), \ \ \ \ \ 2 \ \leq  \ p  \ \leq \ \infty,
\end{equation}
is a $(l^2(A) \hookrightarrow L^p)$-\emph{representation} --  an $L^p$-\emph{representation of} $l^2(A)$ --   if 
\begin{equation} \label{extending1}
\begin{split}
 \Phi({\bf{x}}) \ &= \ \|{\bf{x}}\|_2 \ \Phi(\boldsymbol{\sigma}{\bf{x}}) \ \ \  \ {\bf{x}} \neq {\bf{0}},\\\
\Phi({\bf{0}}) \ &= \ {\bf{0}}, 
\end{split}
\end{equation}\\
where 
\begin{equation}  \label{sig}
{\boldsymbol{\sigma}{\bf{x}}} \  = \  \left \{
\begin{array}{lcc}
 {\bf{x}}/\|{\bf{x}}\|_2, &\ \ {\bf{x}} \neq {\bf{0}},  \\\
 {\bf{0}}, &\ \ {\bf{x}} = {\bf{0}} 
\end{array} \right.
\end{equation}
(Remark \ref{simple1} {\bf{i}}),\\
\begin{equation} \label{parseval2}
 \begin{split} 
  \int _{\Omega}\Phi({\bf{x}}) \ \Phi({\bf{y}}) \ d\mu \ & = \ \sum_{\alpha \in A} {\bf{x}}(\alpha) {\bf{y}}(\alpha) \ \overset{\text{def}}{=} \ {\bf{x}}\boldsymbol{\cdot}{\bf{y}} \ \   \text{(dot product)}, \ \ \  \ \ {\bf{x}} \in l^2(A), \ {\bf{y}} \in l^2(A),
\end{split}
 \end{equation}
 and\\
 \begin{equation} \label{bdtailp}
 \begin{split}
 \sup \bigg\{\|\Phi({\bf{x}})\|_{L^p}: & \  {\bf{x}} \in \mathcal{S}_{l^2(A)} \bigg \}  \  \overset{\text{def}}{=} \  \|\Phi\|_{l^2(A) \hookrightarrow L^p}  \ < \ \infty, \\\
 & \ \ \ \ \ \  \text{$\big(l^2(A) \hookrightarrow L^p \big)$-norm of $\Phi$}\\\
 \end{split}
 \end{equation}
 where $\mathcal{S}_{l^2(A)}$ is the unit sphere in $l^2(A)$ (Remark \ref{simple1} {\bf{i}}, {\bf{ii}}).
  \end{definition}

   \noindent
 \begin{remark}[notation and simple observations] \label{simple1} \ 
  \ \\
  
  \noindent
{\bf{i}}. \  The unit sphere in a normed linear space $Y$ is denoted by $$\mathcal{S}_Y  \overset{def}{=}  \big\{ {\bf{y}} \in Y: \|{\bf{y}}\|_Y = 1 \big\},$$
and  \  $\boldsymbol{\sigma}_Y: Y \rightarrow \mathcal{S}_Y \cup \{{\bf{0}}\}$ \  is defined by 
 \begin{equation}  \label{sigY}
{\boldsymbol{\sigma}_Y{\bf{y}}} \  = \  \left \{
\begin{array}{lcc}
 {\bf{y}}/\|{\bf{y}}\|_Y, &\ \ {\bf{y}} \in Y, \ \ {\bf{y}} \neq {\bf{0}},  \\\\
 {\bf{0}}, &\ \ {\bf{0}} \in Y. 
\end{array} \right.
\end{equation}\\
When $Y = l^p(A)$ ($1 \leq p \leq \infty$), we write $\boldsymbol{\sigma}_p$, and  $\boldsymbol{\sigma}$ when $p=2$. \ \\

\noindent 
{\bf{ii}}. \  If \ $Y$ and $Z$ are normed linear spaces, and $F: \ Y \rightarrow Z$, then
\begin{equation} \label{bdtailp1}
\|F\|_{Y \rightarrow  Z} \ \overset{def}{=} \ \sup \big\{\|F({\bf{y}})\|_Z:  \  {\bf{y}} \in \mathcal{S}_Y \big \},
\end{equation}  
which defines a norm on the space of $Z$-valued functions on $Y$.  If $F$ is an injection, then we write $\|F\|_{Y \hookrightarrow  Z}$. \ \\
 
\noindent
{\bf{iii}}. \  \eqref{extending1} implies that a $(l^2(A) \hookrightarrow L^p)$-representation $\Phi$ is determined by its values on  $\mathcal{S}_{l^2(A)}$, and, in particular, that  $\Phi(c{\bf{x}}) = c \Phi({\bf{x}})$ for all $c \in \mathbb{R}$ precisely when  $\Phi$ is an \emph{odd} function (i.e., $\Phi(-{\bf{x}}) = - \Phi({\bf{x}})$). \ \\

\noindent

\noindent
{\bf{iv}}. \  \eqref{bdtailp} implies the tail estimates
\begin{equation}
\mu \big (|\Phi({\bf{x}})| \geq t \big ) \ \leq \ \bigg(\frac{\|\Phi\|_{l^2(A) \hookrightarrow L^p} \ \|{\bf{x}}\|_2}{t}\bigg)^p, \ \ \ t > 0, \  \ \ \ \ {\bf{x}} \in {l^2(A)}, \ \ \ p \in [2,\infty), \ \ \ \  
\end{equation}
\begin{equation} \label{focus}
\inf \{t: \mu \big (|\Phi({\bf{x}})| \geq t \big ) = 0\} \  \leq \ \|\Phi\|_{l^2(A) \hookrightarrow L^{\infty}} \ \|{\bf{x}}\|_2.
\  \ \ \ \ {\bf{x}} \in {l^2(A)}, \ \   p = \infty. 
\end{equation}\\
Viewing the representations in Definition \ref{Linfrepresent} as 'slightly less' than unitary maps, we rephrase Question \ref{question}: 
\begin{question} \label{question2}
Do $L^{\infty}$-representations of $l^2(A)$ exist for arbitrary $A$?\\
\end{question}
 \end{remark}
\begin{remark} [Khintchin] \label{expsq} \ The inequalities 
  \begin{equation} \label{expsq1}
 \|U_{R_A}{\bf{x}}\|_{L^p} \ \leq \ \sqrt{p}, \ \ \ \ \ \text{sets}   \ A, \ \  \ p > 2,  \ \ \ {\bf{x}} \in  \mathcal{S}_{l^2(A)},
\end{equation}
stated and proved first in \citet{khintchine1923dyadische} for $p = 2n, \ n = 2,3,\ldots$, are arguably among the important mathematical discoveries in the 20th century.  We refer to \eqref{expsq1} as the \emph{Khintchin} $(L^2_R \hookrightarrow L^p)$-inequalities, which can be equivalently expressed as the sub-Gaussian tail estimates in \eqref{Khint}, or as the exponential-square integrability property       
\begin{equation} \label{expsq2}
 \sup \bigg \{ \int_{\Omega_A} e^{ |U_{R_A}{\bf{x}}|^2} \ d\mathbb{P}_A : \ \text{sets} \ A, \ \   {\bf{x}} \in \mathcal{S}_{l^2(A)}\bigg \} \ \overset{\text{def}}{=} \ \mathfrak{K} \ < \infty.
\end{equation}
(For a proof that \eqref{expsq1}, \eqref{Khint}, and \eqref{expsq2} are equivalent, see \citet[Appendix D] {stein2016singular}.)

By  \eqref{expsq1}, the unitary map $U_{R_A}$  is a $\big(l^2(A) \hookrightarrow L^p\big)$-representation  for all  $p \in (2,\infty)$.  Whereas  $U_{R_A}$ is a  $\big(l^2(A) \hookrightarrow L^{\infty}\big)$-representation \emph{only if} $A$ is finite -- by \eqref{optim}, but more simply, because \ $2\|U_{R_A}{\bf{x}}\|_{L^{\infty}} \geq \|{\bf{x}}\|_1$ \ for ${\bf{x}} \in \mathbb{C}^A$ with finite support. In due course, by using \eqref{expsq1}, we will 'slightly' perturb $U_{R_A}$ to construct  $\big(l^2(A) \hookrightarrow L^{\infty}\big)$-representations of $l^2(A)$ for arbitrary $A.$
 \end{remark}
 
\begin{remark}[representations \emph{vs.} unitary maps] \label{rem1}  Unlike unitary maps,  $(l^2 \hookrightarrow L^2)$-representations need not be linear, homogeneous, or continuous. (E.g., see proof of Proposition \ref{equivalence} in the next chapter.)  If a $(l^2 \hookrightarrow L^2)$-representation $\Phi$ commutes with complex conjugation, i.e., if
\begin{equation} \label{commute}
\Phi(\overline{{\bf{x}}}) = \overline{\Phi({\bf{x}})}, \ \ \  {\bf{x}} \in l^2(A),
\end{equation}
then $\Phi$ is unitary, and (therefore) linear.  Conversely, if an $(l^2 \hookrightarrow L^2)$-representation $\Phi$ is linear, then $\Phi$ is continuous, and (therefore) commutes with conjugation.  That is, 
\begin{proposition} \label{versus}
A $(l^2 \hookrightarrow L^2)$-representation is unitary if and only if it is linear.
\end{proposition} 

If \ $V: l^2(A) \rightarrow L^2(\Omega, \mu)$ is unitary, \ and \ $V({\bf{e}}_{\alpha})$ is real-valued ($\alpha \in A$), where $\{{\bf{e}}_{\alpha}\}_{\alpha \in A}$ is the canonical basis of $l^2(A)$,
\begin{equation}  \label{basis}
\begin{split}
{\bf{e}}_{\alpha}(\alpha^{\prime}) \ & = \  \left \{
\begin{array}{lcc}
1, &\quad \  \alpha = \alpha^{\prime}\\\\
0, &\quad \    \alpha \neq \alpha^{\prime}, 
\end{array} \right. 
\end{split}
\end{equation}
then $V$ is a $L^2$-representation of $l^2(A)$.  Otherwise, not every unitary map is a $(l^2 \hookrightarrow L^2)$-representation; e.g., in Remark \ref{classical}, if  $A = \mathbb{Z}$, \ $\Omega = [0,2\pi)$, and  $\mu = \text{Lebesgue measure}$, then $\Phi = U_{\{e_j\}}$   fails \eqref{parseval2}.

 If an $L^{\infty}$-representation of $l^2(A)$ is unitary, then the underlying $A$ must be finite:  e.g., by \eqref{optim1}; or, via the precise quantitative results in \citet{Alon:2006} (cf. Remark \ref{variant}).  In Chapter 4, we construct $(l^2(A) \hookrightarrow L^{\infty})$-representations for arbitrary $A$, uniformly bounded in the $(l^2 \rightarrow L^{\infty})$-norm, and arbitrarily close to  unitary maps in the $(l^2 \rightarrow L^2)$-norm.   Our focus here on uniformly bounded $(l^2 \hookrightarrow L^{\infty})$-representations is motivated mainly by Proposition \ref{equivalence} (in the next chapter), asserting that existence of such representations is equivalent to a phenomenon known as  \emph{the Grothendieck inequality}.
 \end{remark} 
 \section{\bf{The Grothendieck inequality}}\label{s0}
\subsection{Statement}   
  \emph{The Grothendieck constant} $\mathcal{K}_G$ is the 'smallest'  $K > 1$  such that for all sets $A$, finite sets $B$,  and scalar-valued arrays $(a_{uv})_{(u,v) \in B \times B} \in \mathbb{C}^{B\times B}$,  
\begin{equation} \label{grothen}
 \begin{split}
\sup \big \{\big|&\sum_{(u,v) \in B \times B}  a_{uv} \  {\bf{x}}_u \boldsymbol{\cdot}  {\bf{y}}_v  \big| :{\bf{x}}_u \in \mathcal{S}_{l^2(A)}, \  {\bf{y}}_v \in \mathcal{S}_{l^2(A)}, \ (u,v) \in B \times B \big \} \\\\
&\ \leq \ K \ \sup \big\{ \big|\sum_{(u,v) \in B \times B} a_{uv} \ s_u  t_v \big |: |s_u| = 1, \  |t_v| = 1, \  (u,v) \in B \times B \big \}.\\\
\end{split}
\end{equation}\\
 \emph{The Grothendieck inequality} is the assertion 
\begin{equation} \label{constant}
\mathcal{K}_G \ < \ \infty.
\end{equation}
$\mathcal{K}_G$ has two values, depending on the underlying scalar field:  if the supremum on the left of \eqref{grothen} is over  $\mathcal{S}_{l^2_{\mathbb{R}}(A)}$ ($\mathbb{R}$-valued vectors) and on the right  over  $\{s_u = \pm1, \ t_v = \pm1\}$, then 
$$\mathcal{K}_G  \  =  \  \mathcal{K}_G^{\mathbb{R}} \ \ \ \ \text{(the 'real' Grothendieck constant)};$$  if supremum is over $\mathcal{S}_{l^2(A)}$ ($\mathbb{C}$-valued vectors) on the left and  $\{|s_u|  =  1, \ |t_v| = 1\}$ on the right, then  
$$\mathcal{K}_G \  =  \ \mathcal{K}_G^{\mathbb{C}} \ \ \ \ \text{(the 'complex' Grothendieck constant)}.$$
 The numerical values of  \ $\mathcal{K}_G^{\mathbb{R}}$ \ and \ $\mathcal{K}_G^{\mathbb{C}}$ \  are open problems of long standing, and indeed of ongoing interest; e.g., see \citet{braverman2011grothendieck}.  
 
 Distinguishing between scalar fields on the right of \eqref{grothen}, 
 we have for finite  $B$ and $(a_u)_{u \in B} \in \mathbb{C}^B$,
\begin{equation} \label{Sid1}
\sum_{u\in B} |a_{u}| \ \leq \ \frac{\pi}{2} \sup \big \{\big|\sum_{u \in B} a_u s_u \big|: \ s_u = \pm1, \ u \in B \big\} \ =  \ \frac{\pi}{2}\big\|\sum_{u \in B} a_u r_u \big\|_{\infty},
\end{equation}
where $\pi/2$ is 'best' possible  \citep{seigner1997rademacher}, and then for \ $(a_{uv})_{(u,v) \in B \times B} \in \mathbb{C}^{B \times B}$, \\
\begin{equation} \label{CR1}
\begin{split}
\|{\bf{a}}\|_{\otimes_{\epsilon}} \ \overset{def}{=} \  \sup \big\{ \big|\sum_{(u,v) \in B \times B} a_{uv}& \ s_u  t_v \big |: |s_u| = 1, \  |t_v| = 1, \ (u,v) \in B \times B \big \} \\\
&  \leq \ \frac{\pi^2}{4}  \sup \big\{ \big|\sum_{(u,v) \in B \times B} a_{uv} \ s_u  t_v \big |: s_u =  \pm1, \  t_v = \pm1, \ (u,v) \in B \times B  \big \}\\\
& = \  \frac{\pi^2}{4} \big\|\sum_{(u,v) \in B \times B} a_{uv} \ r_u \otimes r_v \big \|_{\infty} \  \overset{def}{=} \  \frac{\pi^2}{4} \ \|{\bf{a}}\|_{\otimes_{\epsilon}^{\mathbb{R}}}.
\end{split}
\end{equation}
\begin{remark}[about the norms] \label{notation1} \  
A finite-dimensional scalar matrix  $${\bf{a}} = (a_{uv})_{(u,v) \in B \times B} \in \mathbb{C}^{B\times B} \ \ \  (\text{finite} \ B)$$ can be viewed as a representing matrix of a \emph{bilinear} functional acting on $\mathbb{C}$-valued functions defined on $B$, 
\begin{equation} \label{bic_0}
({\bf{s}},{\bf{t}}) \ \mapsto \ \sum_{(u,v) \in B \times B} a_{uv} \ s_u  t_v, \ \ \ \ {\bf{s}} = (s_u)_{u\in B} \ \in \mathbb{C}^B, \ \ {\bf{t}} = (t_u)_{u\in B} \ \in \mathbb{C}^B.
\end{equation}
With the supremum norm on  $\mathbb{C}^B$,
\begin{equation}
\|{\bf{s}}\|_{\infty} \ \overset{\text{def}}{=} \ \max\{|s_u|: u \in B\} \ \ \ \ \ {\bf{s}} =  (s_u)_{u\in B} \ \in \mathbb{C}^B,
\end{equation}\
the norm of the bilinear functional in  \eqref{bic_0} is\\ 
\begin{equation} \label{bic_00}
\|{\bf{a}}\|_{\otimes_{\epsilon}} \  \overset{\text{def}}{=} 
\ \sup \big\{ \big|\sum_{(u,v) \in B \times B} a_{uv} \ s_u  t_v \big |: |s_u|  = 1, \  |t_v| = 1 \big \}. 
\end{equation}
We can similarly view  \  ${\bf{a}} = (a_{uv})_{(u,v) \in B \times B} \in \mathbb{C}^{B\times B}$ as a representing matrix of the bilinear functional on $\big(l^2(A)\big)^B$ defined by
\begin{equation} \label{bic_1}
\begin{split}
({\bf{x}},{\bf{y}}) \ \mapsto \ \sum_{(u,v) \in B \times B}  a_{uv} \ & {\bf{x}}_u \boldsymbol{\cdot}  {\bf{y}}_v, \\\
& {\bf{x}} = ({\bf{x}}_u)_{u \in B} \ \in \  \big(l^2(A)\big)^B, \ \   {\bf{y}} = ({\bf{y}}_v)_{v \in B} \ \in \ \big(l^2(A)\big)^B,
\end{split}
\end{equation}\\
whence, with the supremum norm on $\big(l^2(A)\big)^B$ \\
\begin{equation}
\|{\bf{x}}\|_{l^{\infty}(B;l^2)} \ \overset{\text{def}}{=} \ \max\{\|{\bf{x}}_u\|_{l^2(A)}: u \in B\}, \ \ \ \ \ {\bf{x}} = ({\bf{x}}_u)_{u \in B} \ \in \  \big(l^2(A)\big)^{B},
\end{equation}\\ 
the norm of the functional in \eqref{bic_1} becomes\\
\begin{equation}
\|{\bf{a}}\|_{\otimes_{\epsilon,l^2(A)}} \  \overset{\text{def}}{=} \ \sup \big \{\big|\sum_{(u,v) \in B \times B}  a_{uv} \  {\bf{x}}_u \boldsymbol{\cdot}  {\bf{y}}_v  \big| :{\bf{x}}_u \in \mathcal{S}_{l^2(A)}, \  {\bf{y}}_v \in \mathcal{S}_{l^2(A)} \big \}.
\end{equation}\\
The Grothendieck inequality asserts that  \  $\|\boldsymbol{\cdot}\|_{\otimes_{\epsilon,l^2}}$  \ and  \ $\|\boldsymbol{\cdot}\|_{\otimes_{\epsilon}}$ \  are equivalent norms;
 namely, that there exist $K \in (1,\infty)$ with the property that for all sets $A$ and finite sets $B$,\\
\begin{equation} \label{grothen1}
\|{\bf{a}}\|_{\otimes_{\epsilon}} \ \leq \ \|{\bf{a}}\|_{\otimes_{\epsilon,l^2(A)}} \ \leq \ K \|{\bf{a}}\|_{\otimes_{\epsilon}}, \ \ \ \ \ \ {\bf{a}}  \in   \mathbb{C}^{B \times B}.
\end{equation}\\
In particular, 
\begin{equation} \label{constant}
\sup \big \{\|{\bf{a}}\|_{\otimes_{\epsilon,l^2(A)}}/\|{\bf{a}}\|_{\otimes_{\epsilon}}: A,  \ \text{finite} \ B, \ \text{non-zero} \ {\bf{a}} \in \mathbb{C}^{B \times B} \big \} \ {=} \ \mathcal{K}_G \ < \ \infty.
\end{equation}\\


In a setting of topological tensor products (e.g., \citet{Ryan:2002}), the norm \  $\|\boldsymbol{\cdot}\|_ {\otimes_{\epsilon}}$ in \eqref{bic_00} is an instance of the \emph{injective tensor norm}, sometimes denoted by $\|\boldsymbol{\cdot}\|_{{\vee \atop \otimes}}$ and referred to as the \emph{smallest crossnorm}  (e.g., \citet{Schatten:1943}).  In a framework of harmonic analysis, $\|\boldsymbol{\cdot}\|_ {\otimes_{\epsilon}}$ is the supremum norm of linear combinations of two-fold products of Rademacher or Steinhaus characters, a view that will facilitate here duality arguments and, in particular, applications of the Riesz-Kakutani representation theorem:  
 \begin{theorem} [F. Riesz  \citep{Riesz1909}, S. Kakutani \citep{kakutani1941concrete}] \label{RK} \  Let $\mathcal{X}$ be a compact Hausdorff space, $C(\mathcal{X})$ the space of scalar-valued continuous functions on $\mathcal{X}$, and $M(\mathcal{X})$ the space of scalar-valued regular Borel measures on $\mathcal{X}$.  If $\theta$ is a continuous linear functional on $C(\mathcal{X})$, then there exists a unique $\mu \in M(\mathcal{X})$, such that 
 \begin{equation}
 \theta(f) \ = \ \int_{\mathcal{X}} f \ d\mu, \ \ \ \ \ f \in C(\mathcal{X}),
 \end{equation} 
 and 
 \begin{equation}
 \|\theta\| \ \overset{\text{def}}{=} \ \sup_{\|f\|_{\infty} \leq 1} |\theta(f)| \ = |\mu|(\mathcal{X}) \ \overset{\text{def}}{=} \ \|\mu\|_M,
 \end{equation}
 where $|\mu|$ is the total variation measure of  \ $\mu$.
 \end{theorem}
  \end{remark}

\subsection{Dual statement} \ 
For a set $A$, let 
 \begin{equation} \label{dualg11}
  \mathcal{K}^{\star}(A) \ \overset{\text{def}}{=} \ \inf \bigg \{ \|\Phi\|_{l^2(A) \hookrightarrow L^{\infty}}: \big(l^2(A) \hookrightarrow L^{\infty}\big)\text{-representations} \ \Phi \bigg \},
 \end{equation}
  \ \ \ \ \  and  
 \begin{equation} \label{dualg}
  \mathcal{K}^{\star} \ \overset{\text{def}}{=} \ \sup_{A} \mathcal{K}^{\star}(A),
 \end{equation}
 where supremum ranges over all sets $A$. 
Like the Grothendieck constant,  \  $\mathcal{K}^{\star}$ is  two-valued: $\mathcal{K}^{\star} = \mathcal{K}^{*,\mathbb{R}}$ \ if infimum in \eqref{dualg11} ranges over $L^{\infty}$-representations of $l^2_{\mathbb{R}}(A)$ ($\mathbb{R}$-valued vectors), 
and \ $\mathcal{K}^{\star} = \mathcal{K}^{*,\mathbb{C}}$ \ if infimum ranges over $L^{\infty}$-representations of $l^2(A)$ ($\mathbb{C}$-valued vectors).  Representations in both cases are generally complex-valued:  for, if $\Phi$ is an $L^{\infty}$-representation  of $l^2_{\mathbb{R}}(A),$ and $\Phi({\bf{x}})$ is real-valued for every ${\bf{x}} \in l^2_{\mathbb{R}}(A)$, then $\Phi$ is necessarily unitary, and $A$ must be finite.  

The proposition below -- a statement about duality --  asserts that the Grothendieck inequality and the existence of uniformly bounded $(l^2 \hookrightarrow L^{\infty})$-representations are equivalent in the sense that
\begin{equation} \label{equivalence1}
\mathcal{K}^{\star} < \infty \ \Leftrightarrow \  \mathcal{K}_G < \infty.
 \end{equation}

\begin{proposition}[cf. {\citet[Proposition 1.1, \S 1.3]{blei2014grothendieck}}] \label{equivalence}
\begin{equation} \label{conn1}
 \big(\mathcal{K}^{*,\mathbb{R}}\big)^2/4 \ \leq \ \mathcal{K}_G^{\mathbb{R}} \ \leq \ \frac{\pi^2}{4}\big(\mathcal{K}^{*,\mathbb{R}}\big)^2,
 \end{equation}
 and
\begin{equation} \label{conn2}
  \big(\mathcal{K}^{*,\mathbb{C}}\big)^2/4 \ \leq \ \mathcal{K}_G^{\mathbb{C}} \ \leq \ \big(\mathcal{K}^{*,\mathbb{C}}\big)^2.
\end{equation} 
\end{proposition}
\begin{proof}   (See Remark \ref{notation} below for a refresher of basic notions.) \ To verify the right-side inequality in \eqref{conn1}, assume  \ $\mathcal{K}^{*,\mathbb{R}} < \infty,$ and then for an arbitrary set $A$ and $\epsilon > 0$, let $\Phi$ be an $L^{\infty}(\Omega,\mu)$-representation of $l^2_{\mathbb{R}}(A)$, such that
\begin{equation}
\sup_{{\bf{x}} \in \mathcal{S}_{l^2_{\mathbb{R}}(A)}} \|\Phi({\bf{x}}\|_{L^{\infty}} \ \leq \ \mathcal{K}^{*,\mathbb{R}} + \epsilon.
\end{equation}
Then, for a finite set $B$,   ${\bf{a}} = (a_{uv})_{(u,v) \in B \times B} \subset \mathbb{R}^{B\times B}$,  ${\bf{x}}_u \in \mathcal{S}_{l^2_{\mathbb{R}}(A)},$  and ${\bf{y}}_v \in \mathcal{S}_{l^2_{\mathbb{R}}(A)}$  ($u \in B,\ v \in B)$, we have
\begin{equation} \label{real1}
\begin{split}
\big|\sum_{(u,v) \in B \times B}  a_{uv} \ {\bf{x}}_u \boldsymbol{\cdot} {\bf{y}}_v  \big| \ &= \ \big|\sum_{(u,v) \in B \times B}  a_{uv} \ \int_{\Omega} \Phi({\bf{x}}_u )  \Phi({\bf{y}}_v ) \ d \mu \big| \\\\
& \leq \  \int_{\Omega} \big|\sum_{(u,v) \in B \times B}  a_{uv} \ \Phi({\bf{x}}_u )  \Phi({\bf{y}}_v ) \big| \ d\mu \\\\
& \leq \ \frac{\pi^2}{4}(\mathcal{K}^{*,\mathbb{R}} + \epsilon)^2 \  \|{\bf{a}}\|_{\otimes_{\epsilon}^{\mathbb{R}}}.
\end{split}
\end{equation}\\
($\|\boldsymbol{\cdot}\|_{\otimes_{\epsilon}^{\mathbb{R}}}$ is the injective tensor norm with supremum over $s_u  = \pm 1$ and $t_v  = \pm 1$  in \eqref{grothen}; see \eqref{CR1} above.)

For the left-side inequality, assume  \ $\mathcal{K}_G^{\mathbb{R}} < \infty,$ and let $A$ be an arbitrary set.  Then,  \eqref{grothen}  with \ $K = \mathcal{K}_G^{\mathbb{R}}$  \ implies that 
\begin{equation} \label{funct11}
\begin{split}
g \ \mapsto \ \sum_{({\bf{x}},{\bf{y}}) \in \mathcal{S}_{l^2_{\mathbb{R}}(A)} \times \mathcal{S}_{l^2_{\mathbb{R}}(A)}} &\widehat{g}(r_{{\bf{x}}}, r_{{\bf{y}}}) \  {\bf{x}}\boldsymbol{\cdot} {\bf{y}} ,\\\ 
& \ \ \  \text{Walsh polynomials } \ g \in C_{R_{\mathcal{S}_{l^2_{\mathbb{R}}(A)}} \times R_{\mathcal{S}_{l^2_{\mathbb{R}}(A)}}}\big(\Omega_{\mathcal{S}_{l^2_{\mathbb{R}}(A)}} \times \Omega_{\mathcal{S}_{l^2_{\mathbb{R}}(A)}}\big),\\\\
\end{split}
\end{equation}
determines a bounded linear functional with norm \  $\mathcal{K}_G^{\mathbb{R}}$, \ defined on the space of real-valued continuous functions on 
\begin{equation} \label{sphere}
\Omega_{\mathcal{S}_{l^2_{\mathbb{R}}(A)}} \times \Omega_{\mathcal{S}_{l^2_{\mathbb{R}}(A)}} \ \  \big( \ = \ \big\{-1,1\big\}^{\mathcal{S}_{l^2_{\mathbb{R}}(A)}}  \ \times \  \big\{-1,1\big\}^{\mathcal{S}_{l^2_{\mathbb{R}}(A)}} \ \big),
\end{equation}
whose Walsh transforms are supported by $R_{\mathcal{S}_{l^2_{\mathbb{R}}(A)}} \times R_{\mathcal{S}_{l^2_{\mathbb{R}}(A)}}$. ($R_{\mathcal{S}_{l^2_{\mathbb{R}}(A)}}$ is the Rademacher system indexed by the unit sphere $\mathcal{S}_{l^2_{\mathbb{R}}(A)}$ in $l^2_{\mathbb{R}}(A)$.)  By Riesz-Kakutani and Hahn-Banach, there exists a symmetric, real-valued Borel measure   $\nu$  on \  $\Omega_{\mathcal{S}_{l^2_{\mathbb{R}}(A)}} \times \Omega_{\mathcal{S}_{l^2_{\mathbb{R}}(A)}}$, \ such that
\begin{equation} \label{dotrep}
\begin{split}
 {\bf{x}} \boldsymbol{\cdot} {\bf{y}} \ &= \ \widehat{\nu}\big(r_{{\bf{x}}}, r_{{\bf{y}}}\big)\\\\ 
 &= \ \int_{\Omega_{\mathcal{S}_{l^2_{\mathbb{R}}(A)}} \times \Omega_{\mathcal{S}_{l^2_{\mathbb{R}}(A)}}} r_{\bf{x}} \otimes r_{\bf{y}} \ d \nu, \ \ \ \ \ {\bf{x}} \in \mathcal{S}_{l^2_{\mathbb{R}}(A)}, \ \ {\bf{y}} \in \mathcal{S}_{l^2_{\mathbb{R}}(A)},
 \end{split}
\end{equation}
and
\begin{equation}
  \|\nu\|_{M} \ = \ \mathcal{K}_G^{\mathbb{R}} \ \ \ (\text{total variation norm}).
\end{equation}\\
Now consider the probability measure  $\mu  =  |\nu|/ \mathcal{K}_G^{\mathbb{R}}$,  where $|\nu|$ is the total variation measure of $\nu$, and rewrite \eqref{dotrep} as \\
\begin{equation} \label{rep33}
\begin{split}
 {\bf{x}} \boldsymbol{\cdot} {\bf{y}}  \ = \ \mathcal{K}_G^{\mathbb{R}}  \ \int_{\Omega_{\mathcal{S}_{l^2_{\mathbb{R}}(A)}} \times \Omega_{\mathcal{S}_{l^2_{\mathbb{R}}(A)}}} r_{{\bf{x}}}\otimes r_{{\bf{y}}} \ h \  d\mu,  \ \ \ \ \ {\bf{x}} \in \mathcal{S}_{l^2_{\mathbb{R}}(A)}, \ \ {\bf{y}} \in \mathcal{S}_{l^2_{\mathbb{R}}(A)},
\end{split}
\end{equation}\\
where \ $h = \pm 1$ \ and \ $h \ d \mu = d \nu/\mathcal{K}_G^{\mathbb{R}}.$  For $j = 1,2,$  define\\
\begin{equation} \label{pol}
\begin{split} 
\phi_j: & \ \ \mathcal{S}_{l^2_{\mathbb{R}}(A)} \ \rightarrow \ L^{\infty}\big(\Omega_{\mathcal{S}_{l^2_{\mathbb{R}}(A)}} \times \Omega_{\mathcal{S}_{l^2_{\mathbb{R}}(A)}}, \mu\big),\\\
  \text{by}  \ \ \ \ \ \ \ \ \ \ \ \  & \ \ \\\ 
\phi_j({\bf{x}})& \ = \ r_{{\bf{x}}} \circ \pi_j  \ \sqrt{ \mathcal{K}_G^{\mathbb{R}} \ h} , \ \ \ \ {\bf{x}} \in \mathcal{S}_{l^2_{\mathbb{R}}(A)},
\end{split}
\end{equation}\\
where \ $\pi_1$ and $\pi_2$ are the canonical projections from the two-fold product \ $\Omega_{\mathcal{S}_{l^2_{\mathbb{R}}(A)}} \times \Omega_{\mathcal{S}_{l^2_{\mathbb{R}}(A)}}$ \ onto the first and second coordinates, respectively. Let
\begin{equation} \label{dualequiv}
\Phi({\bf{x}}) \ = \ \frac{\phi_1({\bf{x}}) + \phi_2({\bf{x}})}{2} \ + \ \mathfrak{i} \ \frac{\phi_1({\bf{x}}) - \phi_2({\bf{x}})}{2}, \ \ \ \ \  {\bf{x}} \in \mathcal{S}_{l^2_{\mathbb{R}}(A)}
\end{equation}
($\mathfrak{i} = \sqrt{-1}$), \ whence from \eqref{pol},
 \begin{equation} \label{estim22}
\sup_{{\bf{x}} \in \mathcal{S}_{l^2_{\mathbb{R}}(A)}} \|\Phi({\bf{x}}\|_{L^{\infty}}  \ \leq \ 2 \sqrt{ \mathcal{K}_G^{\mathbb{R}}} \ ,
\end{equation}\\
and from \eqref{rep33} and symmetry of $\nu$,
\begin{equation}
{\bf{x}} \boldsymbol{\cdot}{\bf{y}} \ = \ \int_{\Omega_{\mathcal{S}_{l^2_{\mathbb{R}}(A)}} \times \Omega_{\mathcal{S}_{l^2_{\mathbb{R}}(A)}}} \Phi({\bf{x}}) \ \Phi({\bf{y}}) \ d \mu, \ \ \ \ \ {\bf{x}} \in \mathcal{S}_{l^2_{\mathbb{R}}(A)}, \ \ {\bf{y}} \in \mathcal{S}_{l^2_{\mathbb{R}}(A)}.
\end{equation}\\
We extend $\Phi$ to ${\bf{x}} \in l^2_{\mathbb{R}}(A)$ by
\begin{equation} \label{extending}
\begin{split}
\Phi({\bf{x}}) \ &= \ \|{\bf{x}}\|_2 \ \Phi(\boldsymbol{\sigma}{\bf{x}}) \ \ \  \ {\bf{x}} \neq {\bf{0}},\\\
\Phi({\bf{0}}) \ &= \ {\bf{0}}. 
\end{split}
\end{equation}\\
(See \eqref{sig} for the definition of $\boldsymbol{\sigma}{\bf{x}}$.) 
Then, $\Phi$ is an $L^{\infty}$-representation of $l^2_{\mathbb{R}}(A)$ with \begin{equation}
  \|\Phi\|_{l^2(A) \hookrightarrow L^{\infty}} \ \leq \ 2 \sqrt{ \mathcal{K}_G^{\mathbb{R}}},
\end{equation}
and therefore,  
\begin{equation}
\mathcal{K}^{*,\mathbb{R}} \  \leq \ 2 \sqrt{  \mathcal{K}_G^{\mathbb{R}}}.
\end{equation}\\

The proof of \eqref{conn2} is similar. By using
\begin{equation}
\big\|\sum_{(u,v) \in B \times B}  a_{uv} \ \Phi({\bf{x}}_u )  \Phi({\bf{y}}_v ) \big\|_{L^{\infty}} \ \leq \ (\mathcal{K}^{*,\mathbb{C}} + \epsilon)^2 \  \|{\bf{a}}\|_{\otimes_{\epsilon}},
\end{equation}
 we obtain, as in \eqref{real1} (without $\frac{\pi^2}{4}$), the right-side inequality in \eqref{conn2}. 
 To verify the left-side inequality, consider the circle group 
\begin{equation} \label{circ}
\mathcal{T} \ = \ \{e^{\mathfrak{i}t}: t \in \mathbb{T}\},
\end{equation}
and   
then take the compact abelian group \  $\mathcal{T}^{\mathcal{S}_{l^2(A)}}$ (  = \ $\mathcal{T}$-valued functions on $\mathcal{S}_{l^2(A)}$), whose character group $\big({\mathcal{T}}^{\mathcal{S}_{l^2(A)}}\big)^{\wedge}$ is generated by the Steinhaus system $$\mathcal{Z}_{\mathcal{S}_{l^2(A)}}\ \overset{\text{def}}{=} \ \big\{\zeta_{{\bf{x}}}: {\bf{x}} \in \mathcal{S}_{l^2(A)}\big\},$$ \
 where   
\begin{equation}
\begin{split}
\zeta_{{\bf{x}}}: \ &\ \mathcal{T}^{\mathcal{S}_{l^2(A)}} \  \rightarrow \  \mathcal{T}, \ \ \ \ \ {\bf{x}} \in \mathcal{S}_{l^2(A)}, \\\\
\zeta_{{\bf{x}}}(\omega)& \ = \ \omega({\bf{x}}), \ \ \ \ \ {\bf{x}} \in \mathcal{S}_{l^2(A)}, \ \ \omega \in \mathcal{T}^{\mathcal{S}_{l^2(A)}}.
\end{split}
\end{equation}\\
(See Remark \ref{notation} below.)  From \ $\mathcal{K}_G^{\mathbb{C}}  <  \infty$, we obtain that \\
\begin{equation} \label{funct111}
\begin{split}
g \ \mapsto \ \sum_{({\bf{x}},{\bf{y}}) \in \mathcal{S}_{l^2} \times \mathcal{S}_{l^2}} & \widehat{g}\big(\zeta_{{\bf{x}}}, \zeta_{{\bf{y}}}\big)  \  {\bf{x}}\boldsymbol{\cdot} {\bf{y}} ,\\\
& \ \ \ \ \ \ \text{$\big(\mathcal{T}^{\mathcal{S}_{l^2(A)}}\big)^{\wedge}$-polynomials} \ g \in C_{\mathcal{Z}_{\mathcal{S}_{l^2}} \times \mathcal{Z}_{\mathcal{S}_{l^2}}},\\\
\end{split}
\end{equation}
determines a bounded linear functional on\\  
\begin{equation} \label{sphere}
C_{\mathcal{Z}_{\mathcal{S}_{l^2}} \times \mathcal{Z}_{\mathcal{S}_{l^2}}}\big(\mathcal{T}^{\mathcal{S}_{l^2}} \times \mathcal{T}^{\mathcal{S}_{l^2}}\big) 
\end{equation}\\
(continuous functions on $\mathcal{T}^{\mathcal{S}_{l^2}} \times \mathcal{T}^{\mathcal{S}_{l^2}}$ with spectra in  $\mathcal{Z}_{\mathcal{S}_{l^2}} \times \mathcal{Z}_{\mathcal{S}_{l^2}}$), 
and the proof proceeds as in the 'real' case.
\end{proof}
\begin{remark}[some basics] \label{notation} \  $\widehat{\Omega}$-\emph{polynomials} refer to linear combinations of characters in $\widehat{\Omega}$, where $\widehat{\Omega}$  is the dual of a compact abelian group $\Omega$. (A \emph{trigonometric polynomial} is a linear combination of exponentials, and a \emph{Walsh polynomial} is a  linear combination of Walsh characters.)  $\widehat{\Omega}$-polynomials are norm-dense in $C(\Omega)$  and \ $L^q(\Omega,\mu), \ q \in [1,\infty)$, \ and $\text{weak}^{\star}$-dense in $L^{\infty}(\Omega,\mu$) \ and \ $M(\Omega)$ (regular Borel measures on $\Omega$), \ where $\mu$ is  Haar measure of \ $\Omega$. (E.g., see \citet{rudin2011fourier}.)

The Steinhaus system $\mathcal{Z}_A = \{\zeta_{\alpha}: \alpha \in A \}$  ('complex' analogue of the Rademacher system $R_A$) is defined by 
\begin{equation}
\begin{split}
\zeta_{\alpha}: \ &\ \mathcal{T}^{A} \  \rightarrow \  \mathcal{T}, \ \ \ \ \ \alpha \in A, \\\
\zeta_{\alpha}(\omega)& \ = \ \omega(\alpha), \ \ \ \ \ \alpha \in A, \ \ \omega \in \mathcal{T}^{A},
\end{split}
\end{equation}\\
where $\mathcal{T}$ is the unit circle in $\mathbb{C}$ parameterized by $[0,2\pi)$, as per \eqref{circ}.
The Steinhaus system $\mathcal{Z}_A$ is a set of independent characters of the compact abelian group $\mathcal{T}^{A}$ that generates the dual group $\widehat{\mathcal{T}^A}$.  I.e.,  
\begin{equation}
\widehat{\mathcal{T}^A} \ = \ \bigcup_{k=0}^{\infty} \mathcal{Z}_{A,k},
\end{equation}\\
where   \ $\zeta_0 \equiv 1$ on $\mathcal{T}^A$, \ $\mathcal{Z}_{A,0} = \{\zeta_0\},$ \  and 
\begin{equation}
\mathcal{Z}_{A,k} \ = \ \bigg\{\prod_{\alpha \in F} \zeta_{\alpha}^{n_{\alpha}}: F \subset A, \ \# F = k, \ (n_{\alpha})_{\alpha \in F} \in \mathbb{Z}^F \bigg\}, \ \ \ \ k = 1, \ldots \ . 
\end{equation}

\end{remark}
\begin{remark} [proposition deconstructed, and a preview] \ The proof of Proposition \ref{equivalence} rests on the equivalence (via Riesz-Kakutani and Hahn-Banach) between  '$\mathcal{K}_G < \infty$'  and the assertion (cf. \citet{blei1976sidon}): \ \\

\noindent
for every set $A$, there exist  $\nu \in M(\Omega_{\mathcal{S}_{l^2(A)}} \times \Omega_{\mathcal{S}_{l^2(A)}})$ such that $\|\nu\|_M \leq \mathcal{K}_G$, and 
\begin{equation}
\widehat{\nu}(r_{{\bf{x}}},r_{{\bf{y}}}) \ = \ {\bf{x}} \boldsymbol{\cdot} {\bf{y}} \ \ \text{for all}  \ \ ({\bf{x}},{\bf{y}})  \in \mathcal{S}_{l^2(A)} \times \mathcal{S}_{l^2(A)}.
\end{equation}\\

\noindent
Based on this equivalence,  $\mathcal{K}_G < \infty$  implies the existence of  $\big (l^2(A) \hookrightarrow L^{\infty}\big)$-representations in \eqref{dualequiv}, uniformly bounded in the $(l^2(A) \rightarrow L^{\infty})$-norm, and thus \  $\mathcal{K}^{\star} < \infty$.  Notably, the  $(l^2 \hookrightarrow L^{\infty})$-representations in \eqref{dualequiv} are neither $(l^2 \rightarrow L^2)$-continuous nor homogeneous.
 
 \begin{question} \label{question1}  
Do \  $(l^2 \hookrightarrow L^{\infty})$-representations exist with properties not \emph{a priori} guaranteed by  \ $\mathcal{K}^{\star} < \infty$? \  (Cf. Remark \ref{rem1}.)
\end{question}

\noindent
In \S \ref{rev} and \S \ref{scalarparseval}, we construct uniformly bounded $\big(l^2 \hookrightarrow L^{\infty}\big)$-representations that are $\mathbb{R}$-homogeneous,  $(l^2 \rightarrow L^2)$-continuous, and arbitrarily close to unitary maps in the $(l^2 \rightarrow L^2)$-norm. These constructions are based in part on \emph{Riesz products}, which are reviewed in \S \ref{RieszP} and \S \ref{RieszP0}.  Existence of such representations, sharing properties with unitary maps (linearity excepted), is an \emph{upgrade} of the Grothendieck inequality in the precise sense stated in \S \ref{upgrade3}. 
\end{remark}
\subsection {Quadratic inequalities} \label{unitary}
  \  The quadratic version of the Grothendieck inequality is:  \ there exist $ K \in (1,\infty)$ such that for all sets $A$, finite sets  $B$, and scalar-valued arrays $(a_{uv})_{\{u, v\} \subset B}$,\\  
 \begin{equation} \label{grothenc}
 \begin{split}
\sup \big \{\big|\sum_{\{u, v \} \subset B} a_{uv}  \ {\bf{x}}_u \boldsymbol{\cdot} {\bf{x}}_v  \big| :&\  {\bf{x}}_{u} \in \mathcal{S}_{l^2(A)} \big \}\\\
& \ \leq \ K \ \sup \big\{ \big|\sum_{\{u,v\} \subset  B} a_{uv} \ s_{u}  s_{v}\big |: |s_{u}| = 1  \big \}.
\end{split}
\end{equation}
The 'smallest' $K$ in \eqref{grothenc} is denoted by $\widetilde{\mathcal{K}}_G$.  
Via polarization and symmetrization,  
\begin{equation}
\widetilde{\mathcal{K}}_G/4 \ \leq \mathcal{K}_G \ \leq \ \widetilde{\mathcal{K}}_G, \ \ \ \ \ \text{'real' or 'complex'}.
\end{equation}
(I do not know the exact relation between $\widetilde{\mathcal{K}}_G$ and $\mathcal{K}_G$.) \ \\ 

An injection $\Phi: l^2(A) \ \rightarrow L^2(\Omega, \mu)$ is a \emph{quasi}-$L^{\infty}$-representation of $l^2(A)$ if
\begin{equation} 
\begin{split}
\Phi({\bf{x}}) \ &= \ \|{\bf{x}}\|_2 \ \Phi(\boldsymbol{\sigma}{\bf{x}}) \ \ \  \ {\bf{x}} \neq {\bf{0}},\\\
\Phi({\bf{0}}) \ &= \ {\bf{0}}, 
\end{split}
\end{equation}
\begin{equation}
{\bf{x}} \boldsymbol{\cdot}{\bf{y}} \ = \ \int_{\Omega} \Phi({\bf{x}}) \ \Phi({\bf{y}}) \ d\mu, \ \  \ \ \ {\bf{x}} \in l^2(A), \  \ {\bf{y}} \in l^2(A), \ \ {\bf{x}} \neq  {\bf{y}},
\end{equation}
and
\begin{equation} \label{bd1}
 \sup \big\{\|\Phi({\bf{x}})\|_{L^{\infty}}: \  {\bf{x}} \in \mathcal{S}_{l^2(A)} \big \}  \  {=} \  \|\Phi\|_{l^2(A) \hookrightarrow L^{\infty}} \ < \ \infty.
 \end{equation}\\
 (Cf. Definition \ref{Linfrepresent}.). As in \eqref{dualg}, let 
 \begin{equation} \label{dualg1}
 \widetilde{\mathcal{K}}^{\star} \  \overset{\text{def}}{=} \  \sup_{A} \big(  \inf_{\Phi}   \|\Phi\|_{l^2(A) \hookrightarrow L^{\infty}} \big),
 \end{equation}
 where the infimum ranges over all quasi-$L^{\infty}$-representations of $l^2(A)$, and the supremum ranges over all sets $A$.   Depending on the underlying fields, $ \widetilde{\mathcal{K}}^{\star}$ has two values: \ $\widetilde{\mathcal{K}}^{*,\mathbb{R}}$ \ in the 'real' case, and \ $ \widetilde{\mathcal{K}}^{*,\mathbb{C}}$ \  in the 'complex' case.
 
 \begin{proposition} [cf. Proposition \ref{equivalence}]
\begin{equation}
(\widetilde{\mathcal{K}}^{*,\mathbb{R}})^2 \ \leq \ \widetilde{\mathcal{K}}_G^{\mathbb{R}} \ \leq \  \frac{\pi^2}{4}(\widetilde{\mathcal{K}}^{*,\mathbb{R}})^2,
\end{equation} 
and
\begin{equation}
\widetilde{\mathcal{K}}_G^{\mathbb{C}} \ = \ (\widetilde{\mathcal{K}}^{*,\mathbb{C}})^2.
\end{equation}
  \end{proposition}
 \begin{proof}[Sketch of proof] \ The verifications of  \ $\widetilde{\mathcal{K}}_G^{\mathbb{R}} \ \leq \   \frac{\pi^2}{4}(\widetilde{\mathcal{K}}^{*,\mathbb{R}})^2$ \ and \ $\widetilde{\mathcal{K}}_G^{\mathbb{C}} \ \leq \ (\widetilde{\mathcal{K}}^{*,\mathbb{C}})^2$ \  are identical to the verifications of the corresponding inequalities in  Proposition \ref{equivalence}.
 
 To verify \ $(\widetilde{\mathcal{K}}^{*,\mathbb{R}})^2 \ \leq \ \widetilde{\mathcal{K}}_G^{\mathbb{R}}$, \ use Riesz-Kakutani to obtain a measure $\nu$ on $\Omega_{\mathcal{S}_{l^2_{\mathbb{R}}(A)}}$,  such that \ $ \|\nu\|_{M} \ = \ \widetilde{\mathcal{K}}_G^{\mathbb{R}}$,  \ and
 \begin{equation}
 {\bf{x}} \boldsymbol{\cdot} {\bf{y}} \  = \ \int_{\Omega_{\mathcal{S}_{l^2_{\mathbb{R}}(A)}}} r_{{\bf{x}}} \ r_{{\bf{y}}}  \ d \nu, \ \ \ \ \ {\bf{x}} \in \mathcal{S}_{l^2_{\mathbb{R}}(A)}, \ \ {\bf{y}} \in \mathcal{S}_{l^2_{\mathbb{R}}(A)}, \ \ \ {\bf{x}} \neq {\bf{y}}.\\\\
 \end{equation}\\
Then,  taking the probability space \  $\big(\Omega_{\mathcal{S}_{l^2_{\mathbb{R}}(A)}},  |\nu|/\widetilde{\mathcal{K}}_G^{\mathbb{R}}\big),$ \ let
 \begin{equation}
 \begin{split}
 \Phi({\bf{x}}) \  &\overset{\text{def}}{=} \   \bigg(\sqrt{ \widetilde{\mathcal{K}}_G^{\mathbb{R}} \ h}\bigg) \ r_{{\bf{x}}},  \ \ \ \ \ \ \ {\bf{x}} \in \mathcal{S}_{l^2_{\mathbb{R}}(A)},\\\\
  \Phi({\bf{x}}) \  &\overset{\text{def}}{=} \ \|{\bf{x}}\|_2 \ \Phi(\boldsymbol{\sigma}{\bf{x}}), \ \ \ \ \ \ {\bf{x}} \in l^2_{\mathbb{R}}(A),
  \end{split}
  \end{equation}\\
 where $h = \pm1$ is measurable on $\Omega_{\mathcal{S}_{l^2_{\mathbb{R}}(A)}}$, and $$d\nu  \ = \ h \ d|\nu|.$$
 
To verify $(\widetilde{\mathcal{K}}^{*,\mathbb{C}})^2 \ \leq \ \widetilde{\mathcal{K}}_G^{\mathbb{C}}$, use the same argument, with Steinhaus systems in place of Rademacher systems. 
 \end{proof}
 
\begin{remark} [a variant] \label{variant}  \ In the 'real' case, removal of the 'absolute value' from \eqref{grothenc} leads to a distinctly different phenomenon:  there exist $K \in (1,\infty)$ such that for all sets $A$,  finite sets $B$, and $\mathbb{R}$-valued arrays $(a_{uv})_{\{u,v\} \subset B}$,
 \begin{equation} \label{grothencr}
 \begin{split}
\sup \big \{\sum_{\{u, v \} \subset B} a_{uv}  \ {\bf{x}}_u \boldsymbol{\cdot} {\bf{x}}_v   :&\  {\bf{x}}_{u} \in \mathcal{S}_{l^2_{\mathbb{R}}(A)} \big \}\\\
& \ \leq \ K \ \log (\# B) \ \sup \big\{ \sum_{\{u,v\} \subset  B} a_{uv} \ s_{u}  s_{v}: s_{u} = \pm 1 \big\},
\end{split}
\end{equation}
where $\log (\# B)$ is best possible; see  \citet{Charikar:2004} and \citet{Alon:2006}.  
That \eqref{grothencr} is optimal implies that for infinite $A,$ there can be no unitary $L^{\infty}$-representations of $l^2(A)$. (Cf. also \eqref{optim1}.)
 \end{remark}
 

 

 \section{\bf{The Khintchin, Littlewood, and Orlicz inequalities}} \subsection{Preliminaries} \ In this chapter we consider three classical precursors to the Grothendieck inequality, with emphasis on their dual statements.

We assume some familiarity with harmonic analysis on groups, and will use terminology and notation that is mostly standard. (E.g., Chapters 1 and 2 of  \citet{rudin2011fourier}.)  For a discrete abelian group $\Gamma,$  its compact dual  $\widehat{\Gamma}$ and normalized Haar measure $\mathfrak{m}$ on it, and $E \subset \Gamma$, the following spaces recur throughout:
\begin{equation} \label{iden2}
\begin{split}
L^p_{E} \  & {=} \ \big \{f \in L^p(\widehat{\Gamma},\mathfrak{m}): \widehat{f}(\gamma) = 0,  \ \gamma \in \Gamma \setminus E \big \}, \ \ \ \ \ 1 \leq p \leq \infty, \\\\
C_{E} \  & {=} \ \big \{f \in C(\widehat{\Gamma}): \widehat{f}(\gamma) = 0,  \ \gamma \in \Gamma \setminus E \big \},\\\\
M_E \ & {=}  \ \big \{\mu \in M(\widehat{\Gamma}): \widehat{\mu}(\gamma) = 0,  \ \gamma \in \Gamma \setminus E \big \}.
\end{split}
\end{equation}\\
The space of  $\mathbb{C}$-valued $\Gamma$-series spanned by $E$ is denoted by
 \begin{equation} \label{series1}
\mathfrak{S}_{E} \ = \ \bigg\{\sum_{\gamma \in E}  \phi(\gamma) \ \gamma: \ \phi \in \mathbb{C}^{E} \bigg\},
\end{equation}
and if $\phi$ in \eqref{series1} is restricted to a subspace of  $\mathbb{C}^E$, then the corresponding resulting subspace of $\mathfrak{S}_E$ will be formally so designated, e.g.,
\begin{equation} \label{series2}
\mathfrak{S}_E^{s} \ \overset{def}{=} \ \bigg\{\sum_{\gamma \in E}  \phi(\gamma) \ \gamma: \ \phi \in l^s(E) \bigg\}, \ \ \ \ \ s \in [1,\infty].
\end{equation}
The $\Gamma$-\emph{transform of} $\mathfrak{s} =  \sum_{\gamma \in \Gamma}  \phi(\gamma)\gamma \  \in \mathfrak{S}_{\Gamma}$  is denoted by 
\begin{equation}
\widehat{\mathfrak{s}}(\gamma) \ = \ \phi(\gamma), \ \ \ \ \gamma \in \Gamma.
\end{equation}
The \emph{spectrum} of $\mathfrak{s}$ is the support of $\widehat{\mathfrak{s}}$.  If $E \subset \Gamma$, and the spectrum of $\mathfrak{s} \in \mathfrak{S}_{E}$ is finite, then $\mathfrak{s}$ is an $E$-\emph{polynomial}.

With each $\mu \in M(\widehat{\Gamma})$ (a regular Borel measure on $\widehat{\Gamma}$) we associate the $\Gamma$-series
\begin{equation} \label{iden1}
\begin{split}
 \mu \ &= \ \sum_{\gamma \in \Gamma} \widehat{\mu}(\gamma) \gamma,\\\
 \widehat{\mu}(\gamma) \ = \ \int_{\widehat{\Gamma}} \overline{\gamma} \ d\mu, \ \ \ \ & \gamma \in \Gamma \ \ \ (\text{\emph{Fourier-Stieltjes transform})},
\end{split}
\end{equation}  
which uniquely determines $\mu$. We then formally identify, via \eqref{iden1}, each space in \eqref{iden2} as a subspace of $\mathfrak{S}_{\Gamma}.$ 

  In this monograph we work mostly in the 'dyadic' setting \  $$\Gamma \  = \  W_A,   \ \ \  \widehat{\Gamma} \  = \  \Omega_A, \ \  \ \mathfrak{m} \ = \ \mathbb{P}_A,$$ a basic and convenient setting for our purposes here.

 We will be dealing in this chapter, and also in later chapters, with mixed-norm spaces.  Specifically, for $s \in [1,\infty]$, $t \in [1,\infty]$, sets $A$ and $B$, and $\phi \in \mathbb{C}^{A \times B}$, consider the mixed norms
\begin{equation}
\begin{split}
\|\phi\|_{s,t} \ &\overset{def}{=} \ \bigg(\sum_{\alpha \in A} \big(\sum_{\beta \in B} |\phi(\alpha,\beta)|^t\big)^{s/t}\bigg)^{1/s}, \ \ \ \ \ s \in (1,\infty), \ \ t \in (1, \infty) \\\\
\|\phi\|_{\infty,t} \ &\overset{def}{=} \ \sup_{\alpha \in A} \big(\sum_{\beta \in B} |\phi(\alpha,\beta)|^t\big)^{1/t}, \ \ \ \ \ s = \infty, \ \ t \in (1, \infty), \\\\
\|\phi\|_{s,\infty} \ &\overset{def}{=} \  \bigg(\sum_{\alpha \in A} \  \sup_{\beta \in B} |\phi(\alpha,\beta)|^s\bigg)^{1/s}, \ \ \ \ \ s \in (1,\infty), \ \ t = \infty,
\end{split}
\end{equation}\\
and let  
\begin{equation}
l^s\big(A;l^t(B)\big) \overset{def}{=} \ l^{s,t}(A,B) \ = \ \big\{\phi \in \mathbb{C}^{A \times B}: \ \|\phi\|_{s,t} < \infty \big\}.
\end{equation}\\  
  \subsection{$(L^1_R \hookrightarrow L^2)$-inequality,  \ $(C_{R \times R} \hookrightarrow l^{1,2})$-inequality, \  $(C_{R \times R} \hookrightarrow l^{2,1})$-inequality} \ \\


\noindent
{\bf{i.}} \ The inclusions
\begin{equation} \label{inclu}
L^1_{R_A} \ \subset \ L^2(\Omega_A,\mathbb{P}_A) \ \ \ \text{for all $A$},
\end{equation}
or equivalently,  
\begin{equation} \label{Khintchin}
  \kappa \  \overset{\text{def}}{=}   \  \sup \big\{1/\|U_{R_A}{\bf{x}} \|_{L^1}:   {\bf{x}} \in \mathcal{S}_{l^2(A)}, \ \text{finite sets} \ A \big\} \ < \ \infty,
\end{equation}
had been independently observed in the 1930's by Banach, Littlewood, Orlicz, Steinhaus, and Zygmund, with various estimates of the constant $\kappa$, whose value became at the time an open problem. 
We refer to $\kappa$ as the \emph{Khintchin} $(L^1_R \hookrightarrow L^2)$-\emph{constant}, and to '$\kappa < \infty$'  as the \emph{Khintchin} $(L^1_R \hookrightarrow L^2)$-\emph{inequality}.  That
 \begin{equation} \label{Khintcon}  
   \kappa \ = \ \sqrt{2}
 \end{equation} 
was proved in a master's thesis \citep{szarek1976best}. (See also \citet{Haagerup1981}.)

 So far that I can determine, Littlewood  \citep{Littlewood:1930} was first to state and prove
\begin{equation}\label{lit1}
\bigg(\sum_{\alpha \in A} |{\bf{x}}(\alpha)|^2 \bigg)^{1/2} \ \leq \ \sqrt{3} \ \big\|\sum_{\alpha \in A} {\bf{x}}(\alpha) r_{\alpha} \big\|_{L^1}, \ \ \ \ \text{finite sets} \  A, \ \  {\bf{x}} \in \mathbb{C}^A,
\end{equation}
i.e., that $\kappa \leq \sqrt{3}$. To this end, Littlewood 
verified the $(L^2_R \hookrightarrow L^4)$-inequality ( $p = 4$ in \eqref{expsq1} ),
and then deduced  \eqref{lit1} from it by convexity, essentially via    
\begin{lemma}[cf. {\citet[p. 128]{rudin2011fourier}}] \label{trick1}  For  $E \subset \Gamma$, if for some $p > 2$,
\begin{equation} \label{trick2}
\kappa_2(E;p) \  \overset{\text{def}}{=}  \ \sup \big\{\|f\|_{L^p}: f \in \mathcal{S}_{L^2_E} \big \} \   < \ \infty,\\\
\end{equation}
then \\
\begin{equation} 
 \kappa_1(E;2) \ \overset{\text{def}}{=}  \ \sup \big\{1/\|U_{F}{\bf{x}} \|_{L^1(\widehat{\Gamma},\mathfrak{m}}): \ \text{finite sets} \ F \subset E, \ {\bf{x}} \in \mathcal{S}_{l^2(F)} \big\}  \  \leq \ [\kappa_2(E;p)]^{\frac{p}{p-2}}. \ \ \ \ \ \ \ 
\end{equation}\\
I.e., for $p > 2$,
\begin{equation}
L^2_E  \subset L^p(\widehat{\Gamma},\mathfrak{m}) \ \Rightarrow \ L^1_E  \subset L^2(\widehat{\Gamma},\mathfrak{m}).
\end{equation}
 \end{lemma}
 \begin{proof}\ 
 For ${\bf{x}} \in \mathcal{S}_{l^2(F)}$,  finite $F \subset E$, write $f = U_F{\bf{x}}$, and estimate (via H\"older)
 \begin{equation} \label{modi}
1 \ = \ \int_{\widehat{\Gamma}}|f|^{\frac{p-2}{p-1}} \ |f|^{\frac{p}{p-1}} \ d\mathfrak{m} \ \leq \ \big(\|f\|_{L^1}\big)^{\frac{p-2}{p-1}} \ \big(\|f\|_{L^p}\big)^{\frac{p}{p-1}} \ \leq \  \big(\|f\|_{L^1}\big)^{\frac{p-2}{p-1}} \ [\kappa_2(E;p)]^{\frac{p}{p-1}}. \ \ \ \ \ \ 
 \end{equation}\\
 \end{proof}
 \begin{remark} [$\Lambda(p)$-sets \citep{rudin1960trigonometric}] \label{Lambda} \ For $E \subset \Gamma$, $p > q > 0$, \  let 
 \begin{equation} \label{Lambdapq}
\sup \big\{\|f\|_{L^p}/\|f\|_{L^q}: \text{$\Gamma$-polynomials} \ f,  \ f \neq 0, \ \widehat{f} = 0 \ \text{on} \ \Gamma \setminus E \big\}  \ \overset{{\text{def}}}{=}  \ \kappa_q(E;p). \ \ \ \ \ \ \ \ \ \ 
\end{equation}\\
 As in \eqref{modi}, an application of H\"older implies
\begin{equation}
\kappa_q(E;p) < \infty \ \Rightarrow \ \kappa_s(E;p) < \infty, \ \ \ \ p > q > s>0.
\end{equation}
We refer to '$\kappa_q(E;p) < \infty$' as an $(L^q_E \hookrightarrow L^p)$-inequality.
\begin{definition} [{\citet[Theorem 1.4, Definition 1.5]{rudin1960trigonometric}}] \label{Lambda1}  $E \subset \Gamma$ is  $\Lambda(p)$ if  \ $\kappa_q(E;p) < \infty$ for some $q \in (0,p)$.  
\end{definition}
\noindent
E.g., \eqref{expsq1} implies that Rademacher systems are $\Lambda(p)$ for all $p \in (0,\infty).$ 
 \end{remark}
 
\noindent
{\bf{ii.}} \ Littlewood used $\kappa < \infty$  in a derivation of a mixed-norm inequality:\ \\

 for all finite sets $A$ and $B$, and scalar-valued arrays $(a_{\alpha \beta})_{\alpha \in A, \beta \in B}$,
\begin{equation} \label{littlewood1}
\begin{split}
\big\|\sum_{\alpha \in A, \  \beta \in B} a_{\alpha \beta} \ r_{\alpha} \otimes r_{\beta}\big \|_{\infty} \ &\overset{\text{def}}{=} \ \sup_{\omega \in \Omega_A, \ \omega^{\prime} \in \Omega_B} \big|\sum_{\alpha \in A, \  \beta \in B} a_{\alpha \beta} \ r_{\alpha}(\omega)  r_{\beta}(\omega^{\prime}\big| \\\\
& \geq \ \frac{2}{\pi} \ \sup_{\omega \in \Omega_A} \sum_{\beta \in B} \big|\sum_{\alpha \in A} a_{\alpha \beta} \ r_{\alpha}(\omega) \big|\\\\
 &\geq \ \frac{2}{\pi} \ \sum_{\beta \in B} \int_{\omega \in \Omega_A}\big|\sum_{\alpha \in A} a_{\alpha \beta} \ r_{\alpha}(\omega)\big| \mathbb{P}_A(d\omega)\\\\
 & \geq \ \frac{2}{\pi \kappa} \ \sum_{\beta \in B} \bigg (\sum_{\alpha \in A} |a_{\alpha \beta}|^2 \bigg)^{\frac{1}{2}},
\end{split}
\end{equation}
where the second line follows from \eqref{Sid1}, and the fourth from \eqref{Khintchin}. 
Let  $\kappa_L$  be the infimum of $K > 0$ such that for all finite sets $A$ and $B$, and scalar-valued arrays $(a_{\alpha \beta})_{\alpha \in A, \beta \in B}$, 
\begin{equation} \label{littlewood}
\sum_{\beta \in B} \bigg (\sum_{\alpha \in A} |a_{\alpha \beta}|^2 \bigg)^{\frac{1}{2}} \ \leq \ K \ \big\|\sum_{\alpha \in A, \  \beta \in B} a_{\alpha \beta} \ r_{\alpha} \otimes r_{\beta}\big \|_{\infty},
\end{equation} 
whence from \eqref{littlewood1},
\begin{equation} \label{khilit}
\kappa_L  \ \leq \ \frac{\pi \kappa}{2}.
\end{equation}\ 

Littlewood's mixed-norm inequality 
\begin{equation} \label{litconst}
\kappa_L  \ < \ \infty
\end{equation}
is equivalent to the assertion 
\begin{equation} \label{lit2}
f \in C_{R_A \times R_B} \ \Rightarrow \  \widehat{f} \in l^{1,2}(A,B), \ \ \ \ \text{for all sets $A$ and $B$},
\end{equation}\\
and we refer to it as  $(C_{R \times R} \hookrightarrow l^{1,2})$-inequality.\\


\begin{remark} [a special case of the Grothendieck inequality] \label{special} \ The inequality in \eqref{littlewood}, restated as\\   
\begin{equation*}
\sup \big \{\big|\sum_{\alpha \in A, \  \beta \in B} a_{\alpha \beta} \  {\bf{e}}_{\alpha} \boldsymbol{\cdot} {\bf{y}}_{\beta}  \big| :  \ {\bf{y}}_{\beta} \in \mathcal{S}_{l^2(A)}, \ \beta \in B \big \} \ \leq \ K \ \big\|\sum_{\alpha \in A, \  \beta \in B} a_{\alpha \beta} \ r_{\alpha} \otimes r_{\beta}\big \|_{\infty},
\end{equation*}\\
is an instance of \eqref{grothen}, whereby \footnote{\label{footnote5} Littlewood's result, which had appeared twenty or so years prior to Grothendieck's,  was indeed cited in the Resum\'{e}  \citep{Grothendieck:1956}.  Whether Littlewood's mixed-norm inequality at least partly inspired Grothendieck's \emph{fundamental theorem}, is open to speculation.}  
\begin{equation} \label{littlegr}
\kappa_L \ \leq \ \mathcal{K}_G.
\end{equation} 

\end{remark} \  

\noindent
{\bf{iii.}} \ Let  $\kappa_O$ denote the infimum of  $K > 0$ over all finite sets $A$ and $B$, and  scalar-valued arrays $(a_{\alpha \beta})_{\alpha \in A, \beta \in B}$, such that
 \begin{equation} \label{orlicz1}
\bigg( \sum_{\alpha \in A} \big (\sum_{\beta \in A} |a_{\alpha \beta}|\big )^2 \bigg)^{\frac{1}{2}} \ \leq \ K \ \big\|\sum_{\alpha \in A, \  \beta \in A} a_{\alpha \beta} \ r_{\alpha} \otimes r_{\beta}\big \|_{\infty}.
\end{equation}
From \eqref{littlewood}, via Minkowski (the 'triangle inequality'),
\begin{equation} \label{orlicz}
\kappa_O \ \leq \ \kappa_L,
\end{equation}
and therefore, from \eqref{litconst},
\begin{equation} \label{orlicz2}
\kappa_O < \infty,
\end{equation}
which was deduced in \citet{orlicz1933unbedingte} independently of  \citet{Littlewood:1930}. We refer to \eqref{orlicz2} as the Orlicz $(C_{R \times R} \hookrightarrow l^{2,1})$-inequality; cf. \eqref{lit2}. 

\ \
 
\noindent
{\bf{iv.}}  \ The three inequalities are equivalent in the sense that 
\begin{equation} \label{equivalence3}
\kappa < \infty  \ \Leftrightarrow \ \kappa_L < \infty \ \Leftrightarrow \ \kappa_O < \infty,
\end{equation}
where $\kappa$ is defined in \eqref{Khintchin}, $\kappa_L$  via \eqref{littlewood}, and $\kappa_O$ via \eqref{orlicz1}.  

Following \eqref{khilit} and \eqref{orlicz} above, to prove \eqref{equivalence3} it suffices  to verify  \ $\kappa  \leq  \kappa_O$.   To this end, we assume $\kappa_O < \infty,$  and show that for a finite set $A,$  if ${\bf{x}} \in \mathbb{C}^A$ and 
\begin{equation} \label{normalize}
 \big\|\sum_{\alpha \in A}{\bf{x}}(\alpha) r_{\alpha}\big\|_{L^1(\Omega_A,\mathbb{P}_A)} \ = \ 1,
 \end{equation}   
then $\|{\bf{x}}\|_2 \leq \kappa_O$.  We fix $\omega^{\prime} \in \Omega_A$, and consider the \emph{Riesz product} 
\begin{equation}
\mathfrak{R}_A(\omega^{\prime}) \ = \ \prod_{\alpha \in A} \big(1 +  r_{\alpha}(\omega^{\prime}) r_{\alpha}\big).
\end{equation}
(E.g., see \S \ref{RieszP}.)  Because $\mathfrak{R}_A(\omega^{\prime}) \geq 0$, 
\begin{equation}
\|\mathfrak{R}_A(\omega^{\prime})\|_{L^1} \ = \ \int_{\Omega_A} \mathfrak{R}_A(\omega^{\prime}) d\mathbb{P}_A \ = \ 1, 
\end{equation}
which implies
\begin{equation} \label{above}
\begin{split}
\big\|\sum_{\alpha \in A}{\bf{x}}(\alpha)r_{\alpha}(\omega^{\prime}) r_{\alpha}\big\|_{L^1} \ &= \ \|\mathfrak{R}_A(\omega^{\prime}) \convolution \big(\sum_{\alpha \in A}{\bf{x}}(\alpha) r_{\alpha}\big)\big\|_{L^1}   \\\\
& \leq \ \big \|\mathfrak{R}_A(\omega^{\prime})\big\|_{L^1} \ \big\|\sum_{\alpha \in A}{\bf{x}}(\alpha) r_{\alpha}\big\|_{L^1} \ \leq \ 1,\\\
\end{split}
\end{equation}
where  \ $\convolution$ \ denotes convolution. Then, by the duality  $\big(L^1(\Omega_A,\mathbb{P}_A)\big)^{\star} = L^{\infty}(\Omega_A,\mathbb{P}_A)$, \ and because $\Omega_A$ is a finite set,
\begin{equation}
\begin{split}
\big|\sum_{\omega \in \Omega_A}\bigg(\sum_{\alpha \in A}\frac{1}{2^{\#A}}{\bf{x}}(\alpha)r_{\alpha}(\omega)r_{\alpha}(\omega^{\prime})\bigg) \tilde{r}_{\omega}(u)\big| \ &\leq \ \big\|\sum_{\alpha \in A}{\bf{x}}(\alpha)r_{\alpha}(\omega^{\prime}) r_{\alpha}\big\|_{L^1} \ \leq \ 1,\\\
& \omega^{\prime} \in \Omega_A, \ u \in \Omega_{\Omega_A},
\end{split}
\end{equation}
where $$\{\tilde{r}_{\omega}: \omega \in \Omega_A\}  \ \overset{\text{def}}{=}  \ R_{\Omega_A} \ \ \ \ \text{(Rademacher system indexed by $\Omega_A$)}.$$\\
Applying the assumption $\kappa_O < \infty$ to \ $(a_{\alpha \omega})_{\alpha \in A, \omega \in \Omega_A} \ \in \ \mathbb{C}^{A \times \Omega_A}$, where
\begin{equation}
a_{\alpha \omega} \ = \ \frac{1}{2^{\#A}}{\bf{x}}(\alpha)r_{\alpha}(\omega), \ \ \ \ \alpha \in A, \ \omega \in \Omega_A,
\end{equation}
we obtain
\begin{equation}
\bigg(\sum_{\alpha \in A} \big(\sum_{\omega \in \Omega_A} \big|\frac{1}{2^{\#A}}{\bf{x}}(\alpha)r_{\alpha}(\omega)\big|\big)^2 \bigg)^{1/2} 
\ = \ \|{\bf{x}}\|_2 \ \leq \ \kappa_O, 
\end{equation}\\
which implies  (via \eqref{normalize} and the definition of $\kappa$) 
\begin{equation}
\kappa \ \leq \ \kappa_O.
\end{equation} 
We summarize:
\begin{proposition} [cf. {\citet[Ch. II \S 4]{Blei:2001}}] \label{derive}
\begin{equation} \label{samecon}
\kappa \ \leq \kappa_O \ \leq \ \kappa_L \ \leq \ \frac{\pi \kappa}{2}.
\end{equation}
\end{proposition}
\noindent
Note:  if the underlying scalar field throughout is $\mathbb{R}$, then $\pi/2$ can be removed from \eqref{littlewood1}, i.e., in the 'real' case,
\begin{equation} \label{samecon1}
 \kappa_O  = \kappa_L  =  \kappa.
\end{equation}
\begin{remark}[generalizations] \ Maximizing  \eqref{grothen} over  ${\bf{y}}_{v} \in \mathcal{S}_{l^2(A)} \ (v \in B)$ implies\\
\begin{equation} \label{genLit}
 \ \sum_{v \in B}\bigg(\sum_{\alpha \in A} \big|\sum_{u \in B} a_{uv} \  {\bf{x}}_{u} (\alpha) \big |^2\bigg)^{1/2} \  \leq \ K\big\|\sum_{u \in B, \  v \in B} a_{uv} \ r_{u} \otimes r_{v}\big \|_{\infty}, \ \ \ \ \ {\bf{x}}_{u} \in \mathcal{S}_{l^2(A)}  \ \ (u \in B),
\end{equation}\\
and then by Minkowski,\\
\begin{equation} \label{genOrl}
\bigg(\sum_{\alpha \in A} \big(\sum_{v \in B}\big|\sum_{u \in B}  a_{uv} \  {\bf{x}}_{u} (\alpha) \big|\big)^2 \bigg)^{1/2} \  \leq \ K\big\|\sum_{u\in B, \  v \in B} a_{u v} \ r_{u} \otimes r_{v}\big \|_{\infty},  \ \ \ \ \ \ {\bf{x}}_{u} \in \mathcal{S}_{l^2(A)} \  \ (u \in B).
\end{equation}\\\
Inequalities \eqref{genLit} and \eqref{genOrl} were dubbed  \emph{the generalized} Littlewood inequality and \emph{the generalized} Orlicz inequality, respectively, in \citet[pp. 280-281]{Lindenstrauss:1968}.  The generalized Littlewood inequality is a restatement of  \eqref{grothen},
whereas the generalized Orlicz inequality is the instance ${\bf{y}}_{v} = {\bf{y}}$ in \eqref{grothen}.
\end{remark}
\subsection{Interpolants} \label{DLit} \ \\

\noindent
{\bf{i  \ (Khintchin).}} \  The Khintchin inequality $\kappa < \infty$ implies (via Parseval) that for any set $A$ and ${\bf{x}} \in l^2(A)$, 
\begin{equation}
h \  \mapsto \ \sum_{\alpha \in A} \widehat{h}(r_{\alpha})  {\bf{x}}(\alpha), \ \ \ \ \ \text{Walsh polynomials}  \ \ h \in L^1_{R_A},
\end{equation}
determines a continuous linear functional on $L^1_{R_A}$ with norm bounded by $\kappa$. 
Then, by Hahn-Banach and $L^1(\Omega_A,\mathbb{P}_A)^{\star}  =  L^{\infty}(\Omega_A,\mathbb{P}_A)$, there exist  $G_A({\bf{x}})  \in L^{\infty}(\Omega_A,\mathbb{P}_A)$ such that
\begin{equation}
\widehat{G_A({\bf{x}})}(r_{\alpha}) \ = \ {\bf{x}}(\alpha),\ \ \ \alpha \in A,
\end{equation}
and
\begin{equation} \label{con}
\|G_A({\bf{x}})\|_{L^{\infty}} \ \leq \ \kappa \|{\bf{x}}\|_2.
\end{equation}\\
That is, $l^2(A)$ and $L^{\infty}(\Omega_A,\mathbb{P}_A)/L^{\infty}_{W_A \setminus R_A}$ are isopmorphic as Banach spaces via the linear injection
\begin{equation} \label{equiv3}
\mathcal{G}_A: \ l^2(A) \ \overset{onto}{\longrightarrow} \  L^{\infty}(\Omega_A,\mathbb{P}_A)/L^{\infty}_{W_A \setminus R_A},\\\\
\end{equation}
where
\begin{equation}\label{equiv2}
\mathcal{G}_A({\bf{x}}) \ = \ \big\{f \in L^{\infty}(\Omega_A,\mathbb{P}_A): \  \widehat{f}(r_{\alpha}) = {\bf{x}}(\alpha), \ \alpha \in A \big \}, \ \ \ \ \  \ {\bf{x}} \in l^2(A), 
\end{equation}
\begin{equation} \label{equiv4}
\|{\bf{x}}\|_2 \ \leq \ \vertiii{\mathcal{G}_A({\bf{x}})} \ \leq  \ \kappa \|{\bf{x}}\|_2, 
\end{equation}\\
and  \ $\vertiii{\boldsymbol{\cdot}}$ \ is the quotient norm,\\ 
\begin{equation} \label{equiv5}
\vertiii{\mathcal{G}_A({\bf{x}})} \ \overset{\text{def}}{=} \ \inf \big \{\|f\|_{L^{\infty}}: f \in \mathcal{G}_A({\bf{x}}) \big \}.
\end{equation}\\
Invoking the \emph{axiom of choice}  (with a small apology for the pedantry), we select for every ${\bf{x}} \in  \mathcal{S}_{l^2(A)}$, a class representative  $G_A({\bf{x}}) \in \mathcal{G}_A({\bf{x}})$  such that
\begin{equation} \label{Khint111}
\|G_A({\bf{x}})\|_{L^{\infty}} \ \leq \ \kappa.
\end{equation}
For ${\bf{x}} \in l^2(A) \setminus \mathcal{S}_{l^2(A)}$, define
\begin{equation} \label{extending2}
\begin{split}
G_A({\bf{x}}) \ &= \ \|{\bf{x}}\|_2 \ G_A(\boldsymbol{\sigma}{\bf{x}}), \ \ \ \ \  {\bf{x}} \neq {\bf{0}},\\\
G_A({\bf{0}}) \ &= \ {\bf{0}},
\end{split}
\end{equation}
where $\boldsymbol{\sigma}: \ {l^2(A)} \rightarrow \mathcal{S}_{l^2(A)} \cup \{ {\bf{0}}\}$  \ is given in \eqref{sig}.  The result is a \emph{choice function}\footnote{\label{footnote1} The \emph{axiom of choice} is invoked here not as a tool, but to signal that we use the 'simplest' and 'least efficient' algorithm to construct $G_A$.  Namely, for each ${\bf{x}} \in \mathcal{S}_{l^2(A)}$, we choose 'randomly' from $\mathcal{G}_A({\bf{x}})$, constrained only by \eqref{Khint111}; the resulting \emph{choice function} is a $(l^2(A) \hookrightarrow L^{\infty})$-interpolant uniformly bounded in $A$. Arguably, we could (should?) invoke here the \emph{probabilistic method} \citep{alon2016probabilistic}, rather than the \emph{axiom of choice}, or invoke both...  Subsequent algorithms -- in effect 'upgrades' of the Khintchin inequality -- will be designed to be 'more efficient,' as well as more versatile.} 
\begin{equation} \label{Khint10}
G_A  \ = \ U_{R_A} + \mathfrak{p}_A \ : \ {l^2(A)} \ \rightarrow \ L^{\infty}(\Omega_A,\mathbb{P}_A),
\end{equation}
where
\begin{equation} \label{Khint11}
\sup_{{\bf{x}} \in \mathcal{S}_{l^2(A)}}\|G_A({\bf{x}})\|_{L^{\infty}}  \ = \ \|G_A\|_{l^2 \hookrightarrow L^{\infty}} \ \leq \ \kappa,
\end{equation}
and
\begin{equation} \label{Khint13}
\mathfrak{p}_A:  \  {l^2(A)}  \ \rightarrow  \ \mathfrak{S}^2_{W_A \setminus R_A}, \ \ \ \ \text{as per \eqref{series2}},
\end{equation}
whence 
\begin{equation}  \label{e1}
\begin{split}
\|\mathfrak{p}_A({\bf{x}})\|_{L^2} \  \leq \   \sqrt{\kappa^2 \|{\bf{x}}\|_2^2  -  \|U_{R_A}{\bf{x}}\|_{L^2}^2} \  = \  \sqrt{\kappa^2 - 1}  \|{\bf{x}}\|_2 \ & = \   \|{\bf{x}}\|_2,  \ \ \ \ \ {\bf{x}} \in l^2(A)\\\
& (\kappa = \sqrt{2}).
\end{split}
\end{equation}
\begin{definition} \label{expand} \ For sets $A^{\prime} \supset A$, an injection
\begin{equation}
G_A: \ l^2(A) \ \hookrightarrow \ L^{\infty}(\Omega_{A^{\prime}},\mathbb{P}_{A^{\prime}})
\end{equation}
is a \emph{$(l^2 \hookrightarrow L^{\infty})$-interpolant on $R_A$} if
\begin{equation} \label{extending3}
\begin{split}
G_A({\bf{0}}) \ &= \ {\bf{0}},\\\
G_A({\bf{x}}) \ &= \ \|{\bf{x}}\|_2 \ G_A(\boldsymbol{\sigma}{\bf{x}}), \ \ \ \ \  {\bf{x}} \neq {\bf{0}},\\\
\text{and} \ \ \ \ \ \ \ \ \ \ \ \ \ \ \ \ \  \ \ \ \ \ \ \ \ & \\\
\widehat{G_A({\bf{x}})}(r_{\alpha}) \ &= \ {\bf{x}}(\alpha), \ \ \ \ {\bf{x}}\in l^2(A), \ \ \alpha \in A.
 \end{split}
\end{equation}
Denote
\begin{equation}
\mathcal{I}_{l^2 \hookrightarrow L^{\infty}}(R_A) \ = \ \big\{\text{$(l^2 \hookrightarrow L^{\infty})$-interpolants on $R_A$} \big\},
\end{equation}
and let
\begin{equation} \label{Khint01}
\kappa^{\star} \ \overset{\text{def}}{=} \ \sup_A \ \inf \big \{ \|G_A\|_{l^2 \hookrightarrow L^{\infty}}: \  G_A \in \mathcal{I}_{l^2 \hookrightarrow L^{\infty}}(R_A) \big \}.
\end{equation}
\end{definition}
\begin{proposition} \label{KhintD}
\begin{equation} \label{Khint12}
\kappa^{\star} \ =  \ \kappa \ ( \ = \sqrt{2} \ ).
\end{equation} 
\end{proposition}
\begin{proof} \ Assuming $\kappa < \infty$, we produce the $(l^2 \hookrightarrow L^{\infty})$-interpolant $G_A$ \ in \eqref{Khint10} (a choice function) that satisfies \eqref{Khint11}, and thus $\kappa \geq \kappa^{\star}.$

  To verify the reverse inequality, assume $\kappa^{\star} < \infty$ and $A$ arbitrary.   For ${\bf{x}} \in l^2(A)$,  let ${\bf{v}} \in \mathcal{S}_{l^2(A)}$ be such that
\begin{equation}
\sum_{\alpha \in A} {\bf{x}}(\alpha)\overline{{\bf{v}}(\alpha)} \ = \ \|{\bf{x}}\|_2.
\end{equation}
Then (by Parseval and H\"older), for any $G_A \in \mathcal{I}_{l^2 \hookrightarrow L^{\infty}}(R_A)$, 
\begin{equation} \label{reverse}
\begin{split}
\|{\bf{x}}\|_2 \ = \ \big|\int_{\Omega_{A^{\prime}}} U_{R_A}{\bf{x}}\ \overline{G_A({\bf{v}})} \ d\mathbb{P}_{A^{\prime}}\big| \ \leq \ \|U_{R_A}{\bf{x}}\|_{L^1} \ \|G_A({\bf{v}})\|_{L^{\infty}} \ \leq \  \|U_{R_A}{\bf{x}}\|_{L^1} \ \|G_A\|_{l^2 \hookrightarrow L^{\infty}},\\
\end{split}
\end{equation}\\
which implies  \ $\kappa^{\star}  \geq  \kappa$.
\end{proof}
\begin{remark} [about  $A^{\prime} \supset A$] \ \label{sup} 
To define  \ $\kappa^{*}$, and then verify Proposition \ref{KhintD}, it suffices to take  infimum in \eqref{Khint01} over interpolants with  $A^{\prime} = A$. That is,
\begin{equation} \label{Khint03}
\kappa^{\star} \ {=} \ \sup_A \ \inf \big \{ \|G_A\|_{l^2 \hookrightarrow L^{\infty}}: \  \text{$\big(l^2(A) \hookrightarrow L^{\infty}(\Omega_A,\mathbb{P}_A)\big)$-interpolants $G_A$ on $R_A$} \big \}.
\end{equation}

 Taking $L^{\infty}(\Omega_{A^{\prime}},\mathbb{P}_{A^{\prime}})$-valued interpolants with $A^{\prime} \supset A$, as per Definition \ref{expand}, allows for \emph{upgrades} -- a notion explained and illustrated in the next section.  To wit, whereas the Khintchin inequality is an instance of the Grothendieck inequality (via Remark \ref{special} and Proposition \ref{derive}), the Grothendieck inequality can be viewed precisely as an upgrade of the Khintchin inequality, via $\big(l^2(A) \hookrightarrow L^{\infty}(\Omega_{A^{\prime}},\mathbb{P}_{A^{\prime}})\big)$-interpolants with  $A^{\prime} \supsetneq A$ (Remark \ref{upgrade2}).
\end{remark}

\begin{remark} [template for later use] \label{temp} \  Given a $\big(l^2(A) \hookrightarrow L^{\infty}(\Omega_{A^{\prime}},\mathbb{P}_{A^{\prime}})\big)$-interpolant  $G_A$ on $R_A$, we write $G_A = U_{R_A} + \mathfrak{p}_{A^{\prime}}$, where 
\begin{equation}
\mathfrak{p}_{A^{\prime}}: \ l^2(A) \ \rightarrow \ \mathfrak{S}_{W_{A^{\prime}} \setminus R_A}^2, \ \ \ \ A^{\prime} \supset A, 
\end{equation}
and say that $\mathfrak{p}_{A^{\prime}}$ is an \emph{orthogonal $(l^2,L^{\infty})$-perturbation of $U_{R_A}$}.  By \eqref{extending3},
\begin{equation} \label{homog11}
\mathfrak{p}_{A^{\prime}}({\bf{x}}) \ = \ \|{\bf{x}}\|_2 \ \mathfrak{p}_{A^{\prime}}(\boldsymbol{\sigma}{\bf{x}}), \ \ \ \ \ {\bf{x}} \in l^2(A),
\end{equation}
and by the Khintchin $(L^2_R \hookrightarrow L^q)$-inequalities (Remark \ref{expsq}), 
\begin{equation}
\|\mathfrak{p}_{A^{\prime}}\|_{l^2 \rightarrow L^q}  \ = \ \mathcal{O}(\sqrt{q}), \ \ \ \ q \geq 2.
\end{equation}
\

We will consider in due course other types of interpolants and perturbations.   To avoid repetition -- and also to highlight a recurring theme -- we formalize the following template that we use throughout.  Let $E \subset $  discrete Abelian group  \ $\Gamma$, \  $\mathfrak{X} = $  normed linear subspace of \  $\mathbb{C}^E$,   \ $\mathfrak{B} = $ normed linear subspace of  \ $\mathfrak{S}_{\Gamma}$.  An injection
\begin{equation}
G_E: \ \mathfrak{X} \ \hookrightarrow \ \mathfrak{B}
\end{equation}
is a $(\mathfrak{X} \hookrightarrow \mathfrak{B})$-interpolant on $E$ if
\begin{equation} \label{temp0}
G_E({\bf{x}}) \ = \ \|{\bf{x}}\|_{\mathfrak{X}} \ G_E(\boldsymbol{\sigma}_{\mathfrak{X}}{\bf{x}}), \ \ \ \ \ \ {\bf{x}} \in \mathfrak{X},
\end{equation}
and
\begin{equation}
\widehat{G_E({\bf{x}})}(\gamma) \ = \ {\bf{x}}(\gamma), \ \ \ \ \gamma \in E, \ \ {\bf{x}} \in \mathfrak{X}.
\end{equation}\\
We denote
\begin{equation}
\mathcal{I}_{\mathfrak{X} \hookrightarrow \mathfrak{B}}(E) \ \overset{def}{=} \ \big\{\text{$(\mathfrak{X} \hookrightarrow \mathfrak{B})$-interpolants on $E$} \big\}.
\end{equation}\\
An orthogonal $(\mathfrak{X},\mathfrak{B})$-perturbation of  $U_E$  is a map  $\mathfrak{p}:  \mathfrak{X} \rightarrow   \mathfrak{S}_{\Gamma \setminus E}$ 
such that  \  $U_E + \mathfrak{p}$ \  is a \  $(\mathfrak{X} \hookrightarrow \mathfrak{B})$-interpolant on $E$.  E.g., in Definition \ref{expand}, \  $\Gamma = W_{A^{\prime}}$, \  $E = R_A$, \ $\mathfrak{X} = l^2(A)$, \  $\mathfrak{B} = L^{\infty}(\Omega_{A^{\prime}},\mathbb{P}_{A^{\prime}}).$ \footnote{\label{footnote3} $L^{\infty}(\Omega_{A^{\prime}},\mathbb{P}_{A^{\prime}})$ \  is identified with  \  $\{\text{$W_A$-series representing elements in $L^{\infty}(\Omega_{A^{\prime}},\mathbb{P}_{A^{\prime}}$)}$\}.}
\end{remark}

\

 \begin{remark} [also for later use] \label{general} \  
   If $E \subset W_A$ is  $\Lambda(2)$, i.e., if 
 \begin{equation}
 \kappa(E) \  \overset{\text{{def}}}{=}  \ \kappa_1(E,2) \  < \  \infty
 \end{equation}
 (Remark \ref{Lambda}),  then there is a linear injection 
 \begin{equation} 
 \mathcal{G}_E: \ l^2(E) \ \overset{onto}{\longrightarrow} \ L^{\infty}(\Omega_A,\mathbb{P}_A)/L^{\infty}_{W_A \setminus E},
 \end{equation}
where
\begin{equation} \label{equiv6}
\mathcal{G}_E({\bf{x}}) \ = \ \big \{g \in L^{\infty}(\Omega_A,\mathbb{P}_A): g - U_E{\bf{x}} \in L^2_{W_A \setminus E} \big \} \ \neq \ \emptyset,
\end{equation}
and
\begin{equation}
\|{\bf{x}}\|_2 \ \leq \ \vertiii{\mathcal{G}_E({\bf{x}})} \ \overset{\text{def}}{=} \ \inf  \big \{\|g\|_{L^{\infty}}: g \in \mathcal{G}_E({\bf{x}} \big \} \ \leq \ \kappa(E)   \|{\bf{x}}\|_2, \ \ \ \ \ {\bf{x}} \in l^2(E).
\end{equation}\\
(Cf. \eqref{equiv3} - \eqref{equiv5}.) Therefore, there exists a  $(l^2 \hookrightarrow L^{\infty})$-interpolant (a choice function) 
 \begin{equation}  \label{general1}
 G_E =  U_E +  \mathfrak{p} :  \ l^2(E)  \ \hookrightarrow \ L^{\infty}(\Omega_A,\mathbb{P}_A),
 \end{equation}
 where $\mathfrak{p}: \ l^2(E) \rightarrow l^2(W_A \setminus E)$ is an orthogonal $(l^2,L^{\infty})$-perturbation of $U_E$,
\begin{equation}
\|G_E\|_{l^2 \hookrightarrow L^{\infty}} \ \leq \ \kappa(E), \ \ \ \ \ \|\mathfrak{p}\|_{l^2 \rightarrow l^2} \ \leq \ \sqrt{[\kappa(E)]^2 - 1},
\end{equation} 
and
\begin{equation}
 \kappa^{\star}(E) \  \overset{{\text{def}}}{=} \  \inf \big \{\|G_E\|_{l^2 \hookrightarrow L^{\infty}}:  (l^2(E) \hookrightarrow L^{\infty}) \text{-interpolants}  \ G_E \big\} \ = \  \kappa(E).
 \end{equation}\\
(Cf. \eqref{Khint10} - \eqref{e1}, Proposition \ref{KhintD}, Remark \ref{temp}.)
  \end{remark}

 \ \\
 \noindent
{\bf{ii \ (Littlewood).}} \  From $\kappa_L < \infty$,  by duality (Parseval, Riesz-Kakutani, Hahn-Banach), for all sets $A$ and $B$, and $\phi \in l^{\infty,2}(A,B)$, i.e., 
\begin{equation}
\|\phi\|_{\infty,2} \ \overset{\text{def}}{=} \ \sup_{\alpha \in A} \big(\sum_{\beta \in B} |\phi(\alpha,\beta)|^2\big)^{\frac{1}{2}} \ < \ \infty, 
\end{equation}
there exist $\chi  \in M(\Omega_A \times \Omega_B),$ such that
\begin{equation}
\widehat{\chi}(r_{\alpha},r_{\beta}) \ = \ \phi(\alpha,\beta), \ \ \ \ \alpha \in A, \ \beta \in B,
\end{equation}
and
\begin{equation}
\|\chi\|_M \ \leq \ \kappa_L  \ \|\phi\|_{\infty,2}.
\end{equation}\\
That is, we have an injection\\ 
\begin{equation}
\begin{split}
\boldsymbol{\chi}_{AB}: \ l^{\infty,2}(A,B)& \ \rightarrow \ M(\Omega_A \times \Omega_B)/M_{(R_A \times R_B)^c}\\\\
&\big( \ (R_A \times R_B)^c \ =  \ W_A \times W_B \ \setminus  \ R_A \times R_B \ \big),
\end{split}
\end{equation}
given by  
\begin{equation}
\boldsymbol{\chi}_{AB}(\phi) \ = \ \big \{ \chi \in M(\Omega_A \times \Omega_B): \widehat{\chi}(r_{\alpha} \otimes r_{\beta}) = \phi(\alpha,\beta), \ (\alpha,\beta) \in A \times B \big\},
\end{equation}
such that 
\begin{equation}
\vertiii{\boldsymbol{\chi}_{AB}(\phi)} \ \leq \ \kappa_L \ \|\phi\|_{\infty,2}, \ \ \ \ \ \phi \in l^{\infty,2}(A,B).
\end{equation}\\
($\vertiii{\boldsymbol{\cdot}} = $ the quotient norm.)  Therefore, there exists a $\big(l^{\infty,2}(A,B) \hookrightarrow M(\Omega_A \times \Omega_B) \big)$-interpolant (\emph{choice function})
 \begin{equation} \label{mapL1}
 \chi_{AB} = U_{R_A \times R_B} + \mathfrak{p}_{AB}:  \ l^{\infty,2}(A,B) \ \rightarrow \ M(\Omega_A \times \Omega_B),
 \end{equation}
 where 
\begin{equation} \label{mapL}
\mathfrak{p}_{AB}: \ l^{\infty,2}(A,B) \ \rightarrow  \ \mathfrak{S}_{W_A \times W_B  \setminus   R_A \times R_B}^{\infty}
\end{equation}\\
is an orthogonal $\big(l^{\infty,2}(A,B),M(\Omega_A \times \Omega_B)\big)$-perturbation, and 
\begin{equation}
\|\chi_{AB}\|_{l^{\infty,2} \hookrightarrow M}  \ \leq\ \kappa_L. 
\end{equation}\\
(In Remark \ref{temp}, take $\Gamma = W_A \times W_B,$ \ $E = R_A \times R_B$, \ $\mathfrak{X} = l^{\infty,2}(A,B)$, \  $\mathfrak{B} = M(\Omega_A \times \Omega_B).$)

Denote
\begin{equation}
\mathcal{I}_{l^{\infty,2} \hookrightarrow M}(A,B) \ = \ \big\{\text{$(l^{\infty,2} \hookrightarrow M)$-interpolants on $R_A \times R_B$} \big\},
\end{equation}
and let
\begin{equation} 
\kappa_L^{\star} \ \overset{\text{def}}{=} \ \sup_{A,B} \ \inf \big \{ \|\chi_{AB}\|_{l^{\infty,2} \hookrightarrow M}: \  \chi_{AB} \in \mathcal{I}_{l^{\infty,2} \hookrightarrow M}(A,B) \big \}.
\end{equation}\


\begin{proposition} \label{LittleD}
\begin{equation}
 \kappa_L^{\star} \ = \  \kappa_L.
\end{equation}
\end{proposition}
\noindent
(Proof similar to that of Proposition \ref{KhintD}.)
\ \\

\noindent
{\bf{iii \ (Orlicz).}} \  For the dual statement of the Orlicz $(C_{R \times R} \hookrightarrow l^{2,1})$-inequality $\kappa_O < \infty$, replace $l^{\infty,2}(A,B)$ in {\bf{ii}}  with $l^{2}\big(A;l^{\infty}(B)\big) = l^{2,\infty}(A,B)$, and let
\begin{equation} 
 \kappa_O^{\star} \ \overset{\text{def}}{=}\ \sup_{A,B} \ \inf_{\chi_{AB}} \ \|\chi_{AB}\|_{l^{2,\infty} \hookrightarrow M},
 \end{equation}
 with supremum over sets $A$ and $B$, and infimum over $(l^{2,\infty} \hookrightarrow M)$-interpolants on $R_A \times R_B$.\

 \begin{proposition}  [cf. Propositions \ref{KhintD}, \ref{LittleD}] \label{OrliczD} 
 \begin{equation}
 \kappa_O^{\star} \ = \ \kappa_O.
\end{equation} 
 \end{proposition}
 
 \ 
 \begin{remark}['real' vs. 'complex'] \ Each constant considered here has two values depending on the underlying scalar field, and the relation between the two values is not always obvious.  For example, in the case of the Grothendieck inequality,
 \begin{equation} 
  \mathcal{K}_G^{\mathbb{C}} \ <  \  \mathcal{K}_G^{\mathbb{R}}
 \end{equation}
 (e.g., see \citet{pisier2012grothendieck} and \citet{braverman2011grothendieck}), whereas in the case of the Khintchin $(L^1_R \hookrightarrow L^2)$-inequality,
 \begin{equation} 
 \text{\emph{the 'complex'}} \  \kappa \ =  \ \text{\emph{the 'real'}} \  \kappa  \  \ ( \ =  \ \sqrt{2} \ ).
 \end{equation}
(See \citet{szarek1976best}.)  For our purposes, we will henceforth ignore distinctions between the 'real' and 'complex' constants, and will  assume the 'default' scalar field always to be $\mathbb{C}$, which, modulo 'best' constants, subsumes $\mathbb{R}$ in every case.  
That is, in every instance, 
 \begin{equation}
 \text{\emph{the 'complex' constant}} < \infty \ \Leftrightarrow \ \text{\emph{the 'real' constant}} < \infty.
 \end {equation}
 \end{remark}
 
\noindent
\subsection{Notion of an upgrade} \label{upgrade3} \ The inequalities 
\begin{equation} \label{eq2}
 \mathcal{K}_G < \infty, \ \  \kappa < \infty, \ \ \kappa_L < \infty, \ \  \kappa_O < \infty,
\end{equation}
in their equivalent dual forms
\begin{equation} \label{eq1}
\mathcal{K}^{\star} < \infty, \ \  \kappa^{\star} < \infty,\ \ \kappa_L^{\star}< \infty, \ \  \kappa_O^{\star} < \infty
\end{equation}
(Propositions  \ref{equivalence}, \ref{KhintD}, \ref{LittleD}, \ref{OrliczD}) imply  existence of uniformly bounded \emph{dual maps}:
\begin{equation}
\begin{split}
&(l^2 \hookrightarrow L^{\infty})\text{-representations  \ \ (Grothendieck),}\\\
&(l^2 \hookrightarrow L^{\infty})\text{-interpolants  \ \ (Khintchin),}\\\
&(l^{\infty,2} \hookrightarrow M)\text{-interpolants \ \ (Littlewood),}\\\
  \text{and}  \ \ \ &\\
&(l^{2,\infty} \hookrightarrow M)\text{-interpolants  \ \ (Orlicz).}
\end{split}
\end{equation} 
Namely, starting from an inequality in \eqref{eq1}, 
we obtain a dual map by selecting an element from each equivalence class in a quotient space, whereby the inequality guarantees that elements can be  
chosen to be uniformly bounded, and guarantees no more. An \emph{upgrade} is the feasibility of selecting an element  in each equivalence class so that the resulting dual map is bounded by an absolute constant, and possesses also an additional property not a priori guaranteed by the inequality proper.

\begin{remark} [more constants...] \ Typically, an upgrade is characterized by a constant greater than or equal to the inequality's constant in \eqref{eq1}.  
For example, existence of uniformly bounded $(l^2 \hookrightarrow L^{\infty})$-interpolants  
that are also $(l^2 \rightarrow L^2)$-continuous is conveyed by
\begin{equation} \label{whereas}
 \kappa^{\star,c} \ \overset{\text{def}}{=} \  \sup_{A}  \inf \big \{ \|G_A\|_{l^2 \hookrightarrow L^{\infty}}:   \text{continuous} \ (l^2 \hookrightarrow L^{\infty})\text{-interpolants} \  G_A  \big \} \ < \ \infty.
\end{equation}
A construction based on Riesz products  (\S \ref{RieszP}.{\bf{i}} below) verifies  $\kappa^{\star,c} < \infty$.  Obviously, \ $\kappa^{\star} \ \leq \ \kappa^{\star,c}$; see \eqref{Khint01}. But it is not known whether  \ $\kappa^{\star} \ < \ \kappa^{\star,c}$.
\end{remark}

Next we verify an upgrade of the Khintchin inequality used in the next section to prove $ \mathcal{K}^{\star} < \infty$.  To start, note that $\kappa = \kappa^* = \sqrt{2}$  \ implies for arbitrary set $A$, existence of orthogonal $(l^2, L^{\infty})$-perturbations   
\begin{equation} \label{best2} 
 \mathfrak{p}_A: \ l^2(A)  \ \rightarrow  \ \mathfrak{S}^2_{W_A \setminus R_A} \ \ (\  = L^2_{W_A \setminus R_A } \ ),
\end{equation}  
 with the norm estimate
 \begin{equation} \label{best}
\|\mathfrak{p}_A\|_{l^2 \hookrightarrow L^2} \ \leq \ 1.
 \end{equation}\\
 (See \eqref{e1}.)  Can the estimate in \eqref{best} be improved?  In particular, can $\mathfrak{p}_A$  be chosen arbitrarily 'small?'
 To make the question precise, let
\begin{equation} \label{equiv7}
\mathcal{G}_{A,\delta}({\bf{x}}) \ = \ \big\{g \in \mathcal{G}_A({\bf{x}}): \|g - U_{R_A}{\bf{x}}\|_{L^2} \leq \delta \big\}, \ \ \ \  {\bf{x}} \in \mathcal{S}_{l^2(A)}, \ \  \delta \geq 0,
\end{equation}
where $\mathcal{G}_A({\bf{x}})$ is defined in \eqref{equiv2}, and then 
\begin{equation} \label{uniform6}
\begin{split}
\beta(A,\delta) \ & \overset{\text{def}}{=}  \ \sup \bigg \{ \inf \big \{\|g\|_{L^{\infty}}: g \in \mathcal{G}_{A,\delta}({\bf{x}}) \big \}: {\bf{x}} \in  \mathcal{S}_{l^2(A)} \bigg \},\\\
\beta_R(\delta) \ &\overset{\text{def}}{=}  \  \sup_{A} \beta(A,\delta), \ \ \ \ \ \delta \geq 0.
\end{split}
\end{equation}
The function   $\beta_R: \ [0,\infty) \rightarrow [0,\infty]$ \ is obviously non-increasing. \ 
 At $\delta = 0$, 
\begin{equation} \label{uniform7}
\beta_R(0) = \infty \ \ \  \text{(because $L^2_{R_A} = L^{\infty}_{R_A}$ only if $A$ is finite),}
\end{equation}
and for $\delta \geq 1$, 
\begin{equation} \label{uniform8}
\beta_R(\delta)  =  \sqrt{2}, \ \ \  \text{(because $\kappa =   \sqrt{2}$).}
\end{equation}
The question whether $(l^2,L^{\infty})$-perturbations can be chosen arbitrarily 'small' becomes:  is $\beta_R(\delta) < \infty$ for arbitrary $\delta \in (0,1)$?  
To resolve the issue, we use the Khintchin $(L^2_R \hookrightarrow L^p)$-inequalities (in their equivalent form, as per Remark \ref{expsq})   
\begin{equation} \label{Khinteq}
\sup \bigg\{\int_{\Omega_A} e^{ |U_{R_A}{\bf{x}}|^2} \ d\mathbb{P}_A : \text{sets} \ A, \  {\bf{x}} \in \mathcal{S}_{l^2(A)} \bigg\} \ \overset{\text{def}}{=} \ \mathfrak{K} \ < \infty,
\end{equation}
and the function
\begin{equation}
\mathfrak{e}(\xi) \ = \ 2\sqrt{\mathfrak{K}} \ \xi  \ e^{ -\frac{\xi^2}{2}}, \ \ \ \ \xi \in [1,\infty),
\end{equation}
with its inverse
\begin{equation}
\xi \ = \ \mathfrak{e}^{-1}(u), \ \ \ \ u \in (0, 2\sqrt{e\mathfrak{K}} \ ].
\end{equation}

\noindent 
\begin{lemma} \label{trunc2} \ For set  $A$, \ ${\bf{x}} \in \mathcal{S}_{l^2(A)}$,  and $\delta \in (0,1)$,  there exist  $g \in \mathcal{G}_{A,\delta}({\bf{x}})$,  such that
\begin{equation}
\|g\|_{L^{\infty}} \ \leq \  \mathfrak{e}^{-1}(\delta)+\frac{\delta}{\sqrt{2}}.
\end{equation} 
\end{lemma}
\begin{proof}  \  Given $\delta \in (0,1)$, let 
 \begin{equation}
 \xi \ = \  \mathfrak{e}^{-1}(\delta).
 \end{equation}
 For ${\bf{x}} \in \mathcal{S}_{l^2(A)}$, \  let
\begin{equation} \label{khint55} 
h_{{\bf{x}},\xi} \  \overset{\text{def}}{=} \  \left \{
\begin{array}{lcc}
 U_{R_A}{\bf{x}}, & \ \ \   |U_{R_A}{\bf{x}}| \ \leq \ \xi  \\\\  0,   &\text{otherwise}, 
\end{array} \right.
\end{equation} \\
and  \ 
\begin{equation}
\phi_{{\bf{x}},\xi}  \ \overset{\text{def}}{=} \  U_{R_A}{\bf{x}} - h_{{\bf{x}},\xi}.
\end{equation}\\
 Using \eqref{Khinteq}, we estimate 
\begin{equation} \label{khint22}
\begin{split}
  \int_{\Omega_A} |\phi_{{\bf{x}},\xi}|^2 \ d\mathbb{P}_A \ &= \  \int_{\big\{|U_{R_A}{\bf{x}}| \ > \ \xi \big\}} e^{|U_{R_A}{\bf{x}}|^2} \  \big(|U_{R_A}{\bf{x}}|^2 \ e^{-|U_{R_A}{\bf{x}}|^2}\big)  \ d\mathbb{P}_A \\\\
 &\leq  \ \xi^{2}  e^{ -\xi^2} \ \mathfrak{K}.
 \end{split}
\end{equation}\\
Splitting $\phi_{{\bf{x}},\xi}$ into two Walsh series spanned by $R_A$ and $W_A \setminus R_A$,   
we obtain from  \eqref{khint22} \\
\begin{equation} \label{khint44}
\big\|\sum_{\alpha \in A}\widehat{\phi_{{\bf{x}},\xi}}(r_{\alpha}) r_{\alpha}\big\|_{L^2} \ \leq \ \xi e^{ -\frac{\xi^2}{2}} \ \sqrt{\mathfrak{K}},  \ \ \ \ \ \ \big\|\sum_{w \in W_A \setminus R_A}\widehat{\phi_{{\bf{x}},\xi}}(w) w \big\|_{L^2} \ \leq \ \xi e^{ -\frac{\xi^2}{2}} \ \sqrt{\mathfrak{K}}.
\end{equation}\\
Taking  $G_A = U_{R_A} + \mathfrak{p}_A$ \  in  \eqref{best2}, and applying it to ${\bf{v}}_{{\bf{x}}} \in l^2(A)$, where  $${\bf{v}}_{{\bf{x}}}(\alpha) \  \overset{\text{def}}{=}  \ \widehat{\phi_{{\bf{x}},\xi}}(r_{\alpha}), \ \ \alpha \in A,$$\\ 
we obtain \
\begin{equation}
G_A({\bf{v}}_{{\bf{x}}}) \ = \ \sum_{\alpha \in A}\widehat{\phi_{{\bf{x}},\xi}}(r_{\alpha}) r_{\alpha} \  + \ \mathfrak{p}_A({\bf{v}}_{{\bf{x}}}),
\end{equation}\\
\begin{equation} \label{khint333}
\|G_A({\bf{v}}_{{\bf{x}}})\|_{L^{\infty}} \ \leq  \ \xi e^{ -\frac{\xi^2}{2}} \ \sqrt{2\mathfrak{K}}, \ \ \ \ \ \ \ \|\mathfrak{p}_A({\bf{v}}_{{\bf{x}}}\|_{L^2} \ \leq \ \xi e^{ -\frac{\xi^2}{2}} \ \sqrt{\mathfrak{K}}.
\end{equation}\\
Putting it all together,  we have 
\begin{equation}
\begin{split}
g \ \overset{\text{def}}{=}  \ h_{{\bf{x}},\xi} + G_A({\bf{v}}_{{\bf{x}}}) \ &=  \ U_{R_A}{\bf{x}} \ + \ \bigg(\mathfrak{p}_A({\bf{v}}_{{\bf{x}}}) \ - \ \sum_{w \in W_A \setminus R_A}\widehat{\phi_{{\bf{x}},\xi}}(w) w \bigg)\\\
&  \overset{\text{def}}{=}  \ \ U_{R_A}{\bf{x}} \ + \ \widetilde{\mathfrak{p}_A}({\bf{x}}),
\end{split}
\end{equation}\\
whence, from \eqref{khint333}, \eqref{khint44}, and \eqref{khint55},\\
\begin{equation} \label{khint6}
\|g\|_{L^{\infty}} \ \leq \ \xi  +  \xi e^{ -\frac{\xi^2}{2}} \ \sqrt{2\mathfrak{K}} \ =  \  \mathfrak{e}^{-1}(\delta)+\frac{\delta}{\sqrt{2}},
\end{equation}
and
\begin{equation} \label{khint77}
\| g   -  U_{R_A}{\bf{x}}\|_{L^2} \ \leq \ 2\xi e^{ -\frac{\xi^2}{2}} \ \sqrt{\mathfrak{K}} \ = \delta.
\end{equation}\\
\end{proof}
\begin{corollary} [cf. {\citet[Corollary III.10]{Blei:2001}}] \label{trunc3} \ 
\begin{equation} \label{trunc1}
\sqrt{2} \  =  \ \kappa \ \leq \ \beta_R(\delta) \  \leq \  \mathfrak{e}^{-1}(\delta)+\frac{\delta}{\sqrt{2}}, \ \ \ \ \delta \in (0,1).
\end{equation} 
\end{corollary}
 
\begin{remark} [uniformizability] \label{upgrade4}   \ Corollary \ref{trunc3} implies that for every $\delta > 0$, there exist \ $\beta  < \infty$, such that for each set $A$, there are orthogonal $(l^2,L^{\infty})$-perturbations $\mathfrak{p}_A$ (of $U_{R_A}$) satisfying
\begin{equation} \label{modulus1}
\begin{split}
 \|U_{R_A} + \mathfrak{p}_A\|_{l^2 \hookrightarrow L^{\infty}} \ \leq \ \beta \ \  \ \text{and} \ \  \  \|\mathfrak{p}_A\|_{l^2 \hookrightarrow L^2} \ \leq \ \delta .
\end{split}
\end{equation}
Succinctly put, 
\begin{equation} \label{unif7}
\beta_R(\delta) < \infty, \ \ \  \delta > 0,
\end{equation}
where $\beta_R$ is defined in \eqref{uniform6}.   We refer to the property expressed by  \eqref{unif7} as  $\Lambda(2)$-\emph{uniformizability} \citep{Blei:1977uq}, and to $\beta_R$ as the $\Lambda(2)$-\emph{modulus} of Rademacher systems.  Little is known about  $\beta_R$ beyond the estimate in \eqref{trunc1}.  (Is the estimate sharp? Is  $\beta_R$ everywhere continuous?) 
\end{remark}
 
 \begin{remark} [open questions] \label{open1} \ Suppose $E \subset W_A$ is $\Lambda(2)$, i.e.,  $\kappa(E) \overset{\text{def}} {=} \kappa_1(E,2) < \infty$ (Remark \ref{general}).  Let  
\begin{equation} \label{equiv8}
\mathcal{G}_{E,\delta}({\bf{x}}) \ \overset{\text{def}}{=} \ \big\{g \in \mathcal{G}_E({\bf{x}}): \|g - U_{E}{\bf{x}}\|_{L^2} \leq \delta \big\}, \ \ \ \  {\bf{x}} \in \mathcal{S}_{{l^2(E)}}, \ \  \delta \geq 0,
\end{equation}
 where $\mathcal{G}_E({\bf{x}})$ is given by \eqref{equiv6}, and define the $\Lambda(2)$-modulus of $E$ 
\begin{equation}
\beta_E(\delta) \ \overset{\text{def}}{=} \ \sup \bigg \{ \inf \big \{\|g\|_{L^{\infty}}: g \in \mathcal{G}_{E,\delta}({\bf{x}}) \big \}: {\bf{x}} \in \mathcal{S}_{l^2(E)} \bigg \}, \ \ \ \ \ \delta \geq 0.
\end{equation}
(Cf. \eqref{equiv7} and \eqref{uniform6}.) If $\beta_E(\delta) < \infty$ for all $\delta > 0$, then  $E$ is $\Lambda(2)$-\emph{uniformizable}. 

 Whether every $\Lambda(2)$-set is $\Lambda(2)$-uniformizable is an open problem. 
 The best known result in this direction is Lemma \ref{digress1} below (derived by modifying the proof of Lemma \ref{trunc2}), which implies that if  $E \subset W_A$ is $\Lambda(p)$ for some $p > 2$, i.e., 
\begin{equation} \label{digress2}
\kappa_2(E,p) \  < \  \infty,
\end{equation} 
then $E$ is $\Lambda(2)$-uniformizable.   Notably, whether every $\Lambda(2)$-set is $\Lambda(p)$ for some $p > 2$ is an open problem of long standing, known as the $\Lambda(2)$-\emph{set problem}.  (See \S \ref{open}.)
 
\begin{lemma} \label{digress1}  If $E \subset W_A$ is $\Lambda(p)$ for some $p > 2,$ then for all ${\bf{x}} \in l^2(E)$ and $\delta > 0$,  there exist \ $g \in \mathcal{G}_{E,\delta}({\bf{x}})$,  \ such that
\begin{equation}
\|g\|_{L^{\infty}} \ \leq \  C_p \ \delta^{\frac{2}{2-p}} + \delta/2,
\end{equation} 
where
\begin{equation} \label{Cp}
C_p \ = \ 4^{\frac{1}{p-2}} \  [\kappa_2(E,p)]^{\frac{p}{p-2}} \ [\kappa(E)]^{\frac{2}{p-2}}.
\end{equation} 

\end{lemma}
\begin{proof}  \  For $\xi > 1$ (to be specified), and for arbitrary ${\bf{x}} \in \mathcal{S}_{l^2(E)}$, \  let
\begin{equation} \label{khint555} 
h_{{\bf{x}},\xi} \  \overset{\text{def}}{=} \  \left \{
\begin{array}{lcc}
 U_{E}{\bf{x}}, & \ \ \   |U_{E}{\bf{x}}| \ \leq \ \xi  \\\\  0,   &\text{otherwise}, 
\end{array} \right.
\end{equation} \\
and  \ 
\begin{equation}
\phi_{{\bf{x}},\xi}  \ \overset{\text{def}}{=} \  U_{E}{\bf{x}} - h_{{\bf{x}},\xi}.
\end{equation}\\
 Applying \eqref{digress2}, we estimate 
\begin{equation} \label{khint222}
\begin{split}
  \int_{\Omega_A} |\phi_{{\bf{x}},\xi}|^2 \ d\mathbb{P}_A \ &= \  \int_{\big\{|U_{E}{\bf{x}}| \ > \ \xi \big\}} |U_{E}{\bf{x}}|^p \  |U_{E}{\bf{x}}|^{2-p}  \ d\mathbb{P}_A \\\\
 &\leq  \ \xi^{2-p} \ [\kappa_2(E,p)]^p,
 \end{split}
\end{equation}
and thus obtain\\
\begin{equation} \label{khint444}
\big\|\sum_{w \in E}\widehat{\phi_{{\bf{x}},\xi}}(w) w\big\|_{L^2} \ \leq \ \xi^{\frac{2-p}{2}} \ [\kappa_2(E,p)]^{\frac{p}{2}} ,  \ \ \ \ \ \ \big\|\sum_{w \in W_A \setminus E}\widehat{\phi_{{\bf{x}},\xi}}(w) w \big\|_{L^2} \ \leq \  \xi^{\frac{2-p}{2}} \ [\kappa_2(E,p)]^{\frac{p}{2}}. \ \ \ \ \ \
\end{equation}\\
Noting that $E \subset W_A$ is \emph{a fortiori} $\Lambda(2)$, i.e., $\kappa(E) < \infty$ (Lemma \ref{trick1}), we can apply the interpolant $G_E$ in  \eqref{general1} to ${\bf{v}}_{{\bf{x}}} \in l^2(E)$, where  $${\bf{v}}_{{\bf{x}}}(w) \  \overset{\text{def}}{=}  \ \widehat{\phi_{{\bf{x}},\xi}}(w), \ \ w \in E,$$\\ 
and obtain \
\begin{equation}
G_E({\bf{v}}_{{\bf{x}}}) \ = \ \sum_{w \in E}\widehat{\phi_{{\bf{x}},\xi}}(w)\ w \  + \ \mathfrak{p}_E({\bf{v}}_{{\bf{x}}}),
\end{equation}\\
\begin{equation} \label{khint33}
\begin{split}
\|G_E({\bf{v}}_{{\bf{x}}})\|_{L^{\infty}} \ &\leq  \ \xi^{\frac{2-p}{2}} \ [\kappa_2(E,p)]^{\frac{p}{2}} \ \kappa(E),\\\\
 \|\mathfrak{p}_E({\bf{v}}_{{\bf{x}}}\|_{L^2} \ & \leq  \ \xi^{\frac{2-p}{2}} \ [\kappa_2(E,p)]^{\frac{p}{2}} \ \big(\sqrt{[\kappa(E)]^2 - 1)}\big).
 \end{split}
\end{equation}\\
Let 
\begin{equation}
\begin{split}
g \ \overset{\text{def}}{=}  \ h_{{\bf{x}},\xi} + G_E({\bf{v}}_{{\bf{x}}}) \ &=  \ U_{E}{\bf{x}} \ + \ \bigg(\mathfrak{p}_E({\bf{v}}_{{\bf{x}}}) \ - \ \sum_{w \in W_A \setminus E}\widehat{\phi_{{\bf{x}},\xi}}(w) w \bigg)\\\
&  \overset{\text{def}}{=}  \ \ U_{E}{\bf{x}} \ + \ \widetilde{\mathfrak{p}_E}({\bf{x}}),
\end{split}
\end{equation}\\
and deduce from \eqref{khint33}, \eqref{khint444}, and \eqref{khint555},\\
\begin{equation} \label{khint66}
\|g\|_{L^{\infty}} \ \leq \ \xi  +  \xi^{\frac{2-p}{2}} \ [\kappa_2(E,p)]^{\frac{p}{2}} \ \kappa(E),
\end{equation}
and
\begin{equation} \label{khint777}
\begin{split}
\| \widetilde{\mathfrak{p}_E}({\bf{x}})\|_{L^2} \ &\leq \ \xi^{\frac{2-p}{2}} \ [\kappa_2(E,p)]^{\frac{p}{2}} \bigg(\sqrt{[\kappa(E)]^2 - 1} \ +  \ 1\bigg)\\\\
& \leq \ 2\xi^{\frac{2-p}{2}} \ [\kappa_2(E,p)]^{\frac{p}{2}} \ \kappa(E).\\\\
\end{split}
\end{equation}\\
Putting  
\begin{equation}
\xi \ = \ \big(\delta/2[\kappa_2(E,p)]^{\frac{p}{2}} \ \kappa(E)\big)^{\frac{2}{p-2}} 
\end{equation}\\
in \eqref{khint66} and \eqref{khint777} completes the proof.
\end{proof}
\end{remark}

\begin{remark}[ultra-interpolants] \label{upgrade2} \ In the next section we produce $A^{\prime} \supset A$, and  $\big(l^2(A) \hookrightarrow L^{\infty}(\Omega_{A^{\prime}},\mathbb{P}_{A^{\prime}})\big)$-interpolants  $\Phi_{A}$ (on $R_A$) that satisfy
\begin{equation} \label{as}
{\bf{x}} \boldsymbol{\cdot} {\bf{y}} \ = \ \int_{\Omega_{A^{\prime}}} \Phi_{A}({\bf{x}}) \ \Phi_{A}({\bf{y}}) \ d\mathbb{P}_{A^{\prime}}, \ \ \ \ \ {\bf{x}} \in l^2(A), \ \ {\bf{y}} \in l^2(A),
\end{equation}\\
and
\begin{equation}
\|\Phi_{A}({\bf{x}}\|_{L^{\infty}} \ \leq \ K  \|{\bf{x}}\|_2, \ \ \ \ \ {\bf{x}} \in l^2(A),
\end{equation}\\
where  $K > 1$ does not depend on $A$. (See Theorem \ref{MT1} and Corollary \ref{ultraG}.)  Existence of $(l^2 \hookrightarrow L^{\infty})$-interpolants that are also $(l^2 \hookrightarrow L^{\infty})$-representations upgrades both the Khintchin inequality and the Grothendieck inequality in the sense stated at the start of the section. We refer to such maps as \emph{ultra-interpolants.}

\noindent


\end{remark} 

\section{\bf{The Grothendieck inequality: upgrades and extensions}}
\subsection{From Parseval to Grothendieck (via Khintchin)} \label{rev} \  An interpolant  
\begin{equation}
G_A = U_{R_A} + \mathfrak{p}_{A^{\prime}}: \ l^2(A) \rightarrow L^{\infty}(\Omega_{A^{\prime}},\mathbb{P}_{A^{\prime}}), \ \ \ \ A^{\prime} \supset A,
\end{equation} 
is an ultra-interpolant if and only if for all  \ ${\bf{x}}$  \ and  ${\bf{y}}$  in  $l^2(A)$,
\begin{equation} \label{short11}
\widehat{\mathfrak{p}_{A^{\prime}}({\bf{x}})} \boldsymbol{\cdot} \widehat{\mathfrak{p}_{A^{\prime}}({\bf{y}})} \  = \ \int_{\Omega_{A^{\prime}}} \mathfrak{p}_{A^{\prime}}({\bf{x}}) \ \mathfrak{p}_{A^{\prime}}({\bf{y}}) \ d\mathbb{P}_{A^{\prime}} \ = \  0.
\end{equation} 
Namely, the Parseval formula
 \begin{equation} \label{short1}
\begin{split}
\int_{\Omega_{A^{\prime}}} G_A({\bf{x}})  \ G_A({\bf{y}}) \ d\mathbb{P}_{A^{\prime}} \ & = \  \int_{\Omega_{A^{\prime}}} U_{R_A}{\bf{x}} \ U_{R_A}{\bf{y}} \ d\mathbb{P}_{A^{\prime}} \ + \  \int_{\Omega_{A^{\prime}}} \mathfrak{p}_{A^{\prime}}({\bf{x}}) \ \mathfrak{p}_{A^{\prime}}({\bf{y}}) \ d\mathbb{P}_{A^{\prime}},\\\\
 = \  {\bf{x}} \boldsymbol{\cdot} {\bf{y}} \ &+ \ \sum_{w \in W_{A^{\prime}} \setminus R_A} \widehat{\mathfrak{p}_{A^{\prime}}({\bf{x}})}(w) \ \widehat{\mathfrak{p}_{A^{\prime}}({\bf{y}})}(w), \ \ \ \  \ \ {\bf{x}} \in l^2(A), \ \ {\bf{y}} \in l^2(A),
\end{split}
\end{equation}  
morphs into the Parseval-like formula in \eqref{parseval2} precisely when \eqref{short11} holds for all ${\bf{x}}$  and   ${\bf{y}}$  in  $l^2(A)$.
This simple observation is the motivation behind the recursive construction of the ultra-interpolants below.

To start, given an arbitrary set $A$, we initialize $A_1 = A$, and generate from it a sequence of disjoint sets 
\begin{equation} \label{rec3}
A_{j+1} \ \overset{\text{def}}{=} \ W_{A_j} \setminus (R_{A_j} \cup \{r_0\}), \ \ \ \ j = 1, \ldots. \ 
\end{equation}\\
(We refer to $\big(A_j\big)_j$ as a \emph{set-cascade generated by $A$}.) Fix  $\delta \in (0,1)$, and let $$\beta = \beta_R(\delta) < \infty$$ (Corollary \ref{trunc3}).  For $j = 1, \ldots,$ select $\big(l^2(A_j) \hookrightarrow L^{\infty}(\Omega_{A_j},\mathbb{P}_{A_j})\big)$-interpolants 
  \begin{equation}
 G_{A_j,\delta} = U_{R_{A_j}} + \mathfrak{p}_{A_j,\delta},
 \end{equation}
 such that for ${\bf{x}} \in l^2(A_j)$,
\begin{equation} \label{rec2}
 \|G_{A_j,\delta}({\bf{x}})\|_{L^{\infty}} \ \leq \ \beta\|{\bf{x}}\|_2 \ \ \ \ \text{and} \ \ \ \ \|\mathfrak{p}_{A_j,\delta}({\bf{x}})\|_{L^2} \ \leq \ \delta\|{\bf{x}}\|_2 \ .
 \end{equation}
 We can assume 
 \begin{equation} \label{symmetry}
 \int_{\Omega_{A_j}}G_{A_j,\delta}({\bf{x}}) \ d\mathbb{P}_{A_j} \ = \ 0, \ \ \ \ \ {\bf{x}} \in l^2(A), \ \ j = 1, \ldots \ .
 \end{equation}
 Let
 \begin{equation} \label{disjoint}
 A^{\prime} \ = \ \bigcup_{j=1}^{\infty} A_j,
 \end{equation}
 and consider the product space (a compact abelian group)
 \begin{equation} \label{disjoint1}
 \Omega_{A^{\prime}} \ = \ \bigtimes_{j=1}^{\infty} \Omega_{A_j},
 \end{equation}
 with the uniform probability measure on it (normalized Haar measure)
 \begin{equation} \label{disjoint2}
 \mathbb{P}_{A^{\prime}} \ = \ \bigtimes_{j=1}^{\infty} \mathbb{P}_{A_j}.
 \end{equation}
 We take the independent projections
 \begin{equation}
\pi_{A_j}: \  \Omega_{A^{\prime}} \ \rightarrow \ \Omega_{A_j}, \ \ \ j = 1, \ldots,
 \end{equation}
 given by
 \begin{equation}
 \pi_{A_j}(\boldsymbol{\omega}) \ = \ \big(\boldsymbol{\omega}(\alpha)\big)_{\alpha \in A_j}, \ \ \ \ \ \boldsymbol{\omega}  \in   \{-1,+1\}^{A^{\prime}},
 \end{equation}
and view the $G_{A_j,\delta}$ as independent $\big(l^2(A_j) \hookrightarrow L^{\infty}(\Omega_{A^{\prime}},\mathbb{P}_{A^{\prime}})\big)$-interpolants  $G_{j,\delta}$, where\\
\begin{equation} \label{view}
G_{j,\delta}({\bf{x}}) \ = \ G_{A_j,\delta}({\bf{x}}) \circ \pi_{A_j} \ = \ U_{R_{A_j}}{\bf{x}} \ + \ \mathfrak{p}_{j,\delta}({\bf{x}}), \ \ \ \ {\bf{x}} \in l^2(A_j), \ \ j = 1, \ldots \ ,
\end{equation}
and
\begin{equation}
\mathfrak{p}_{j,\delta}({\bf{x}}) \ = \ \mathfrak{p}_{A_j,\delta}({\bf{x}})\circ \pi_{A_j}.
\end{equation}\\
For $j = 1, \ldots,$ and ${\bf{x}} \in l^2(A_j)$,\\
\begin{equation} \label{transform0}
\widehat{G_{j,\delta}({\bf{x}})}(r_{\alpha}) \ = \ {\bf{x}}(\alpha), \ \ \ \ \alpha \in A_j,
\end{equation}
and
\begin{equation} \label{support}
\big \{ w \in W_{A^{\prime}}:  \widehat{G_{j,\delta}({\bf{x}})}(w) \ \neq  \ 0 \big \} \ \subset W_{A_j}.
\end{equation}\\
In particular, for  \ $ j_1 \neq j_2, \ \ {\bf{x}}_1 \in l^2(A_{j_1}),   \ {\bf{x}}_2 \in l^2(A_{j_2})$,\\ 
\begin{equation} \label{orthog}
\int_{\Omega_{A^{\prime}}} G_{j_1,\delta}({\bf{x}}_1) \ G_{j_2,\delta}({\bf{x}}_2) \ d\mathbb{P}_{A^{\prime}} \ = \ 0.
\end{equation}\\

Next, for arbitrary ${\bf{x}} \in {l^2(A)}$, we initialize ${\bf{x}}^{(1)} = {\bf{x}}$, and generate a sequence of vectors ${\bf{x}}^{(j)} \in l^2(A_j),$ defined recursively by
\begin{equation} \label{rec1}
{\bf{x}}^{(j+1)}(\alpha) \ = \ \mathfrak{p}_{j,\delta}({\bf{x}}^{(j)})^{\wedge}(\alpha), \ \ \ \ \alpha \in A_{j+1},  \ \ j = 1, \ldots \ .
\end{equation} 
(We refer to $({\bf{x}}^{(j)})_j$ as a \emph{vector-cascade} generated by ${\bf{x}} \in l^2(A)$.)

\begin{lemma} \label{rec10} \ For ${\bf{x}} \in {l^2(A)}$,
\begin{equation} \label{rec8}
\|{\bf{x}}^{(j)}\|_2 \ \leq \ \delta^{j-1} \ \|{\bf{x}}\|_2, \ \ \ \ j = 1, \ldots \ .
\end{equation}
\end{lemma}
\begin{proof}[Proof (by induction on $j$)]  \ The case $j=1$ is the initial assignment ${\bf{x}}^{(1)} = {\bf{x}}$.  For $j \geq 1$,\\
\begin{equation}
\|{\bf{x}}^{(j+1)}\|_2 \ = \ \| \mathfrak{p}_{j,\delta}({\bf{x}}^{(j)})\|_{L^2} \ \leq \ \delta \|{\bf{x}}^{(j)}\|_2 \leq \ \delta \cdot \delta^{j-1} \ \|{\bf{x}}\|_2, \\\\
\end{equation}\\
where the equality on the left follows from \eqref{rec1}, the first inequality from \eqref{rec2}, and the second from the induction hypothesis.
\end{proof} 
\begin{lemma} \ For ${\bf{x}}$ and ${\bf{y}}$ in ${l^2(A)}$, 
\begin{equation} \label{rec4} 
\begin{split}
\sum_{j=1}^{n}&(-1)^{j-1} \  \int_{\Omega_{A^{\prime}}} G_{j,\delta}({\bf{x}}^{(j)}) \  G_{j,\delta}({\bf{y}}^{(j)}) \ d\mathbb{P}_{A^{\prime}} \\\\
&= \  {\bf{x}}\boldsymbol{\cdot}{\bf{y}} \ +(-1)^{n-1}  \ \int_{\Omega_{A^{\prime}}}\mathfrak{p}_{n,\delta}({\bf{x}}^{(n)}) \ \mathfrak{p}_{n,\delta}({\bf{y}}^{(n)}) \ d\mathbb{P}_{A^{\prime}}, \ \ \ \  \ \ n = 1, \ldots \ .
\end{split} 
\end{equation}
\end{lemma}
\begin{proof}[Proof (by induction on $n$)] \ The case $n = 1$  is \eqref{short1}. For $n > 1,$
\begin{equation} \label{rec5}
\begin{split}
& \sum_{j=1}^{n}(-1)^{j-1}  \int_{\Omega_A^{\prime}} G_{j,\delta}({\bf{x}}^{(j)}) \  G_{j,\delta}({\bf{y}}^{(j)}) \ d\mathbb{P}_{A^{\prime}} \ = \  {\bf{x}}\boldsymbol{\cdot}{\bf{y}}  \\\\
& +(-1)^{n-2}   \int_{\Omega_{A^{\prime}}}\mathfrak{p}_{n-1,\delta}({\bf{x}}^{(n-1)}) \ \mathfrak{p}_{n-1,\delta}({\bf{y}}^{(n-1)}) \ d\mathbb{P}_A \ + \ (-1)^{n-1} \int_{\Omega_{A^{\prime}}}G_{n,\delta}({\bf{x}}^{(n)}) \  G_{n,\delta}({\bf{y}}^{(n)}) \ d\mathbb{P}_{A^{\prime}} \\\\
 &\ \ \ \ \ \ \ \ \ \ \ \ \ \ \ \ \ \ \ \ \ \ \ \ \ \ \ \text{(induction hypothesis}).\ 
\end{split}
\end{equation}
By \eqref{short1}, \eqref{rec3}, and \eqref{rec1} (via Parseval),
\begin{equation} \label{rec6}
\begin{split}
\int_{\Omega_{A^{\prime}}}&G_{n,\delta}({\bf{x}}^{(n)}) \  G_{n,\delta}({\bf{y}}^{(n)}) \ d\mathbb{P}_{A^{\prime}} \ = \ {\bf{x}}^{(n)} \boldsymbol{\cdot} {\bf{y}}^{(n)} \ + \  \int_{\Omega_{A^{\prime}}}\mathfrak{p}_{n,\delta}({\bf{x}}^{(n)}) \ \mathfrak{p}_{n,\delta}({\bf{y}}^{(n)}) \ d\mathbb{P}_A\\\\
&= \ \int_{\Omega_{A^{\prime}}}\mathfrak{p}_{n-1,\delta}({\bf{x}}^{(n-1)}) \ \mathfrak{p}_{n-1,\delta}({\bf{y}}^{(n-1)}) \ d\mathbb{P}_{A^{\prime}} \ + \ \int_{\Omega_{A^{\prime}}}\mathfrak{p}_{n,\delta}({\bf{x}}^{(n)}) \ \mathfrak{p}_{n,\delta}({\bf{y}}^{(n)}) \ d\mathbb{P}_{A^{\prime}},
\end{split}
\end{equation}\\
which, put in \eqref{rec5}, implies \eqref{rec4}.\\ 
\end{proof}

\begin{lemma} \label{Rep} \  \ For ${\bf{x}} \in {l^2(A)}$ and ${\bf{y}} \in {l^2(A)}$,
\begin{equation} \label{abcon}
\sum_{j=1}^{\infty}  \int_{\Omega_{A^{\prime}}} \big|G_{j,\delta}({\bf{x}}^{(j)}) \  G_{j,\delta}({\bf{y}}^{(j)})\big| \ d\mathbb{P}_{A^{\prime}} \ \leq \ \frac{\beta^2}{1-\delta^2} \  \|{\bf{x}}\|_2 \|{\bf{y}}\|_2,
\end{equation}
and
\begin{equation}  
\sum_{j=1}^{\infty}(-1)^{j-1} \  \int_{\Omega_{A^{\prime}}} G_{j,\delta}({\bf{x}}^{(j)}) \  G_{j,\delta}({\bf{y}}^{(j)}) \ d\mathbb{P}_{A^{\prime}} \ = \  {\bf{x}}\boldsymbol{\cdot}{\bf{y}}.
\end{equation}
\end{lemma}
\begin{proof} By H\"older, Cauchy-Schwarz,  \eqref{rec2}, and \eqref{rec8},\\
\begin{equation} \label{abcon1}
\begin{split}
\sum_{j=1}^{n} \int_{\Omega_{A^{\prime}}} \big|G_{j,\delta}({\bf{x}}^{(j)}) \  G_{j,\delta}({\bf{y}}^{(j)})\big| \ d\mathbb{P}_{A^{\prime}}\big| \ &\leq \ \sum_{j=1}^{n}  \ \|G_{j,\delta}({\bf{x}}^{(j)})\|_{L^{\infty}} \ \|G_{j,\delta}({\bf{y}}^{(j)})\|_{L^{\infty}}\\\\
&\leq \ \beta^2 \ \|{\bf{x}}\|_2 \  \|{\bf{y}}\|_2 \ \sum_{j=1}^{n}  \ \delta^{2(j-1)},
\end{split}
\end{equation}
and
\begin{equation} \label{rec9}
\begin{split}
\bigg|\int_{\Omega_{A^{\prime}}}\mathfrak{p}_{n,\delta}({\bf{x}}^{(n)}) \ \mathfrak{p}_{n,\delta}({\bf{y}}^{(n)}) \ d\mathbb{P}_{A^{\prime}}\bigg| \ &\leq \ \|\mathfrak{p}_{n,\delta}({\bf{x}}^{(n)})\|_{L^2} \  \|\mathfrak{p}_{n,\delta}({\bf{y}}^{(n)})\|_{L^2}\\\\
&\leq \ \delta^2 \ \|{\bf{x}}^{(n)}\|_2 \  \|{\bf{y}}^{(n)}\|_2 \\\\
&\leq \ \delta^{2n} \ \|{\bf{x}}\|_2 \  \|{\bf{y}}\|_2 \ .\\
\end{split}
\end{equation}\\
The lemma follows by letting $n \rightarrow \infty$ in \eqref{abcon1} and \eqref{rec4}.
\end{proof}
Putting it all together, we obtain

\begin{theorem} \label{MT1}  \ For $\delta \in (0,1)$ and set $A$,  
\begin{equation} \label{rec11}
  \Phi_{A,\delta}({\bf{x}}) \ \overset{\text{def}}{=} \ \ \sum_{j=1}^{\infty} \  \mathfrak{i}^{j-1} \ G_{j,\delta}({\bf{x}}^{(j)}), \ \ \ \ \ {\bf{x}} \in l^2(A)  \ \ \ \ (\mathfrak{i} = \sqrt{-1}),
\end{equation} 
converges absolutely in $L^{\infty}(\Omega_{A^{\prime}},\mathbb{P}_{A^{\prime}})$ and uniformly on bounded subsets of  \ $l^2(A)$.  Moreover, $\Phi_{A,\delta}$ is a $\big(l^2(A) \hookrightarrow L^{\infty}(\Omega_{A^{\prime}},\mathbb{P}_{A^{\prime}})\big)$-ultra-interpolant on $R_A$,
and
\begin{equation} \label{rec13}
\|\Phi_{A,\delta}\|_{\l^2 \hookrightarrow L^{\infty}} \ \leq \ \frac{\beta_R(\delta)}{1-\delta}.
\end{equation}
\end{theorem}
\begin{proof} \ By \eqref{rec2} and  \eqref{rec8}, 
\begin{equation} \label{rec12}
\begin{split}
\sum_{j=1}^{\infty}  \  \|G_{j,\delta}({\bf{x}}^{(j)})\|_{L^{\infty}} \ \leq \ \sum_{j=1}^{\infty}\beta \ \|{\bf{x}}^{(j)}\|_2  \ \leq \ \beta \ \sum_{j=1}^{\infty}\delta^{j-1} \ \|{\bf{x}}\|_2  \ & = \ \frac{\beta}{1-\delta} \  \|{\bf{x}}\|_2,,\\\
& \ \ \ {\bf{x}} \in l^2(A), \ \ \ \beta = \beta_R(\delta),\\\
\end{split}
\end{equation}
implying absolute convergence of the series in \eqref{rec11}, uniformly for ${\bf{x}}$  in bounded subsets of  \ $l^2(A)$.  

By \eqref{rec12}, \eqref{transform0}, and \eqref{support}, for ${\bf{x}} \in l^2(A)$,\\
\begin{equation}
\begin{split}
\widehat{\Phi_{A,\delta}({\bf{x}})}(r_{\alpha}) \ &= \ G_{1,\delta}(r_{\alpha})\\\\
&= \ {\bf{x}}(\alpha), \ \ \ \ \alpha \in A_1 \  ( \ =  A \ ),\\\
\end{split}
\end{equation}
implying that $\Phi_{A,\delta}$ is a \big($l^2(A) \hookrightarrow L^{\infty}(\Omega_{A^{\prime}},\mathbb{P}_{A^{\prime}}) \big)$-interpolant on $R_A$.

By Lemma \ref{Rep}, and \eqref{orthog}, for ${\bf{x}} \in l^2(A),  \ {\bf{y}} \in l^2(A)$,\\
\begin{equation}
\begin{split}
\int_{\Omega_{A^{\prime}}} \Phi_{A,\delta}({\bf{x}}) \ &\Phi_{A,\delta}({\bf{y}}) \ d\mathbb{P}_{A^{\prime}} \ = \ \int_{\Omega_{A^{\prime}}}(-1)^{j-1} \bigg( \sum_{j=1}^{\infty} \  G_{j,\delta}({\bf{x}}^{(j)}) \  \sum_{j=1}^{\infty} \   G_{j,\delta}({\bf{y}}^{(j)}) \bigg) \ d \mathbb{P}_{A^{\prime}}\\\\
&= \ \sum_{j=1}^{\infty} (-1)^j \int_{\Omega_{A^{\prime}}} G_{j,\delta}({\bf{x}}^{(j)}) \  G_{j,\delta}({\bf{y}}^{(j)}) \ d\mathbb{P}_{A^{\prime}} \ = \ {\bf{x}} \boldsymbol{\cdot} {\bf{y}},\\\
\end{split} 
\end{equation}
implying that  $\Phi_{A,\delta}$ is a \big($l^2(A) \hookrightarrow L^{\infty}(\Omega_{A^{\prime}},\mathbb{P}_{A^{\prime}}) \big)$-representation.  The estimate in \eqref{rec13} follows from \eqref{rec12}.
\end{proof}
  \ 
  \begin{corollary} \label{ultraG} \ For all sets $A$ and $\epsilon > 0$, there exist ultra-interpolants 
  \begin{equation}
 G_A \ = \ U_{R_A} + \mathfrak{p}_{A^{\prime}}: \ l^2(A) \ \hookrightarrow \ L^{\infty}(\Omega_{A^{\prime}},\mathbb{P}_{A^{\prime}}),  
  \end{equation}
 such that 
\begin{equation}
\|G_A\|_{l^2 \hookrightarrow L^{\infty}} \leq K, \ \  \ \ \ \|\mathfrak{p}_{A^{\prime}}\|_{l^2 \hookrightarrow L^{2}} \leq \epsilon,
\end{equation}
where $K > 1$ depends only on $\epsilon$, 
$$A^{\prime} \ = \ A  \ \cup \  \bigcup_{j=2}A_j,$$ and $\big(A_j\big)_j$ is the set-cascade generated by $A = A_1$, as per \eqref{rec3}.    
  \end{corollary}
  \begin{proof} \ For $\delta > 0$ (to be determined) and ${\bf{x}} \in \mathcal{S}_{l^2(A)}$, rewrite \eqref{rec11} 
  \begin{equation}
  \begin{split}
  \Phi_{A,\delta}({\bf{x}}) \ &= \  \sum_{j=1}^{\infty} \  \mathfrak{i}^{j-1} \ G_{j,\delta}({\bf{x}}^{(j)})\\\
  &= \ U_{R_{A}}{\bf{x}} \ + \ \bigg(\mathfrak{p}_{1,\delta}({\bf{x}}) \ + \ \sum_{j=2}^{\infty} \  \mathfrak{i}^{j-1} \ G_{j,\delta}({\bf{x}}^{(j)}) \bigg),
  \end{split} 
  \end{equation}
  where $({\bf{x}}^{(j)})_j$ is the vector-cascade given by \eqref{rec1}.  Let
 \begin{equation}
 \mathfrak{p}_{A^{\prime}}({\bf{x}}) \ \overset{def}{=} \ \mathfrak{p}_{1,\delta}({\bf{x}}) \ + \ \sum_{j=2}^{\infty} \  \mathfrak{i}^{j-1} \ G_{j,\delta}({\bf{x}}^{(j)}),
 \end{equation} 
 and use \eqref{rec2},  \eqref{rec1}, and \eqref{rec8} to deduce
  \begin{equation}
  \begin{split}
  \|\mathfrak{p}_{A^{\prime}}({\bf{x}}) \|_{L^2} \ \leq \ \bigg(\delta^2 \ + \ \sum_{j=2}^{\infty} \delta^{2j-2} + \delta^{2j}\bigg)^{\frac{1}{2}} \ \leq \ \delta \sqrt{\frac{2}{1-\delta^2}} \ .
   \end{split}
  \end{equation}
 For $\epsilon > 0$, let 
 \begin{equation}
 \delta \ = \ \delta_{\epsilon} \ = \ \frac{\epsilon}{\sqrt{2+\epsilon^2}} \ ,  \ \ \ \ \ K = \frac{\beta_R(\delta_{\epsilon})}{1-\delta_{\epsilon}} \ ,
 \end{equation} 
 and
 \begin{equation}
 G_A({\bf{x}}) \ = \ \Phi_{A,\delta_{\epsilon}}({\bf{x}}), \ \ \ \ \ {\bf{x}} \in l^2(A).
 \end{equation} 
  \end{proof}
 \begin{remark}[more about constants...] \label{ultra} \ Consider 
\begin{equation} 
\kappa^{\star,u} \ \overset{\text{def}}{=} \ \sup_A \ \inf \big \{  \|\Phi_A\|_{l^2 \hookrightarrow L^{\infty}}: \text{$(l^2 \hookrightarrow L^{\infty})$-ultra-interpolants $\Phi_A$ (on $R_A$)} \big \}.
\end{equation}
(Cf.  \eqref{dualg} and \eqref{Khint01}.) From definitions and Theorem \ref{MT1},
\begin{equation} \label{optimal?}
\mathcal{K}^{\star}  \ \leq \ \kappa^{\star,u}  \ \leq \ \frac{\beta_R(\delta)}{1-\delta}, \ \ \ \ \ \delta \in (0,1). 
\end{equation}
The assertion
\begin{equation} \label{upgrade5}
\kappa^{\star,u}  \ <  \ \infty
\end{equation}
is indeed an upgrade of the Grothendieck inequality as well as the Khintchin inequalty.

 In particular, the Grothendieck inequality is obtained from its upgrade in \eqref{optimal?} via  \eqref{conn2},
\begin{equation} \label{bypass}
\mathcal{K}_G \ \leq \ \big(\mathcal{K}^{\star}\big)^2 \ \leq \ \big(\kappa^{\star,u} \big)^2 \ \leq \ \bigg(\frac{\beta_R(\delta)}{1-\delta}\bigg)^2, \ \ \ \ \ \delta \in (0,1). 
\end{equation}
We can deduce $\mathcal{K}_G < \infty$  directly from Lemma \ref{Rep}, bypassing its upgrade, thereby obtaining an estimate of $\mathcal{K}_G$ sharper than \eqref{bypass}.  Namely,  from Lemma \ref{Rep}, for a set $A$, finite set $B$, scalar array  ${\bf{a}} = (a_{uv})_{(u,v) \in B \times B},$ \ ${\bf{x}}_{u} \in \mathcal{S}_{l^2(A)}$ and  \  ${\bf{y}}_{v} \in \mathcal{S}_{l^2(A)}$, \ $(u,v) \in B \times B$,  
\begin{equation}
\begin{split}
\big|&\sum_{(u,v)  \in  B \times B} a_{uv} {\bf{x}}_{u} \boldsymbol{\cdot} {\bf{y}}_{v} \big|  \\\
&= \ \big| \sum_{(u,v) \in B \times B} a_{uv} \  \sum_{j=1}^{\infty}(-1)^{j-1} \  \int_{\Omega_{A^{\prime}}} G_{j,\delta}({\bf{x}}_u^{(j)}) \  G_{j,\delta}({\bf{y}}_v^{(j)}) \ d\mathbb{P}_{A^{\prime}} \big| \\\
& \leq \    \sum_{j=1}^{\infty}  \ \int_{\Omega_{A^{\prime}}}\big| \sum_{(u,v) \in B \times B} a_{uv} \  G_{j,\delta}({\bf{x}}_u^{(j)}) \  G_{j,\delta}({\bf{y}}_v^{(j)})\big| \ d\mathbb{P}_{A^{\prime}} \\\
 & \leq \  \sum_{j=1}^{\infty}  \ \beta^2 \ \delta^{2(j-1)} \  \|{\bf{a}}\|_{{\vee \atop \otimes}} \ = \ \bigg(\frac{\beta^2}{1-\delta^2} \bigg)\|{\bf{a}}\|_{{\vee \atop \otimes}}, \ \ \ \ \ \ \beta = \beta_R(\delta),
\end{split}
\end{equation} \\
and hence
\begin{equation} \label{optimal}
\mathcal{K}_G \ \leq \ \frac{\big(\beta_R^2(\delta) \big)^2}{1-\delta^2} \ < \ \bigg(\frac{\beta_R(\delta)}{1-\delta}\bigg)^2,\ \ \ \ \ \delta \in (0,1).
\end{equation}\\

Whether $\Lambda(2)$-uniformizability -- of Rademacher or other independent system, e.g., Steinhaus, Gaussian -- could lead to a determination of \ $\mathcal{K}_G$ is open to speculation.  (Cf.  Lemma \ref{trunc2} and Remark \ref{upgrade4}.)
\end{remark}

\begin{remark} [continuity] \ To construct $(l^2 \rightarrow L^2)$-continuous ultra-interpolants, we need $(l^2,L^{\infty})$-perturbations that are  'small'  as well as  $(l^2 \rightarrow L^2)$-continuous.  
For a set $A$ and $\delta > 0$,  let
\begin{equation}
\begin{split}
 \beta^c(A,\delta)  \ \overset{\text{def}}{=} \ \inf \big \{\|U_{R_A} +  \mathfrak{p}_A)\|_{l^2 \hookrightarrow L^{\infty}}: (l^2 \rightarrow L^2)\text{-continuous} \ \mathfrak{p}_A: \mathcal{S}_{l^2(A)} \rightarrow \mathcal{B}_{\delta, L^2_{W_A \setminus R_A}} \big \},\\\ \ \ \ \ \ \ \ \ \ \ 
 \end{split} 
 \end{equation}
where $ \mathcal{B}_{\delta, L^2_{W_A \setminus R_A}}$ is the $\delta$-ball in $L^2_{W_A \setminus R_A}(\Omega_A,\mathbb{P}_A)$, and then
\begin{equation}  \label{unifindex}
\beta^c_{R}(\delta) \  \overset{\text{def}}{=} \ \sup_{A}\beta^c(A,\delta) \ .
\end{equation} 
\begin{proposition} \label{contunif}
\begin{equation}
\kappa \ \leq \ \beta_{R}(\delta) \ \leq  \ \beta^c_{R}(\delta) \  = \mathcal{O}\big(1/\sqrt{\delta}), \ \  \ \ \ 0 <\delta < 1.
\end{equation}
\end{proposition}\

\noindent
(Cf. \eqref{whereas}.)  A proof using Riesz products is outlined in \S \ref{RieszP}.{\bf{i}} below.\\
\end{remark}

\subsection{Riesz products} \label{RieszP} \ For a set $A$ and ${\bf{x}} \in \mathbb{C}^A$,  define
 \begin{equation} \label{Riesz1}
 \mathfrak{R}_A({\bf{x}}) \ = \ \prod_{\alpha \in A} \big(r_0+ {\bf{x}}(\alpha) r_{\alpha} \big)
\end{equation} 
to be the Walsh series
\begin{equation} \label{Riesz2}
\mathfrak{R}_A({\bf{x}}) \ \overset{\text{def}}{=} \ r_0 + \sum_{k=1}^{\infty} \bigg(\sum_{\{\alpha_1,\ldots, \alpha_k\} \subset A} {\bf{x}}(\alpha_1) \cdots  {\bf{x}}(\alpha_k) \ r_{\alpha_1} \cdots r_{\alpha_k}\bigg).
\end{equation}\\  
At the very outset, $\mathfrak{R}_A$ is merely a function on $\mathbb{C}^A$, whose range comprises $W_A$-series formally given by \eqref{Riesz2}.
If $A$ is a finite set, then \eqref{Riesz1} is a bona fide product representing a Walsh polynomial on $\Omega_A$.   But if $A$ is infinite, then the object represented by  \eqref{Riesz2} will depend on ${\bf{x}} \in \mathbb{C}^A$, as well as the mode of convergence of the series in  \eqref{Riesz2}.  The construct in \eqref{Riesz1}, which appeared first in the circle group setting -- with lacunary exponentials in place of Rademacher characters \citep{Riesz:1918} -- is known as a  \emph{Riesz product}. \ \\

 The following general properties of $\mathfrak{R}_A$ are key: \  \\ 
 
 \noindent
 {\bf{1}}. \ $\mathfrak{R}_A$ is a $(\mathbb{C}^A \hookrightarrow \mathfrak{S}_{W_A})$-interpolant on $R_A$, i.e.,
 \begin{equation} \label{interop}
\big( \mathfrak{R}_A({\bf{x}}) \big)^{\wedge}(r_{\alpha}) \ = \ {\bf{x}}(\alpha),\ \ \ \ \ {\bf{x}} \in \mathbb{C}^A, \ \ \alpha \in A,
\end{equation}
 with an orthogonal perturbation
 \begin{equation}
 \mathfrak{p}_A({\bf{x}}) \ = \ r_0 + \sum_{k=2}^{\infty} \bigg(\sum_{\{\alpha_1,\ldots, \alpha_k\} \subset A} {\bf{x}}(\alpha_1) \cdots  {\bf{x}}(\alpha_k) \ r_{\alpha_1} \cdots r_{\alpha_k}\bigg), \ \ \ \ \ {\bf{x}} \in \mathbb{C}^A.
 \end{equation}
 (But for the presence of norms and \eqref{temp0}, $ \mathfrak{R}_A$ satisfies requirements in Remark \ref{temp} with $\Gamma = W_A$, \ $E = R_A$, \ $\mathfrak{X} = \mathbb{C}^A$, \ $\mathfrak{B} = \mathfrak{S}_{W_A}.$)\\

 \noindent
 {\bf{2}}. \ Restrictions of  $\widehat{\mathfrak{R}}_A$ to $W_{A,k}$ are $k$-homogeneous.  Namely,  for scalars $c$ and ${\bf{x}} \in \mathbb{C}^A$, 
\begin{equation}
\mathfrak{R}_A(c {\bf{x}}) \ = \ r_0 \ + \ \sum_{k=1}^{\infty} \bigg(  \sum_{\{\alpha_1,\ldots, \alpha_k\} \subset A} c^{k} \ {\bf{x}}(\alpha_1) \cdots  {\bf{x}}(\alpha_k) \  r_{\alpha_1} \cdots r_{\alpha_k}\bigg),
\end{equation}
and therefore,
\begin{equation} \label{khomog}
\begin{split}
\big(\mathfrak{R}_A(c {\bf{x}}) \big)^{\wedge}(w) \ & = \ c^{k} \ {\bf{x}}(\alpha_1) \cdots  {\bf{x}}(\alpha_k) \ = \ c^k \big(\mathfrak{R}_A({\bf{x}}) \big)^{\wedge}(w),\\\\
& \ \ \ \ \ \ \ \ \ \ \ \ \ \ \ \ \ \ \ \ \ \ \ \ \  w \in W_{A,k}, \ \ w = r_{\alpha_1} \cdots r_{\alpha_k}, \ \ k = 1,2, \ldots \ .\\
\end{split}
\end{equation} \ \\

\noindent 
{\bf{3}}. \ For $p \in (0,\infty)$, 
\begin{equation} \label{extendp}
\begin{split}
 \big \|(\mathfrak{R}_A({\bf{x}})\big)^{\wedge}\big \|_p^p \ & = \  1 + \sum_{k=1}^{\infty} \ \sum_{\{\alpha_1,\ldots, \alpha_k\} \subset A} |{\bf{x}}(\alpha_1) \cdots {\bf{x}}(\alpha_k)|^p, \\\\
&  \leq \ \sum_{k=0}^{\infty}\frac{1}{k!} \bigg(\sum_{\alpha \in A} |{\bf{x}}(\alpha)|^p\bigg)^k\  = \  e^{\|{\bf{x}}\|_p^p}, \ \ \ \ {\bf{x}} \in \mathbb{C}^A.
\end{split}
\end{equation} \ \\

\noindent 
{\bf{4}}. \ For  ${\bf{x}}$ and  ${\bf{y}}$  in $\mathbb{C}^A$, the product ${\bf{x}} {\bf{y}}$ is the usual point-wise product of functions, i.e.,
\begin{equation}
({\bf{x}}{\bf{y}})(\alpha) \ \overset{\text{def}}{=} \  {\bf{x}}(\alpha){\bf{y}}(\alpha), \ \ \ \ \ \alpha \in A. 
\end{equation}\\
For \  ${\bf{a}}$  and \ ${\bf{b}}$  in $\mathbb{C}^{W_A},$ \  
the \emph{convolution} of the Walsh series\\
\begin{equation}
\mathcal{S}_A[{\bf{a}}]  \ = \  \sum_{w \in W_A} {\bf{a}}(w) w, \ \ \ \ \ \mathcal{S}_A[{\bf{b}}]  \ = \  \sum_{w \in W_A} {\bf{b}}(w) w,
\end{equation}\\  
is defined by\\ 
\begin{equation}  \label{convolution}
\mathcal{S}_A[{\bf{a}}] \convolution \mathcal{S}_A[{\bf{b}}] \ \overset{\text{def}}{=} \ \mathcal{S}_A[{\bf{a}}{\bf{b}}] \ = \  \sum_{w \in W_A} {\bf{a}}(w){\bf{b}}(w) w, 
\end{equation}\\
which is consistent with usual convolution in classical spaces.\footnote{\label{footnote2} 
For  $\mu$ and $\nu$ in $M(\Omega_A)$, the convolution $\mu \convolution \nu \in M(\Omega_A)$ is defined via the duality  $C(\Omega_A)^{\star} = M(\Omega_A)$ (Riesz-Kakutani), and satisfies
\begin{equation} \label{con22}
\mu \convolution \nu \ = \ \mathcal{S}_A[\widehat{\mu}] \convolution \mathcal{S}_A[\widehat{\mu}], 
\end{equation}  
where convolution on the right-hand side is defined by \eqref{convolution}.}   Then, from \eqref{Riesz2} and \eqref{convolution},  
\begin{equation} \label{conv22}
\mathfrak{R}_A({\bf{x}}) \convolution \mathfrak{R}_A({\bf{y}}) \ = \ \mathfrak{R}_A({\bf{x}} {\bf{y}}), \ \ \ \ \ {\bf{x}} \in \mathbb{C}^A, \  {\bf{y}} \in \mathbb{C}^A.
\end{equation}\\

\begin{remark} [a broader view] \ Let  $\mathcal{Z}$ be a \emph{unital} normed algebra $\mathcal{Z}$ over the complex scalars $\mathbb{C}$.  Mimicking  \eqref{Riesz2}, we can take  $\mathfrak{R}_A$ to be the function on $\mathcal{Z}^A$, whose values are the $\mathcal{Z}$-valued $W_A$-series 
\begin{equation*} 
\mathfrak{R}_A({\bf{z}}) \ \overset{\text{def}}{=} \ {\bf{z}}_0r_0 + \sum_{k=1}^{\infty} \sum_{\{\alpha_1,\ldots, \alpha_k\} \subset A} \frac{1}{k!} \bigg(\sum_{\pi \in \mathcal{S}_k} {\bf{z}}(\alpha_{\pi(1)}) \cdots {\bf{z}}(\alpha_{\pi(k)}) \bigg)\ r_{\alpha_1} \cdots r_{\alpha_k}, \ \ \ \ \ {\bf{z}} \in \mathcal{Z}^A,
\end{equation*} 
where ${\bf{z}}_0 = $ multiplicative unit of $\mathcal{Z}$, and $\mathcal{S}_k$ is the permutation group of $\{1,\ldots,k\}$.  The interpolation property in \eqref{interop} and the $k$-homogeneity property in \eqref{khomog} carry over verbatim, and the norm estimate in \eqref{extendp} becomes 
\begin{equation}\label{Riesz6}
\sum_{w \in W_A} \|\big(\mathfrak{R}_A({\bf{z}})\big)^{\wedge}(w)\|^p \ \leq \ \exp\big(\sum_{\alpha}\|{\bf{z}}(\alpha)\|^p\big),  \ \ \ \ {\bf{z}} \in \mathcal{Z}^A,
\end{equation}
 where  $\| \boldsymbol{\cdot}\|$ is the $\mathcal{Z}$-norm. For commutative $\mathcal{Z}$, \begin{equation} \label{Riesz5}
\mathfrak{R}_A({\bf{z}}) \ = \ {\bf{z}}_0r_0 + \sum_{k=1}^{\infty} \sum_{\{\alpha_1,\ldots, \alpha_k\} \subset A} {\bf{z}}(\alpha_1) \cdots {\bf{z}}(\alpha_k) \ r_{\alpha_1} \cdots r_{\alpha_k}, \ \ \ \ \ {\bf{z}} \in \mathcal{Z}^A,
\end{equation}
define the products
\begin{equation} \label{Riesz4} 
\mathfrak{R}_A({\bf{z}}) \  {=} \ \prod_{\alpha \in A} \big({\bf{z}}_0r_0 + {\bf{z}}(\alpha) r_{\alpha} \big), \ \ \ \ \ {\bf{z}} \in \mathcal{Z}^A,
\end{equation} 
 in which case also the relation in \eqref{conv22} carries over, where convolution is defined  by \eqref{convolution} with  $\mathcal{Z}$ in place of $\mathbb{C}$.  In this work, $\mathcal{Z}$ will be the real and complex fields, spaces of essentially bounded measurable functions (with point-wise multiplication), and spaces of measures (with convolution). In the latter two cases, $\mathcal{Z} = L^{\infty}$ and $\mathcal{Z} = M$, we will refer to \emph{product amalgams} and \emph{convolution amalgams}, respectively.  
 A recurring task will be to identify objects represented by these products, either in their scalar form \eqref{Riesz2}, or in their more general guise \eqref{Riesz5}.

\end{remark} \ 

 We make use of two classical scenarios that originated in \citet{salem1947lacunary} and \citet{Riesz:1918}. \ \\ 

\noindent
{\bf{i}} \ ($L^{\infty}$-valued Riesz products). \ For ${\bf{x}} \in l^2_{\mathbb{R}}(A)$ (a real-valued Euclidean vector) and $\epsilon > 0$, let\\\
\begin{equation} \label{Riesz3}
\begin{split}
Q_{A,\epsilon}&({\bf{x}}) \ \overset{\text{def}}{=} \ \frak{Im} \ \frac{\|{\bf{x}}\|_2}{\epsilon} \ {\mathfrak{R}_A(\mathfrak{i}\epsilon \boldsymbol{\sigma}{\bf{x}})} \ \ \ \ \ \ \  \text{($\mathfrak{i} = \sqrt{-1}$, \ \ $\frak{Im} = $   \emph{imaginary part})}\\\\
 &= \ \sum_{\alpha \in A}{\bf{x}}(\alpha) r_{\alpha} + \sum_{k=1}^{\infty}(-1)^k \bigg(\frac{\epsilon}{\|{\bf{x}}\|_2}\bigg)^{2k} \bigg(\sum_{\{\alpha_1,\ldots, \alpha_{2k+1}\} \subset A}{\bf{x}}(\alpha_1) \cdots  {\bf{x}}(\alpha_{2k+1}) r_{\alpha_1} \cdots r_{\alpha_{2k+1}}\bigg) \\\\
 &= \  U_{R_A}{\bf{x}} + g_{\epsilon}({\bf{x}}).
 \end{split}
  \end{equation}
  whence
  \begin{equation} \label{interp10}
  \widehat{Q_{A,\epsilon}({\bf{x}})}(r_{\alpha}) \ = \ {\bf{x}}(\alpha), \ \ \ \ \alpha \in A.
  \end{equation}\\
  By a computation similar to  \eqref{extendp},\\
   \begin{equation} \label{estimate4}
   \begin{split}
  \|g_{\epsilon}({\bf{x}})\|_{L^2} \ &\leq \ \bigg(\frac{ \sqrt{\sinh \epsilon^2 -\epsilon^2}}{\epsilon} \bigg)\|{\bf{x}}\|_2 \\\\
  & \leq \ \epsilon^2 \|{\bf{x}}\|_2, \ \ \ \ \ 0 < \epsilon < 1. 
  \end{split} 
  \end{equation}\\
 If $E \subset A$ is finite, then
\begin{equation} \label{es3}
\max_{\xi \in \Omega_E} \big|\big({\mathfrak{R}_E(\mathfrak{i}\epsilon  \boldsymbol{\sigma}{\bf{x}})}\big)(\xi)\big | \ = \  \prod_{\alpha \in E} \bigg(1 + \epsilon^2|{\bf{x}}(\alpha)|^2/\|{\bf{x}}\|_2^2\bigg)^{1/2} \ \leq \ e^{\epsilon^2/2}.
\end{equation}
Therefore (e.g.,  \citet[Lemma 4.3]{blei2014grothendieck}), \ $Q_{A,\epsilon}({\bf{x}}) \in L^{\infty}(\Omega_A,\mathbb{P}_A)$, and 
\begin{equation} \label{estimate3}
  \|Q_{A,\epsilon}({\bf{x}})\|_{L^{\infty}} \leq \big(\frac{e^{\epsilon^2/2}}{\epsilon}\big)\|{\bf{x}}\|_2 \ .  \end{equation}
 For  ${\bf{x}} = {\bf{u}} + \mathfrak{i} {\bf{v}}, \ {\bf{u}} \in l^2_{\mathbb{R}}(A), \ {\bf{v}} \in l^2_{\mathbb{R}}(A)$, let
\begin{equation} \label{imagi}
Q_{A,\epsilon}({\bf{x}}) \ \overset{\text{def}}{=} \ Q_{A,\epsilon}({\bf{u}}) \ + \ \mathfrak{i} Q_{A,\epsilon}({\bf{v}}),
\end{equation}
 and  $$g_{\epsilon}({\bf{x}}) \ \overset{\text{def}}{=}  \ Q_{A,\epsilon}({\bf{x}}) - U_{R_A}{\bf{x}}.$$   Then,  $Q_{A,\epsilon}$ is a $(l^2 \hookrightarrow L^{\infty})$-interpolant with orthogonal $(l^2,L^{\infty})$-perturbation $g_{\epsilon}$. 
 
 \begin{proof}[Proof of Proposition \ref{contunif}] Apply  \eqref{estimate3},  \eqref{estimate4}, and that
 \begin{equation}
 g_{\epsilon}: \mathcal{S}_{l^2(A)} \rightarrow L^2_{W_A \setminus R_A} 
 \end{equation} 
 is $(l^2 \rightarrow L^2)$-continuous (e.g., \citet[Lemma 4.4]{blei2014grothendieck}).\\
 \end{proof}

 \noindent
 {\bf{ii}} \ ($M$-valued Riesz products). \  
 For  ${\bf{x}} \in l^{\infty}_{\mathbb{R}}(A)$ \ and  \ $0 < \epsilon \leq 2$, \  let\\
\begin{equation} \label{RieszM}
\begin{split}
&P_{A,\epsilon}({\bf{x}}) \ 
 \overset{\text{def}}{=} \ \frac{ \|{\bf{x}}\|_{\infty}}{\epsilon}\bigg( \mathfrak{R}_A\big(\frac{\epsilon \boldsymbol{\sigma}_{\infty}{\bf{x}}}{2}\big) \ - \  \mathfrak{R}_A\big(-\frac{\epsilon \boldsymbol{\sigma}_{\infty}{\bf{x}}}{2}\big)\bigg) \ \ \ \ \ \ \ \ \ \ \ \ \ \\\\
 &=  \  \sum_{k=0}^{\infty} \bigg(\frac{\epsilon}{2 \|{\bf{x}}\|_{\infty}}\bigg)^{2k} \bigg(\sum_{\{\alpha_1,\ldots, \alpha_{2k+1}\} \subset A}{\bf{x}}(\alpha_1) \cdots  {\bf{x}}(\alpha_{2k+1}) r_{\alpha_1} \cdots r_{\alpha_{2k+1}}\bigg), \ \ \ \ \ {\bf{x}} \neq {\bf{0}}.
\end{split}
\end{equation}
and write
\begin{equation}
P_{A,\epsilon}({\bf{x}}) \ = \ U_{R_A}{\bf{x}} + h_{\epsilon}({\bf{x}}),
\end{equation}
i.e.,
\begin{equation}
h_{\epsilon}({\bf{x}}) \ = \ \sum_{k=1}^{\infty} \bigg(\frac{\epsilon}{2 \|{\bf{x}}\|_{\infty}}\bigg)^{2k} \bigg(\sum_{\{\alpha_1,\ldots, \alpha_{2k+1}\} \subset A}{\bf{x}}(\alpha_1) \cdots  {\bf{x}}(\alpha_{2k+1}) r_{\alpha_1} \cdots r_{\alpha_{2k+1}}\bigg), \ \ \ \ {\bf{x}} \neq {\bf{0}}.
\end{equation} 
Then,
\begin{equation} \label{interp11}
\widehat{P_{A,\epsilon}({\bf{x}})}(r_{\alpha}) \ = \ {\bf{x}}(\alpha), \ \ \ \ \ \alpha \in A,
\end{equation}
and 
\begin{equation}
 \|\widehat{h_{\epsilon}({\bf{x}})}\|_{\infty} \leq \big(\frac{\epsilon^2}{4}\big) \|{\bf{x}}\|_{\infty}.
 \end{equation}\\ 
Moreover, if $E \subset A$ is finite, and \ $f = \sum_{w \in W_E} \widehat{f}(w)w$, \ then\\
\begin{equation}
\begin{split}
&\bigg| \sum_{w \in W_A}\widehat{P_{A,\epsilon}({\bf{x}})} \widehat{f}(w)\bigg| \ = \ \bigg| \sum_{w \in W_E}\widehat{P_{E,\epsilon}({\bf{x}})} \widehat{f}(w)\bigg| \ = \ \bigg|\int_{\Omega_E} f \ P_{E,\epsilon}({\bf{x}}) \  d \mathbb{P}_E\bigg| \ \leq \ \|f\|_{\infty} \ \|P_{E,\epsilon}({\bf{x}})\|_{L^1}\\\\
& = \ \|f\|_{\infty} \ \frac{ \|{\bf{x}}\|_{\infty}}{\epsilon} \bigg( \int_{\Omega_E}  \mathfrak{R}_E\big(\frac{\epsilon \boldsymbol{\sigma}_{\infty}{\bf{x}}}{2}\big) d \mathbb{P}_E  +   \int_{\Omega_E}  \mathfrak{R}_E\big(- \frac{\epsilon \boldsymbol{\sigma}_{\infty}{\bf{x}}}{2}\big)d \mathbb{P}_E \bigg) \  = \  \|f\|_{\infty} \ \frac{2 \|{\bf{x}}\|_{\infty}}{\epsilon}\\\\
& \ \ \ \ \ \ \ \ \ \ \text{(because $ \mathfrak{R}_E\big(\pm \frac{\epsilon \boldsymbol{\sigma}_{\infty}{\bf{x}}}{2}\big) \geq 0$}).
\end{split}
\end{equation}\\
From the norm density of Walsh polynomials in $C(\Omega_A)$, and the duality $$C(\Omega_A)^{\star}  \ =  \ M(\Omega_A) \ \  \ \text{(Riesz-Kakutani)},$$ we obtain  $P_{A,\epsilon}({\bf{x}}) \in M(\Omega_A)$, and 
 \begin{equation} \label{norm11}
 \big \|P_{A,\epsilon}({\bf{x}})\big \|_M \leq \big(\frac{2}{\epsilon}\big) \|{\bf{x}}\|_{\infty}.\\\
 \end{equation}\\\
  For arbitrary ${\bf{x}} \in l^{\infty}(A)$, define $P_{A,\epsilon}({\bf{x}})$ by applying $P_{A,\epsilon}$ in \eqref{RieszM} separately to the real and imaginary parts of  ${\bf{x}}$. (See \eqref{imagi}.) \ \\
  
  In summary, $P_{A,\epsilon}$  \ is a  $(l^{\infty} \hookrightarrow M)$-interpolant on $R_A$, with  orthogonal   $(l^{\infty},M)$-perturbation $h_{\epsilon}$ of $U_{R_A}$.
    
\begin{remark} [$\mathbb{R}$-homogeneity] \label{homog1} \  Spectra of \ $Q_{A,\epsilon}$ and \  $P_{A,\epsilon}$ comprise Walsh characters of odd order, i.e.,  
\begin{equation}
\widehat{Q_{A,\epsilon}({\bf{x}})}(w)  =  \widehat{P_{A,\epsilon}({\bf{x}})}(w) = 0, \ \ \ \ w \in W_{A,even}, \ \ {\bf{x}} \in l^2(A). 
\end{equation}
(See \eqref{odd}.)  In particular,
\begin{equation}
Q_{A,\epsilon}(-{\bf{x}}) \ = \ - Q_{A,\epsilon}({\bf{x}}),   \ \ \ \ \ 
P_{A,\epsilon}(-{\bf{x}}) \ = \ - P_{A,\epsilon}({\bf{x}}),
\end{equation}\
and therefore,\\
\begin{equation}
Q_{A,\epsilon}(c{\bf{x}})  =  c Q_{A,\epsilon}({\bf{x}})   \ \ \ \text{and} \ \ \  
P_{A,\epsilon}(c{\bf{x}})  =  c P_{A,\epsilon}({\bf{x}}), \ \ \ \   {\bf{x}} \in l^{\infty}(A), \ \ c \in \mathbb{R}.
\end{equation} \\
\end{remark}
  
\subsection{$L^{\infty}_{(s)}$-valued and $M_{(s)}$-valued Riesz products} \label{RieszP0} \  
For $s \in [1,\infty]$, consider 
\begin{equation} \label{xtendp}
\begin{split}
L^{p}_{(s)}(\Omega_A, \mathbb{P}_A) \ &\overset{\text{def}}{=} \  \big \{f \in L^{p}(\Omega_A, \mathbb{P}_A): \widehat{f} \in l^s(W_A) \big \}, \ \ \ \ 1 \leq p \leq \infty,\\\
and  \ \ \ \ \ \ \ \ \ \ \ \ \ \ \ \ \ \ \ \ \ \ \ \ \ \ \ \ \ & \\\ 
M_{(s)}(\Omega_A) \ &\overset{\text{def}}{=} \ \big \{\mu \in M(\Omega_A): \widehat{\mu} \in l^s(W_A) \big \}, 
\end{split}
\end{equation}\\
equipped with 
\begin{equation} \label{quasi}
\begin{split}
\|f\|_{L^{p}_{(s)}} \ & \overset{\text{def}}{=} \ \max\{\|f\|_{L^{p}}, \|\widehat{f}\|_s \}, \ \ \ \ f \in L^{p}_{(s)}(\Omega_A, \mathbb{P}_A), \\\\
\|\mu\|_{M_{(s)}}  \ & \overset{\text{def}}{=} \ \max\{\|\mu\|_M, \|\widehat{\mu}\|_s \}, \ \ \ \ \mu \in M_{(s)}(\Omega_A).
\end{split}
\end{equation}\\
The aforementioned are Banach spaces normed by \eqref{quasi}, where\\  
\begin{equation}
\begin{split}
 L^{\infty}_{(2)}(\Omega_A, \mathbb{P}_A) \ &= \ L^{\infty}(\Omega_A, \mathbb{P}_A),\\\\
L^{\infty}_{(1)}(\Omega_A, \mathbb{P}_A)  \  &= \ M_{(1)}(\Omega_A) \  = \ \mathbb{A}(\Omega_A) \ \ (\ {=} \ \text{absolutely convergent Walsh series}),\\\\
 M_{(2)}(\Omega_A)  \ &= \ L^2(\Omega_A,\mathbb{P}_A), \ \ \ \ \ M_{(\infty)}(\Omega_A)  \ =  \ M(\Omega_A).\\
\end{split}
\end{equation}
Also, by Hausdorff-Young, 
\begin{equation}
L^p_{(q)}(\Omega_A,\mathbb{P}_A) \ = \ L^p(\Omega_A,\mathbb{P}_A), \ \ \ \ 1 \leq p \leq 2, \ \ \ q = \frac{p}{p-1}.
\end{equation}\\
If  \ $1  \leq s  \leq  2   \leq s^{\star}  \overset{\text{def}}{=}   \frac{s}{s-1}  \leq \infty,$ \  $f \in L^{\infty}_{(s)}(\Omega_A, \mathbb{P}_A)$, and  $\mu \in M_{(s^{\star})}(\Omega_A)$, then  the convolution \ $f \convolution \mu$ \ is in $\mathbb{A}(\Omega_A)$, and (therefore)
\begin{equation} \label{Par1}
\int_{\omega \in \Omega_A} f(\omega)  \mu(d\omega) = \sum_{w \in W_A} \widehat{f}(w) \widehat{\mu}(w) \ \ \ \ \ \text{(Parseval formula)}.
\end{equation}
In particular, the integral on the left side of \eqref{Par1} is well-defined, and
\begin{equation}
\bigg|\int_{\omega \in \Omega_A} f(\omega)  \mu(d\omega) \bigg| \ \leq \ \|f\|_{L^{\infty}_{(s)}}\|\mu\|_{M_{(s^{\star})}}.
\end{equation}\\

For $\epsilon > 0 $,\ \ $s \in [1, \infty]$, \ \ ${\bf{x}} \in l^s_{\mathbb{R}}(A), \ \ {\bf{x}} \neq {\bf{0}}$, \ let \ \\
\begin{equation} \label{Riesz32}
Q_{A,\epsilon}^{(s)}({\bf{x}}) \  \overset{\text{def}}{=} \ \frac{\|{\bf{x}}\|_s}{\epsilon}  \ \frak{Im} \  \mathfrak{R}_A\big(\mathfrak{i} \epsilon \boldsymbol{\sigma}_s{\bf{x}}\big), 
\end{equation}
and
\begin{equation}\label{Riesz33}
 P_{A,\epsilon}^{(s)}({\bf{x}}) \  \overset{\text{def}}{=}  \ \frac{\|{\bf{x}}\|_s}{\epsilon} \  \bigg(\mathfrak{R}_A\big(\frac{\epsilon \boldsymbol{\sigma}_{s}{\bf{x}}}{2}\big)  -   \ \mathfrak{R}_A\big(- \frac{\epsilon \boldsymbol{\sigma}_{s}{\bf{x}}}{2}\big)\bigg).
\end{equation}

\begin{lemma}  \label{Key} \ If \  $\epsilon \in (0,2]$,\ \ $s \in [1,\infty]$, \ and \ ${\bf{x}} \in l^s_{\mathbb{R}}(A)$, then:\ \\ 

\noindent
For $s = 1$,
\begin{equation} \label{as1}
\begin{split}
Q_{A,\epsilon}^{(1)}({\bf{x}}) &\in \mathbb{A}(\Omega_A), \ \ \ \ \ \|Q_{A,\epsilon}^{(1)}({\bf{x}})\|_{\mathbb{A}} \ \leq \ \big(\frac{\sinh \epsilon}{\epsilon} \big) \|{\bf{x}}\|_1,\\\\
P_{A,\epsilon}^{(1)}({\bf{x}}) &\in \mathbb{A}(\Omega_A), \ \ \ \ \ \|P_{A,\epsilon}^{(1)}({\bf{x}})\|_{\mathbb{A}} \ \leq \ \big(\frac{2\sinh \frac{\epsilon}{2}}{\epsilon} \big) \|{\bf{x}}\|_1.
\end{split}
\end{equation}\\
For $s \in (1,2]$, 
\begin{equation} \label{as11}
\begin{split}
Q_{A,\epsilon}^{(s)}({\bf{x}}) &\in L^{\infty}_{(s)}(\Omega_A, \mathbb{P}_A), \ \ \ \ \ \|Q_{A,\epsilon}^{(s)}({\bf{x}})\|_{L^{\infty}_{(s)}} \ \leq \ \frac{1}{\epsilon}  \ \exp \bigg(\frac{\epsilon\|{\bf{x}}\|_2}{\sqrt{2}\|{\bf{x}}\|_s}\bigg)^2 \  \|{\bf{x}}\|_s,\\\\
P_{A,\epsilon}^{(s)}({\bf{x}})  & \in  \bigcap_{2 \leq p <  \infty} L^p_{(s)}(\Omega_A,\mathbb{P}_A), \ \ \ \ \|P_{A,\epsilon}^{(s)}({\bf{x}})\|_{L^p} \  \leq \  \frac{2}{\epsilon}  \ \sinh\bigg(\frac{\epsilon \sqrt{p}\|{\bf{x}}\|_2}{2\|{\bf{x}}\|_s}\bigg) \ \|{\bf{x}}\|_s,  \ \  \ p  \geq 2.
\end{split}
\end{equation}\\
For $s \in (2,\infty]$,
\begin{equation} \label{as111}
\begin{split}
P_{A,\epsilon}^{(s)}({\bf{x}}) &\in M_{(s)}(\Omega_A), \ \ \ \ \ \|P_{A,\epsilon}^{(s)}({\bf{x}})\|_{M_{(s)}} \ \leq \  \max \bigg\{\bigg(\frac{2^s\sinh \big(\frac{\epsilon}{2}\big)^s}{\epsilon^s}\bigg)^{1/s}, \ \frac{2}{\epsilon} \bigg\} \|{\bf{x}}\|_s.\\\\
\end{split}
\end{equation}
Also,
\begin{equation} \label{interp3}
\big(Q_{A,\epsilon}^{(s)}({\bf{x}})\big)^{\wedge}(r_{\alpha})  =  \big(P_{A,\epsilon}^{(s)}({\bf{x}})\big)^{\wedge}(r_{\alpha}) \ = \ {\bf{x}}(\alpha), \ \ \ \ \alpha \in A,
\end{equation}\\
\begin{equation} \label{as2}
\begin{split}
\big\|\big(Q_{A,\epsilon}^{(s)}({\bf{x}})\big)^{\wedge}\big|_{W_A \setminus R_A}\big\|_s \ &\leq \ 
\frac{(\sinh \epsilon^s - \epsilon^s)^{1/s}}{\epsilon} \ \|{\bf{x}}\|_s, \ \ \ \ s \in [1,\infty),\\\\
\big\|\big(P_{A,\epsilon}^{(s)}({\bf{x}})\big)^{\wedge}\big|_{W_A \setminus R_A}\big\|_s \ &\leq  \
 \frac{ \big(\sinh \big(\epsilon/2)^s - (\epsilon/2)^s \big)^{1/s}}{\epsilon} \  \|{\bf{x}}\|_s, \ \ \ \ \  s \in [1,\infty]. 
\end{split}
\end{equation}\\\\
Moreover,  the transform maps $$\big(Q_{A,\epsilon}^{(s)}({\boldsymbol{\cdot}})\big)^{\wedge}: \ l^s(A) \ \rightarrow \ l^s(W_A), \ \ \ \ \ \big(P_{A,\epsilon}^{(s)}({\boldsymbol{\cdot}})\big)^{\wedge}: \ l^s(A) \ \rightarrow \ l^s(W_A)$$\\ are $l^s$-norm-continuous.
\end{lemma}
\begin{proof}[Sketch of proof] \ \eqref{as1} and the first line in \eqref{as11}, along with  \eqref{interp3}, \eqref{as2}, and the $l^s$-norm-continuity of the Walsh transforms, follow from spectral analysis and norm estimates applied to Riesz products;  cf. \eqref{extendp}, \cite[Lemmas 4.3, 4.4, and (4.5)]{blei2014grothendieck}.

The second line in \eqref{as11} is proved by applying the Khintchin $(L^2_R \hookrightarrow L^p)$-inequalities (Remark \ref{expsq}) to each of the summands in
\begin{equation} 
P_{A,\epsilon}^{(s)}({\bf{x}}) \ = \ \sum_{k=0}^{\infty} \bigg(\frac{\epsilon}{2 \|{\bf{x}}\|_s}\bigg)^{2k} \bigg(\sum_{\{\alpha_1,\ldots, \alpha_{2k+1}\} \subset A}{\bf{x}}(\alpha_1) \cdots {\bf{x}}(\alpha_{2k+1}) r_{\alpha_1} \cdots r_{\alpha_{2k+1}}\bigg).
\end{equation}
Namely, for $s \in (1,2]$, \  $p > 2$, \ $k = 0, 1, \ldots ,$
\begin{equation} 
\big \| \sum_{_{\{\alpha_1,\ldots, \alpha_{2k+1}\} \subset A}}{\bf{x}}(\alpha_1) \cdots {\bf{x}}(\alpha_{2k+1}) r_{\alpha_1} \cdots r_{\alpha_{2k+1}} \big \|_{L^p} \  \leq \ \frac{(\sqrt{p})^{2k+1}}{(2k+1)!} \ \|{\bf{x}}\|_s^{2k+1},
\end{equation}\\
which we sum over $k$, and then apply Minkowski's inequality.\\
\end{proof}
To define  $Q^{(s)}_{A,\epsilon}({\bf{x}})$ and $P^{(s)}_{A,\epsilon}({\bf{x}})$ for ${\bf{x}} \in l^s(A)$,  apply  \eqref{Riesz32} and \eqref{Riesz33} to the real and imaginary parts of ${\bf{x}}$ separately. (See \eqref{imagi}.) The norm estimates are (at most) "doubled," and all other properties remain intact.  

In the extremal instances \ $\epsilon = 1$, \  $s = 2$, and   $s = \infty$, \ we write 
\begin{equation} \label{ease1}
\begin{split}
Q^{(s)}_{A} \ \ \text{for} \  \  Q^{(s)}_{A,1},  \ \ \ \ \  P^{(s)}_{A} \ \ \text{for} \  \ P^{(s)}_{A,1},&  \ \ \ \ \ Q_{A,\epsilon} \ \ \text{for} \  \ Q^{(2)}_{A,\epsilon},  \ \ \ \ \ P_{A,\epsilon}  \ \ \text{for} \  \  P^{(\infty)}_{A,\epsilon},\\\\
Q_A  \ \ \text{for} \  \  Q_{A,1}, &\ \ \ \text{and} \  \ \ P_A  \ \ \text{for} \  \ P_{A,1}.
\end{split} 
\end{equation}  

From \eqref{conv22}, \eqref{Riesz32} and \eqref{Riesz33}, for $1 \leq  s, \ t  \leq  \infty,$ \  $u \geq \frac{st}{s+t}$, \ ${\bf{x}} \in l_{\mathbb{R}}^s(A)$, \  and \ ${\bf{y}} \in l_{\mathbb{R}}^{t}(A)$,\\
\begin{equation} \label{conv11}
\begin{split}
Q_A^{(s)}({\bf{x}}) \convolution P_A^{(t)}({\bf{y}}) \  = \ P_A^{(s)}({\bf{x}}) &\convolution Q_A^{(t)}({\bf{y}}) \ = \ Q_{A,\epsilon/4}^{(u)}({\bf{x}}  {\bf{y}}),\\\\
P_A^{(s)}({\bf{x}}) \convolution P_A^{(t)}({\bf{y}}) \ & = \ P_{A,\epsilon/2}^{(u)}({\bf{x}}  {\bf{y}}),\\\\
 Q_A^{(s)}({\bf{x}}) \convolution Q_A^{(t)}({\bf{y}}) \  &= \ P_{A,\epsilon}^{(u)}({\bf{x}} {\bf{y}}),\\\\
   \text{where} \ \  \ &\epsilon = \epsilon \big({\bf{x}},{\bf{y}};u,s,t\big) = \frac{\|{\bf{x}} {\bf{y}}\|_u}{\|{\bf{x}}\|_s\|{\bf{y}}\|_t} \ \leq \ 1.\\\\
 \end{split}
\end{equation}

\noindent
We summarize, and record for future use (omitting proof):

\begin{lemma} [Parseval formulae; cf. \eqref{Par1}] \label{key11} \ For \ $1  \leq  s \leq 2 \leq t
  \leq \frac{s}{s-1}$, \ if \ ${\bf{x}} \in l^s(A)$ and  ${\bf{y}} \in l^{t}(A)$, then\\ 
\begin{equation} \label{conv33}
\begin{split}
\int_{\omega \in \Omega_A} Q_A^{(s)}({\bf{x}})(\omega) \  P_A^{(t)}({\bf{y}})(d\omega) \ &= \ \sum_{\alpha \in A} {\bf{x}}(\alpha){\bf{y}}(\alpha) \ + \ \sum_{w \in W_A \setminus R_A} \big(Q_A^{(s)}({\bf{x}})\big)^{\wedge}(w)\big(P_A^{(t)}({\bf{y}})\big)^{\wedge}(w),\\\\
\int_{\omega \in \Omega_A}P_A^{(s)}({\bf{x}})(\omega) \  P_A^{(t)}({\bf{y}})(d\omega) \ &= \ \sum_{\alpha \in A} {\bf{x}}(\alpha){\bf{y}}(\alpha) \ + \ \sum_{w \in W_A \setminus R_A} \big(P_A^{(s)}({\bf{x}})\big)^{\wedge}(w)\big(P_A^{(t)}({\bf{y}})\big)^{\wedge}(w),\\\\
\int_{\omega \in \Omega_A} Q_A^{(s)}({\bf{x}})(\omega) \  Q_A^{(t)}({\bf{y}})(d\omega) \ &= \ \sum_{\alpha \in A} {\bf{x}}(\alpha){\bf{y}}(\alpha) \ + \ \sum_{w \in W_A \setminus R_A} \big(Q_A^{(s)}({\bf{x}})\big)^{\wedge}(w)\big(Q_A^{(t)}({\bf{y}})\big)^{\wedge}(w),
\end{split}
\end{equation}\\
where all sums on the right are absolutely convergent.  
\end{lemma}

\begin{remark} [bona fide integrals?] \label{fide} \ 
The left sides of the first two lines of \eqref{conv33} are well-defined Lebesgue integrals. In the case $s = t = 2$, the left side of the third line is a bona fide integral with respect to $$Q_A^{(2)}({\bf{y}})(d\omega) = Q_A^{(2)}({\bf{y}})(\omega) \ \mathbb{P}_A(d\omega),$$ but otherwise, for $1 \leq s < 2 < t \leq \infty$, the left side is defined by the sums on the right. 

\noindent
\begin{question}  What can be said about  $Q_A^{(t)}({\bf{y}})$ for ${\bf{y}} \in l^t_{\mathbb{R}}(A)$ and $t \in (2,\infty]?$ E.g., does $Q_A^{(t)}({\bf{y}})$ represent a measure?
\end{question} 
\end{remark}

\subsection{More interpolants} \label{more}   \ For $A^{\prime} \supset A$, let
 \begin{equation*} 
\mathfrak{P}^{(s)}(A^{\prime} \supset A) \ \overset{\text{def}}{= } \   \left \{
\begin{array}{lcc}
\big \{\text{orthogonal} \  \big(l^2(A),L^{\infty}_{(s)}(\Omega_{A^{\prime}},\mathbb{P}_{A^{\prime}})\big)\text{-perturbations of} \ U_{R_A}\big\}, & s \in [1,2],  \\\\
 \big \{\text{orthogonal} \  \big(l^2(A),M_{(s)}(\Omega_{A^{\prime}})\big)\text{-perturbations of} \ U_{R_A}\big\}, & s \in (2,\infty].  
\end{array} \right.
\end{equation*}
(See Remark \ref{temp}.) We write $\mathfrak{P}^{(s)}(A)$ for $\mathfrak{P}^{(s)}(A \supset A)$.  (See Remark \ref{sup}.)  Let \\
\begin{equation} \label{Khint02}
\begin{split}
\kappa^{\star,(s)} \ &\overset{\text{def}}{=} \ \sup_A \ \inf \big \{ \|U_{R_A} + \mathfrak{p}_{A}\|_{l^s \hookrightarrow L^{\infty}_{(s)}}: \mathfrak{p}_{A} \in \mathfrak{P}^{(s)}(A) \big \}, \ \ \ \ s \in [1,2],\\\ 
\sigma^{(s)} \ &\overset{\text{def}}{=} \ \sup_A \ \inf \big \{ \|U_{R_A} + \mathfrak{p}_{A}\|_{l^s \hookrightarrow M_{(s)}}: \mathfrak{p}_{A} \in \mathfrak{P}^{(s)}(A) \big \},\ \ \ \ s \in (2,\infty].
\end{split}
\end{equation} 
For $\delta \geq 0$, let 
\begin{equation*} 
 \beta^{(s)}(A,\delta)  \ \overset{\text{def}}{=} \   \left \{
\begin{array}{lcc}
\inf \big \{\|U_{R_A} + \mathfrak{p}\|_{l^s \hookrightarrow L^{\infty}_{(s)}}: \mathfrak{p} \in \mathfrak{P}^{(s)}(A), \ \|\widehat{\mathfrak{p}}\|_{l^s \hookrightarrow l^s}  \leq 
 \delta \big\}, & s \in [1,2],  \\\\
 \inf \big \{\|U_{R_A} + \mathfrak{p}\|_{l^s \hookrightarrow M_{(s)}}: \mathfrak{p} \in \mathfrak{P}^{(s)}(A), \ \|\widehat{\mathfrak{p}}\|_{l^s \hookrightarrow l^s}  \leq 
 \delta \big\}, & s \in (2,\infty],  
\end{array} \right.
\end{equation*}
and then\\
\begin{equation} \label{mod2} 
 \beta_R^{(s)}(\delta) \ \overset{\text{def}}{=}  \ \sup_{A} \beta^{(s)}(A,\delta), \ \ \ \ \  s \in [1,\infty].
\end{equation}\\
(Cf. \eqref{uniform6}.) Adding $\big(l^s(A) \rightarrow l^s(W_{A^{\prime}})\big)$-continuity, we let 
\begin{equation}
\begin{split}
\mathfrak{P}^{(s),c}(A^{\prime} \supset A) \ \overset{\text{def}}{=} \ \big\{\big(\ l^s(A) \rightarrow  l^s(W_{A^{\prime}} \setminus  R_A)\big) \text{-continuous} \  \mathfrak{p}_{A^{\prime}} \in \mathfrak{P}^{(s)}(A^{\prime} \supset A)  \big\}, \ \ \ \ \ \ \ 
\end{split}
\end{equation}
and for $\delta \geq 0$,\\
\begin{equation} \label{mod4} 
\begin{split}
 \beta^{(s),c}(A,\delta)&  \\\
 &\overset{\text{def}}{=} \ \inf \big \{\|U_{R_A} +  \mathfrak{p}_{A^{\prime}})\|_{l^s \hookrightarrow L^{\infty}_{(s)}}: \mathfrak{p}_{A^{\prime}} \in \ \mathfrak{P}^{(s),c}(A^{\prime} \supset A), \  \|\widehat{\mathfrak{p}_{A}}\|_{l^s \hookrightarrow l^s}  \leq 
 \delta \big \}, \ \ \ s \in [1,2],\\\
  & \overset{\text{def}}{=} \ \inf \big \{\|U_{R_A} +  \mathfrak{p}_{A^{\prime}})\|_{l^s \hookrightarrow M_{(s)}}: \mathfrak{p}_{A^{\prime}} \in \ \mathfrak{P}^{(s),c}(A^{\prime} \supset A), \  \|\widehat{\mathfrak{p}_{A}}\|_{l^s \hookrightarrow l^s}  \leq 
 \delta \big \}, \ \ \ s \in (2,\infty],\\\\
&\ \ \ \ \ \ \ \ \ \  \beta_R^{(s),c}(\delta) \ \overset{\text{def}}{=}  \  \sup_{A} \beta^{(s),c}(A,\delta), \ \ \ \ \  s \in [1,\infty].
\end{split}
\end{equation}
Lemma \ref{Key} implies
\begin{corollary} [cf. Corollary \ref{trunc3}, Proposition \ref{contunif}] \label{Key7}  \ For every $s \in (1,\infty],$ there exist $K_s > 0$ such that for $\delta \in (0,1),$
\begin{equation}  \label{Key0}
 K_s/\sqrt{\delta} \  \geq \ \beta_R^{(s),c}(\delta) \ \geq \ \beta_R^{(s)}(\delta) \  \geq \left \{
\begin{array}{lcc}
\kappa^{(s)}, & s \in [1,2],  \\\\
\sigma^{(s)}, & s \in (2,\infty].  
\end{array} \right.
\end{equation}
\end{corollary}
\begin{remark} [$s \geq1$] \label{extremal} \  The extremal case $s=1$ is trivial in the sense that
\begin{equation} 
\beta_R^{(1),c}(\delta) \ = \beta_R^{(1)}(\delta) \ = \ 1, \ \ \ \ \delta \geq 0.
\end{equation}
 For $s > 1$, the proof of  \eqref{Key0} uses Riesz products. (Proposition \ref{contunif} is the case $s = 2$.) I do not know whether the estimate is optimal, or whether $\beta_R^{(s),c}(\delta) \ \geq \ \beta_R^{(s)}(\delta)$ is a strict inequality.   
\end{remark}  

\subsection{Parseval-type formulae:  extensions of the Grothendieck inequality} \label{scalarparseval} \      The \emph{Sidon} property of Rademacher systems,  
\begin{equation} \label{SidR}
\big\|U_{R_A}{\bf{x}}\big\|_{L^{\infty}}  \ = \  \big\|\sum_{\alpha \in A} {\bf{x}}(\alpha) r_{\alpha} \big\|_{L^{\infty}} \ \geq \ \frac{2}{\pi}\|{\bf{x}}\|_1, \ \ \ \ {\bf{x}} \in l^1(A),
\end{equation}
with its dual equivalent, that for every ${\bf{y}} \in l^{\infty}(A)$ there exist $\mu({\bf{y}}) \in M(\Omega_A)$ such that
\begin{equation} \label{SidRD}
\widehat{\mu({\bf{y}})}(r_{\alpha}) \ = \ {\bf{y}}(\alpha), \ \ \ \ \ \alpha \in A,
\end{equation} 
implies 
\begin{equation} \label{usual}
\sum_{\alpha \in A} {\bf{x}}(\alpha) {\bf{y}}(\alpha) \ = \ \int_{\xi \in \Omega_A} \big(U_{R_A}{\bf{x}}\big)(\xi) \ \mu({\bf{y}})(d\xi), \ \ \ \ \ {\bf{x}} \in l^1(A), \ {\bf{y}} \in l^{\infty}(A).
\end{equation}\\  
Based on extensions of the Sidon property, similar expressions of the duality of $l^s$ and $l^{s^{\star}}  \  (1 < s  \leq 2, \ s^{\star} = \frac{s}{s-1} $) were derived in  \citet[Theorem 9.1]{blei2014grothendieck}.
These formulae are recalled below with derivations that slightly differ from the derivations in \citet{blei2014grothendieck}, and closely mimic the proof of Theorem \ref{MT1}. (See Remark \ref{deconstruct}.)  \ \\

Given a set $A$, we initialize $A_1  = A,$ and generate the \emph{set-cascade} 
\begin{equation} \label{cascade1}
A_{j+1} \ = \  W_{A_{j,odd}} \setminus R_{A_j}, \  \ \ j \geq 1.
\end{equation} 
(Cf. \eqref{rec3}.) As in \eqref{disjoint}, \eqref{disjoint1}, and \eqref{disjoint2},  let
\begin{equation} \label{disjoint3}
A^{\prime} \ = \ \bigcup_{j=1}^{\infty}A_j, \ \ \ \ \Omega_{A^{\prime}} \ {=} \ \bigtimes_{j=1}^{\infty} \Omega_{A_j}, \ \ \ \ \mathbb{P}_{A^{\prime}} \ {=} \ \bigtimes_{j=1}^{\infty} \mathbb{P}_{A_j}.
\end{equation} 
 Given ${\bf{x}} \in l^s_{\mathbb{R}}(A)$, we initialize ${\bf{x}}^{(1)}  =  {\bf{x}},$
and generate the \emph{vector-cascade}  ${\bf{x}}^{(j)} \in l^s(A_{j})$, $j \geq 1$,   by \ \\
\begin{equation} \label{recur1}
{\bf{x}}^{(j+1)}(\alpha) = \left \{
\begin{array}{lcc}
\big(Q_{A_j}^{(s)}({\bf{x}}^{(j)})\big)^{\wedge}(\alpha), &\alpha \in A_{j+1}, \  \ s \in [1,2],  \\\\
\big(P_{A_j}^{(s)}({\bf{x}}^{(j)})\big)^{\wedge}(\alpha),&\alpha \in A_{j+1}, \  \  s \in (2,\infty],
\end{array} \right.
\end{equation}\\
where  $Q_{A_j}^{(s)}({\bf{x}}^{(j)})$ and $P_{A_j}^{(s)}({\bf{x}}^{(j)})$  are defined on $\Omega_{A^{\prime}}$. (Cf. \eqref{view}.)  By \eqref{as2} with $\epsilon = 1$,
\begin{equation} \label{iter1}
\|{\bf{x}}^{(j+1)}\|_s \ \leq \ \delta_s \|{\bf{x}}^{(j)}\|_s, \ \ \ \ j = 1,\ldots,
\end{equation}
where\\
\begin{equation} \label{dels}
\delta_s =  \left \{
\begin{array}{lcc}
 (\sinh 1 - 1)^{1/s}, &\quad \  \ s \in [1,2],  \\\\
 \big(\sinh \frac{1}{2^s} -  \frac{1}{2^s} \big)^{1/s}, &\quad \  \  s \in (2,\infty], 
\end{array} \right.
\end{equation}\\
and therefore,
\begin{equation}
\|{\bf{x}}^{(j)}\|_s \ \leq \ \delta_s^{j-1} \|{\bf{x}}\|_s, \ \ \ \ j = 1, \ldots \ .
\end{equation}
(Note:  $0 < \delta_s < 1$.)\ \\\

For $s \in [1,\infty]$ \ and ${\bf{x}} \in l^s_{\mathbb{R}}(A),$ \ define \ \\
\begin{equation} \label{Key1}
\Phi_A^{(s)}({\bf{x}}) \ = \  \left \{
\begin{array}{lcc}
 U_{R_{A_1}}{\bf{x}}, & \ \quad s = 1, \\\\
 \sum_{j=1}^{\infty}  \ \mathfrak{i}^{j-1} \ Q_{A_j}^{(s)}({\bf{x}}^{(j)}), & \ \quad  s \in [1,2],  \\\\
  \sum_{j=1}^{\infty}  \ \mathfrak{i}^{j-1} \ P_{A_j}^{(s)}({\bf{x}}^{(j)}), & \  \quad  s \in (2,\infty],\\\\
  P_{A_1}({\bf{x}}), & \ \quad s = \infty.
\end{array} \right.
\end{equation}\\\
For ${\bf{x}} = {\bf{u}} + \mathfrak{i}{\bf{v}}, \ {\bf{u}} \in l^s_{\mathbb{R}}(A), \ {\bf{v}} \in l^s_{\mathbb{R}}(A),$ let \\
\begin{equation} \label{complex1}
\Phi_A^{(s)}({\bf{x}}) \ \overset{\text{def}}{=} \ \Phi_A^{(s)}({\bf{u}}) + \mathfrak{i} \Phi_A^{(s)}({\bf{v}}).
\end{equation}\\

For $s = 1$ and $s = \infty,$   
\begin{equation} \label{trivial}
\begin{split}
\frac{2}{\pi}\|{\bf{x}}\|_{1} \ \leq \ \|\Phi_A^{(1)}({\bf{x}})\|_{L^{\infty}} \ &\leq \ \|{\bf{x}}\|_{1}, \ \ \ \ \ \ \ {\bf{x}} \in l^1(A),\\\\
\|{\bf{x}}\|_{\infty} \ \leq \ \|\Phi_A^{(\infty)}({\bf{x}})\|_M \ &\leq \ 4\|{\bf{x}}\|_{\infty}, \ \ \ \ \ {\bf{x}} \in l^{\infty}(A),
\end{split}
\end{equation}
\begin{equation} \label{trivial0}
\Phi_A^{(1)}({\bf{x}}) \convolution \Phi_A^{(\infty)}({\bf{y}}) \ = \ \Phi_A^{(1)}({\bf{x}}{\bf{y}}) \ \in \mathbb{A}(\Omega_{A^{\prime}}),\ \ \ \ \  {\bf{x}} \in l^{1}_{\mathbb{R}}(A), \ \ {\bf{y}} \in l^{\infty}_{\mathbb{R}}(A),
\end{equation}\\
and
\begin{equation} \label{trivial1}
\begin{split}
\sum_{\alpha \in A} {\bf{x}}(\alpha){\bf{y}}(\alpha) \ &= \ \int_{\omega \in \Omega_{A^{\prime}}}\big(\Phi_A^{(1)}({\bf{x}})\big)(\omega) \ \big(\Phi_A^{(\infty)}({\bf{y}})\big)(d\omega)\\\\
 &=  \ \big(\Phi_A^{(1)}({\bf{x}}) \convolution \Phi_A^{(\infty)}({\bf{y}})\big)(\boldsymbol{\omega}_0), \ \ \ \ \  {\bf{x}} \in l^{1}(A), \ \ {\bf{y}} \in l^{\infty}(A),\\\\
\end{split}
\end{equation}\\
where $\boldsymbol{\omega}_0 = \text{identity element of} \ \Omega_{A^{\prime}}$.  In \eqref{trivial}, the first line is trivial, and the second follows from \eqref{norm11}. Parseval's formula and \eqref{interp11} imply \eqref{trivial0} and \eqref{trivial1}, as per \eqref{usual}.  For $s \in (1, \infty)$, we have 
\begin{lemma} \label{lem11}
\ For \ ${\bf{x}} \in l^s_{\mathbb{R}}(A)$, \ $1 < s < \infty$, \\ 
\begin{equation} \label{geom}
\Phi_A^{(s)}({\bf{x}}) \ \in \   \left \{
\begin{array}{lcc}
  L^{\infty}_{(s)}(\Omega_{A^{\prime}}, \mathbb{P}_{A^{\prime}}), &  \quad 1 < s \leq 2, \\\\
  M_{(s)}(\Omega_{A^{\prime}}), & \quad 2 < s  < \infty,
\end{array} \right.
\end{equation}\\
and (then) for \ ${\bf{y}} \in l^{s^{\star}}_{\mathbb{R}}(A),$  $1 < s \leq 2 \leq s^{\star} = \frac{s}{s-1} < \infty,$ \\ 
\begin{equation} \label{disj}
\Phi_A^{(s)}({\bf{x}}) \convolution \Phi_A^{(s^{\star})}({\bf{y}}) \  =  \ \left \{
\begin{array}{lcc}
   \sum_{j=1}^{\infty} (-1)^{j-1} Q_{A_j,\epsilon_j/4}^{(1)}\big(({\bf{x}}  {\bf{y}})^{(j)}\big), & \ \quad 1 < s < 2, \\\\
 \sum_{j=1}^{\infty} (-1)^{j-1} P_{A_j,\epsilon_j}^{(1)}\big(({\bf{x}}  {\bf{y}})^{(j)}\big), & \ \quad s = 2,
\end{array} \right.
\end{equation}\\\
are elements of $\mathbb{A}(\Omega_{A^{\prime}})$, where $\epsilon_j = \epsilon({\bf{x}}^{(j)},{\bf{y}}^{(j)};1, s,s^{\star})$ are defined by  \eqref{conv11}.
Moreover, 
\begin{equation} \label{est22}
\big\|\Phi_A^{(s)}({\bf{x}}) \convolution \Phi_A^{(s^{\star})}({\bf{y}}) \big\|_{\mathbb{A}(\Omega_{A^{\prime}})} \ \leq \ K(\delta_s,\delta_{s^{\star}}) \ \|{\bf{x}}\|_s \|{\bf{y}}\|_{s^{\star}},
\end{equation}
where $\delta_s$ and $\delta_{s^{\star}}$ are given in \eqref{dels}, and $K(\delta_s,\delta_{s^{\star}})$  depends only on $\delta_s$ and $\delta_{s^{\star}}$.
\end{lemma}

\begin{proof}[Sketch of proof]
The series in \eqref{Key1} converge absolutely in $L^{\infty}_{(s)}(\Omega_{A^{\prime}}, \mathbb{P}_{A^{\prime}})$ \ and  \ $M_{(s^{\star})}(\Omega_{A^{\prime}})$, respectively, via \eqref{iter1} and a  geometric series argument; and hence \eqref{geom}.  The convolutions in \eqref{disj} are computed by applying \eqref{conv11} and independence  of Rademacher systems $R_{A_j}$.  The norm estimate in \eqref{est22} follows from \eqref{iter1}.\\  
\end{proof}
\begin{lemma} \label{lem22} \ For ${\bf{x}} \in l^s_{\mathbb{R}}(A)$,  ${\bf{y}} \in l^{s^{\star}}_{\mathbb{R}}(A),$  $1 < s \leq 2  \leq s^{\star} = \frac{s}{s-1} < \infty,$    
\begin{equation} \label{key22}
\begin{split}
\sum_{\alpha \in A} {\bf{x}}(\alpha){\bf{y}}(\alpha) \ &= \ \big(\Phi_A^{(s)}({\bf{x}}) \convolution \Phi_A^{(s^{\star})}({\bf{y}})\big)(\boldsymbol{\omega}_{0}) \ \\\
& = \ \int_{\omega \in \Omega_{A^{\prime}}}\big(\Phi_A^{(s)}({\bf{x}})\big)(\omega) \ \big(\Phi_A^{(s^{\star})}({\bf{y}})\big)(d\omega).
\end{split}
\end{equation}
(For $s = 2$,  \ $\big(\Phi_A^{(2)}({\bf{y}})\big)(d\omega)  =  \big(\Phi_A^{(2)}({\bf{y}})\big)(\omega) \ \mathbb{P}_{A^{\prime}}(d\omega).$)
\end{lemma}
\begin{proof}[Sketch of proof] \ For $1 < s < 2$, \  by iterating \eqref{recur1} and Lemma \ref{key11}, we obtain (by induction)\\
\begin{equation} \label{key33}
\begin{split}
&\bigg(\sum_{j=1}^{n}  \mathfrak{i}^{j-1}  Q_{A_j}^{(s)}({\bf{x}}^{(j)}) \convolution \sum_{j=1}^{n}  \mathfrak{i}^{j-1} P_{A_j}^{(s^{\star})}({\bf{y}}^{(j)})\bigg)(\boldsymbol{\omega}_{0}) \\\\ 
&  =  \ \sum_{j=1}^{n}  (-1)^{j-1}  \int_{\Omega_{A^{\prime}}}Q_{A_j}^{(s)}({\bf{x}}^{(j)})  P_{A_j}^{(s^{\star})}({\bf{y}}^{(j)})(d\omega)\\\\ 
& = \  \sum_{\alpha \in A} {\bf{x}}(\alpha){\bf{y}}(\alpha) \  + \  (-1)^{n-1}\sum_{w \in W_{A_n}\setminus R_{A_n}}Q_{A_n,\epsilon_n/4}^{(1)}\big(({\bf{x}}  {\bf{y}})^{(n)}\big)^{\wedge}(w),\ \ \ \ \ \ \  n = 1, \ldots \  \ .\\
\end{split}
\end{equation} 
(See \eqref{conv11}.)  Then, 
\begin{equation*}
\bigg(\sum_{j=1}^{n}  \mathfrak{i}^{j-1}  Q_{A_j}^{(s)}({\bf{x}}^{(j)})\convolution \sum_{j=1}^{n}  \mathfrak{i}^{j-1} P_{A_j}^{(s^{\star})}({\bf{y}}^{(j)})\bigg)(\boldsymbol{\omega}_{0}) \ \ \underset{n \to \infty}{\longrightarrow}\ \ \big(\Phi_A^{(s)}({\bf{x}}) \convolution \Phi_A^{(s^{\star})}({\bf{y}})\big)(\boldsymbol{\omega}_{0})
\end{equation*}\\\
and
\begin{equation*}
\sum_{w \in W_{A_n}\setminus R_{A_n}}Q_{A_n,\epsilon_n/4}^{(1)}\big(({\bf{x}}  {\bf{y}})^{(n)}\big)^{\wedge}(w) \ \ \underset{n \to \infty}{\longrightarrow} \ \ 0
\end{equation*}\\ 
converges uniformly on bounded subsets  of $l^s(A)$ and $l^{s^{\star}}(A)$, and hence \eqref{key22}. For $s = 2$, in the preceding argument replace $P_{A_j}^{(2)}({\bf{y}}^{(j)})$ with $Q_{A_j}^{(2)}({\bf{y}}^{(j)})$.   
\end{proof}
\noindent
We summarize:
\begin{theorem} \label{mainTh} \ The mappings in \eqref{complex1} are injections
\begin{equation} \label{injs}
\Phi_A^{(s)}: \ l^s(A) \ \rightarrow  \left \{
\begin{array}{lcc}
 L^{\infty}_{(s)}(\Omega_{A^{\prime}}, \mathbb{P}_{A^{\prime}}) , & \quad  s \in [1,2]  \\\\
 M_{(s)}(\Omega_{A^{\prime}}),& \quad  s \in (2,\infty],
\end{array} \right.
\end{equation}
with the following properties:\\

\noindent
{\bf{i}}. \ For ${\bf{x}} \in l^s(A),$ $s \in [1,\infty]$,\\ 
\begin{equation} \label{inters}
\big(\Phi_A^{(s)}({\bf{x}})\big)^{\wedge}(r_{\alpha}) \ = \ {\bf{x}}(\alpha), \ \ \ \ \ \alpha \in A.
\end{equation} \ \\

\noindent
{\bf{ii}}. \ For ${\bf{x}} \in l^s(A),$ $s \in [1,\infty]$,\\ 
\begin{equation} \label{est1}
\begin{split}
\frac{2}{\pi}\|{\bf{x}}\|_1 \ \leq \ \|\Phi_A^{(1)}({\bf{x}})\|_{L^{\infty}} \ & \leq \ \|{\bf{x}}\|_1, \ \ \ \  s =1,\\\\
 \|{\bf{x}}\|_s \ \leq \ \|\Phi_A^{(s)}({\bf{x}})\|_{L^{\infty}_{(s)}} \ & \leq \ \frac{ 2\sqrt{e}}{1 - \delta_s} \ \|{\bf{x}}\|_s, \ \ \ \  s \in (1,2],\\\\
 \|{\bf{x}}\|_s \ \leq \ \|\Phi_A^{(s)}({\bf{x}})\|_{M_{(s)}} \ &\leq \ \frac{4}{1 - \delta_s} \ \|{\bf{x}}\|_s, \ \ \ \ s \in (2,\infty],\\\\
 \|{\bf{x}}\|_{\infty} \ \leq \ \|\Phi_A^{(\infty)}({\bf{x}})\|_{M} \ &\leq  \ 4\|{\bf{x}}\|_{\infty}, \ \ \ \ s = \infty,\\\
\end{split}
\end{equation}
where $\delta_s$ is given in \eqref{dels}. \ \\

\noindent
{\bf{iii}}. \ For  \ ${\bf{x}} \in l^s(A)$,  \ ${\bf{y}} \in l^{s^{\star}}(A),$ \ $1 \leq s \leq 2 \leq s^{\star} = \frac{s}{s-1} \leq \infty,$ \\
 \begin{equation} \label{Par}
 \sum_{\alpha \in A} {\bf{x}}(\alpha){\bf{y}}(\alpha) \ = \ \int_{\omega \in \Omega_{A^{\prime}}}\big(\Phi_A^{(s)}({\bf{x}})\big)(\omega) \ \big(\Phi_A^{(s^{\star})}({\bf{y}})\big)(d\omega).
  \end{equation}\\
 {\bf{iv}}. \ For  ${\bf{x}} \in l^s(A), \ \ s  \in [1,\infty],$ \ $c \in \mathbb{R},$\\
  \begin{equation} \label{homog}
  \Phi_A^{(s)}(c{\bf{x}}) \ = \ c  \ \Phi_A^{(s)}(c{\bf{x}}).\\\\
  \end{equation}\\
{\bf{v}}. \ The transform maps\\  
 \begin{equation} \label{cont11}
 \widehat{\Phi_A^{(s)}}: \ l^s(A) \ \rightarrow \ l^s(W_{A^{\prime}}), \ \ \ s \in [1,\infty],\\\\
 \end{equation}
 are $(l^s \hookrightarrow l^s)$-continuous.
\end{theorem}
\begin{proof}[Sketch of proof] \ The interpolation property in \eqref{inters} follows from
\begin{equation} 
\big(\Phi_A^{(s)}({\bf{x}})\big)^{\wedge}(r_{\alpha}) = \left \{
\begin{array}{lcc}
\big(Q_{A_1}^{(s)}({\bf{x}}^{(1)})\big)^{\wedge}(\alpha), &\alpha \in A_{1}, \  \ s \in [1,2]  \\\\
\big(P_{A_1}^{(s)}({\bf{x}}^{(1)})\big)^{\wedge}(\alpha),&\alpha \in A_{1}, \  \  s \in (2,\infty].
\end{array} \right.
\end{equation}\\
The first line in \eqref{est1} is the Sidon property in \eqref{SidR}, and the fourth follows from \eqref{norm11}. In \eqref{Par}, the instance $s = 1$ is a restatement of \eqref{trivial1}, and the case $1 < s \leq 2$ is a restatement of \eqref{key22} in  Lemma \ref{lem22}.  The second and third lines in \eqref{est1} follow from   \eqref{iter1}.  The homogeneity property in \eqref{homog} follows from the homogeneity of  $Q_{A_j}^{(s)}$ and $P_{A_j}^{(s)}$. (See Remark \ref{homog1}.)  The continuity of the transform map in \eqref{cont11}  follows from the last assertion in Lemma \ref{Key}.
\end{proof}

\begin{remark}[representations of $l^s$] \
For $s=2$ in Theorem \ref{mainTh}, the map $\Phi_A^{(2)}$  is the $\big(l^2(A) \hookrightarrow L^{\infty})$-ultra-interpolant of Theorem \ref{MT1}
\begin{equation} \label{template}
  \Phi_{A,\delta}({\bf{x}}) \ {=} \  \sum_{j=1}^{\infty} \  \mathfrak{i}^{j-1} \ G_{j,\delta}({\bf{x}}^{(j)}), \ \ \ \ \ {\bf{x}} \in l^2(A),
\end{equation} 
with $G_{j,\delta}  =  Q_{A_j}$ and $\delta  = \sqrt{\sinh 1 - 1}$.  
We refer to the maps in \eqref{injs} that satisfy  \eqref{inters}, \eqref{est1} (with general constants on the right side), and \eqref{Par}, as representations of \ $l^s(A)$; specifically,  $(l^s(A) \hookrightarrow L^{\infty}_{(s)})$-representations for $s \in [1,2]$, and $(l^s(A) \hookrightarrow M_{(s)})$-representations for $s \in (2,\infty]$.  The gist of Theorem \ref{mainTh}  is that the constructs in  \eqref{complex1} are uniformly bounded $\mathbb{R}$-homogeneous representations of $l^s(A)$ with norm-continuous transforms.

  \end{remark}
  \begin{remark}[idea deconstructed] \label{deconstruct}\ The case $\#A = 2$, which we excluded, is trivial.  For, if $\#A = \#A_1 = 2$, then $A_j = \emptyset$ for $j > 2.$
  
 In \citet[Theorem 9.1]{blei2014grothendieck} we took $A$ to be infinite, without loss of generality.  We fixed a countably infinite partition $\{A_j: j \in \mathbb{N}\}$ of $A$ with the property that each $A_j$ had the same cardinality as $A$, and chose bijections
\begin{equation} \label{bijection}
\tau_0: \ A \rightarrow A_1, \ \  \ \ \ \tau_j: \ W_{A_j} \setminus R_{A_j} \ \rightarrow \ A_{j+1}, \ \ \ \ j \in \mathbb{N}.
\end{equation}
For ${\bf{x}} \in l^s_{\mathbb{R}}(A)$ \ $(1 \leq s  \leq \infty)$,  we defined ${\bf{x}}^{(1)} \in l^s_{\mathbb{R}}(A_1)$ by
\begin{equation} \label{init}
{\bf{x}}^{(1)}(\tau_0\alpha) = {\bf{x}}(\alpha), \ \ \ \alpha \in A,
\end{equation} 
and \  ${\bf{x}}^{(j+1)} \in l^s_{\mathbb{R}}(A_{j+1}),$ \  $ j = 1, 2, \ldots,$ by\\ 
\begin{equation} \label{recur}
{\bf{x}}^{(j+1)}(\tau_j w) = \left \{
\begin{array}{lcc}
\big(Q_{A_j}^{(s)}({\bf{x}}^{(j)}\big)^{\wedge}(w), &\quad w \in W_{A_j} \setminus R_{A_j}, \  \ s \in [1,2],  \\\\
\big(P_{A_j}^{(s)}({\bf{x}}^{(j)}\big)^{\wedge}(w),&\quad w \in W_{A_j} \setminus R_{A_j}, \  \  s \in (2,\infty].
\end{array} \right.
\end{equation}\\
(Cf. \eqref{cascade1} and \eqref{recur1}.) Representations of $l^s$ with respect to $(\Omega_A,\mathbb{P}_A)$ were then derived   
 via recursions nearly identical to those above based on  \eqref{cascade1} and  \eqref{recur1}. 

I prefer cascades leading to the countably infinite product  $(\Omega_{A^{\prime}},\mathbb{P}_ {A^{\prime}})$ in \eqref{disjoint3}, wherein $(\Omega_A,\mathbb{P}_A)$ is the first factor.  Specifically, the "cascade" approach vacates the 'infinite $A$' assumption, and also leads to the  $(l^2 \hookrightarrow L^{\infty})$-ultra-interpolant $\Phi_{A,\delta}$ in Theorem \ref{MT1}, or $\Phi_A^{(2)}$ in Theorem \ref{mainTh}; and thus a view of the Grothendieck inequality as an upgrade of Khintchin's. (Cf. Remark \ref{upgrade2}.)
  \end{remark}
\section{\bf{Interpolation sets (brief overview)}}  \label{interpolation}  
 The proof of Theorem \ref{mainTh} relied on upgrades of interpolation properties of Rademacher systems,
   \begin{equation} \label{LpMp}
\begin{split}
l^s(R_A) &\subset \big(L^{\infty}_{(s)}(\Omega_A,\mathbb{P}_A) \big )^{\wedge} \big |_{R_A}, \ \ \ \ s \in [1,2],\\\
l^s(R_A) &\subset \big(M_{(s)}(\Omega_{A}) \big )^{\wedge} \big |_{R_A}, \ \ \ \ s \in (2,\infty],
\end{split}
\end{equation}
where the instance  $s = 2$  is the dual equivalent of the Khintchin inequality, 
 and the instance $s = \infty$ is the dual equivalent of  
 \begin{equation}
 f \in C_{R_A}(\Omega_A) \ \Rightarrow \ \widehat{f} \in l^1(R_A).
\end{equation} 
The properties in \eqref{LpMp}, formally tagged in \S 6.3, naturally fit in a broader context of interpolation phenomena, which we briefly survey below. 

\
\noindent

\subsection{$\Lambda(p)$-sets} \label{open} \ $E \subset \Gamma$ (= discrete abelian group)  is $\Lambda(p)$, $p \in (0,\infty)$, if $\kappa_q(E,p) < \infty$ for some $q \in (0,p)$.  
(See Remark \ref{Lambda}.) The $\Lambda(p)$ property had been formalized and studied first in the circle group setting $\widehat{\Gamma} = [0,2\pi)$ \citep{rudin1960trigonometric}, but soon was cast in other settings as well (e.g., \citet{bonami1968ensembles}, \citet{bonami1970etude}, \citet{zygmund1974fourier}).
 In dyadic settings, Rademacher systems $R_A$ via the Khintchin inequalities \citep{khintchine1923dyadische}, and more generally, Walsh systems  $W_{A,k}$ via Bonami's inequalities \citep{bonami1968ensembles}, are prototypical $\Lambda(p)$-sets for all  $p > 0$.\footnote{\label{footnote4} Bonami's inequalities became known as  \emph{hypercontractive inequalities}, variants of which had previously appeared in a theoretical physics context  \citep{nelson1966quartic}.  (See \citet{davies1992hypercontractivity} for a detailed survey of hypecontractivity in a mathematical physics context, where the term \emph{hypercontractivity} originated.)
Bonami's inequalities were applied also in a theoretical computer science framework, with their first use in  \citet{kahn1988influence}, in the proof of the KKL theorem; e.g., see \citet{o2008some}.} The question \citep{rudin1960trigonometric}, which became known as the $\Lambda(p)$-\emph{set problem}, whether for a given $p > 0$  there exist  $\Lambda(p)$-sets that are not $\Lambda(s)
$ for any $s > p > 0$, had been resolved affirmatively first for even integers $p \geq 4$ in  \citet{rudin1960trigonometric}; negatively in \citet{bachelis1974lambda} for $p \in (0,2)$ (see also \citet{hare1988elementary}), and affirmatively in \citet{bourgain1989bounded} for  $p \in (2, \infty)$  (see also \citet{talagrand1995sections}).  
At $p = 2$, the $\Lambda(2)$-\emph{set problem} remains open. 

A general notion of $\Lambda(2)$-uniformizability, and, specifically, the  $\Lambda(2)$-modulus of $E \subset \Gamma$ are verbatim transcriptions from the  dyadic setting (Remark \ref{upgrade4}). We have already noted the open question, whether every $\Lambda(2)$-set is $\Lambda(2)$-\emph{uniformizable}, and its connection to the $\Lambda(2)$-set problem; that a negative answer to the uniformizability question would imply the existence of a $\Lambda(2)$-set that is not $\Lambda(s)$ for any $s > 2$ (Lemma \ref{digress1}). The uniformizability question is related also to the \emph{$\Lambda(2)$ union problem}, hitherto open, whether a finite union of $\Lambda(2)$-sets is  $\Lambda(2)$;  an affirmative answer to the former would imply an affirmative answer to the latter. (Cf. \citet[Ch. III \S 6]{Blei:2001}.)\

\noindent
\subsection{$p$-Sidon sets}
\begin{definition} \label{defnp}
$E \subset \Gamma$ is $p$-Sidon, $p \in [1,2)$, if
\begin{equation} \label{Sidonp}
\big \{\hat{f}: f  \in C(\widehat{\Gamma}), \  \widehat{f} = 0 \ \text{on} \  E^c \big \} \  \overset{\text{def}}{=} \  \widehat{C_E} \subset l^p(E),
\end{equation}
or, equivalently (by duality),  if  
\begin{equation} \label{interpolate3}
l^{p^*}(E) \subset \big({M(\widehat{\Gamma})}\big)^{\wedge}\big |_E, \ \  \ \ \  {p^*}  = \frac{p}{p-1}.
\end{equation}\\
\end{definition}

The case $p = 1$ was considered first in  \citet{rudin1960trigonometric}, wherein  $E \subset \mathbb{Z}$ \  satisfying \eqref{Sidonp} with $p = 1$  were dubbed \emph{Sidon sets} -- homage to the result in \citet{sidon1927satz}  that \emph{lacunary} $E \subset \mathbb{Z}^+$  satisfy \eqref{Sidonp} with $p = 1$. The notion of  $p$-\emph{Sidonicity} was formalized later in \citet{edwards1974p}.  Over the years, studies of \emph{Sidonicity} have indeed branched out in different directions and various settings.  (For a detailed survey, see  \citet{graham2013interpolation}.  Beware:  in additive number theory, a \emph{Sidon set} has altogether a different meaning, which originated in \citet{sidon1932satz}; see \citet{o2004complete}.) 
  
The Rademacher system $R_A \  (= W_{A,1})$, an analogue of a lacunary subset of  $\mathbb{Z}^+$, is an archetypal $1$-Sidon set.   Littlewood's $4/3$-inequality  \citep{Littlewood:1930} (in part a corollary to the $(C_{R \times R} \hookrightarrow l^{1,2})$-inequality) asserts that if $A$ is infinite, then  $W_{A,2} \subset W_A $ is an \emph{exact} $4/3$-Sidon set, i.e., $W_{A,2}$ is $p$-Sidon  $\Leftrightarrow  p \geq 4/3.$  \emph{Exact} $p$-Sidon sets  for arbitrary  $p \in (1,2)$ were produced in \citet{blei1979fractional}.  We know more about $1$-Sidonicity (the extremal case) than about $p$-Sidonicity (the full scale).  For example, the union of finitely many $1$-Sidon sets is known to be $1$-Sidon  \citep{drury1972unions},  but whether the same holds for $p$-Sidon sets,   $p \in (1,2)$, is open. 

\begin{remark}[$p$-Sidon constant] \ $E \subset \Gamma$ is $p$-Sidon, $p \in [1,2)$, if and only if
\begin{equation}
\sigma_p(E) \ \overset{def}{=} \ \sup \big\{ {\|\widehat{f}\|_p}/{\|f\|_{L^{\infty}}}: E \text{-polynomials}, \ f \neq 0 \big\} \ < \ \infty.
\end{equation}
By duality, in the 'language' of interpolants (Remark \ref{temp}),  
\begin{equation}
\sigma_p(E) \ {=} \ \inf \big \{\|G\|_{l^{p^*} \hookrightarrow M}: \big(l^{p^*}(E) \hookrightarrow M(\widehat{\Gamma})\big)\text{-interpolants}   \ G \big \}.
\end{equation}
\end{remark}

\noindent
\subsection{The $L(p)$ and $S(p)$ properties}
 \ The  $\Lambda(2)$  and 
$1$-Sidon properties have extensions, which are exemplified by \eqref{LpMp}, and are ostensibly distinct from $\Lambda(p)$ and $p$-Sidonicity.  Let
\begin{equation}
\begin{split}
L^{\infty}_{(p)}(\widehat{\Gamma},\mathfrak{m}) \ &\overset{\text{def}}{=}  \ \big \{f \in L^{\infty}(\widehat{\Gamma},\mathfrak{m}): \hat{f} \in l^p(\Gamma) \big \},  \ \ \ \ p \in [1,2],\\\
M_{(p)}(\widehat{\Gamma}) \  &\overset{\text{def}}{=} \  \big \{\lambda \in M(\widehat{\Gamma}): \Hat{\lambda} \in l^p(\Gamma) \big \}, \ \ \ \ p \in (2,\infty].
\end{split}
\end{equation} \ \\
(Cf. \eqref{xtendp} and  \eqref{quasi}.) 

 \ 
\begin{definition}  \label{SL} \ 
 $E \subset \Gamma$ is an $L(p)$-set for $p \in [1,2]$, if for every  $\phi \in l^p(E)$ there exist  $f \in L^{\infty}_{(p)}(\widehat{\Gamma},\mathfrak{m})$ such that 
 \begin{equation}
 \widehat{f}\big|_E \ = \ \phi,
 \end{equation}
 and an $S(p)$-set for $p \in (2,\infty]$, if for every  $\phi \in l^p(E)$ there exist  $\mu \in M_{(p)}(\widehat{\Gamma})$ such that 
 \begin{equation}
 \widehat{\mu}\big|_E \ = \ \phi.
 \end{equation}
 \end{definition}

\

By duality, as in \S \ref{more}, \  $E \subset \Gamma$ an $L(p)$-set,   $p \in [1, 2]$, if and only if \\
\begin{equation}
\kappa^{(p)}(E) \ \overset{def} = \ \inf \big \{\|G\|_{l^p \hookrightarrow L^{\infty}_{(p)}}: \big(l^p(E) \hookrightarrow L^{\infty}_{(p)}(\widehat{\Gamma},\mathfrak{m})\big)\text{-interpolants}   \ G \big \} \ < \ \infty,
\end{equation}\\
and an $S(p)$-set,  $p \in (2, \infty]$, if and only if \\
\begin{equation}
\sigma^{(p)}(E) \ \overset{def} = \ \inf \big \{\|G\|_{l^p \hookrightarrow M_{(p)}}: \big(l^p(E) \hookrightarrow M_{(p)}(\widehat{\Gamma})\big)\text{-interpolants}   \ G \big \} \ < \ \infty.
\end{equation}\\

Gauging the 'size' of perturbations, we compute \ $\beta^{(p)}(E,\delta)$,  $\delta \geq 0$, as in \eqref{mod2}. We say  $E \subset \Gamma$ is $L(p)$-\emph{uniformizable} ($1 \leq p \leq 2$), or $S(p)$-\emph{uniformizable} ($2 < p \leq \infty$), if \ $\beta^{(p)}(E,\delta) < \infty$ \  for \  $\delta > 0$ \ in the corresponding ranges of $p$.  

To mark $\big(l^p(E) \rightarrow l^p(\Gamma)\big)$-continuity of transforms, we define $\beta^{(p),c}(E,\delta)$ as in \eqref{mod4}, and dub $E \subset \Gamma$ \emph{continuous}-$L(p)$-\emph{uniformizable} ($1 \leq p \leq 2$), or  \emph{continuous}-$S(p)$-\emph{uniformizable} ($2 < p \leq \infty$), if  $\beta^{(p),c}(E,\delta) < \infty$ ($\delta > 0$) in the respective ranges of $p$.

 I know scant little about the $L(p)$ and $S(p)$ properties, beyond their roles in the proof of Theorem \ref{mainTh}.  Two questions I thought about and could not answer are stated in the remarks below; and there are others that will almost surely occur to the reader...  

\begin{remark}[$S(p)$-sets] 
 For $p \in (2,\infty]$, every $S(p)$-set is \emph{a fortiori}  $q$-Sidon for every $q \in [p^{\star},2)$, $p^{\star}  = \frac{p}{p-1}$.  That is, for $E \subset \Gamma$ and $p \in (2,\infty]$, 
 \begin{equation}
 \sigma_{q}(E) \ \leq \ \sigma^{(p)}(E), \ \ \ \ \ p^{*} \leq q < 2.
 \end{equation}
 In particular, $p^*$-Sidonicity is upgraded by the $S(p)$ property, which asserts that interpolants can be chosen to be $M_{(p)}$-valued, and not merely $M$-valued;  see  \S \ref{upgrade3} ("notion of an upgrade").
 
 The $M_{(p)}(\Omega_A)$-valued Riesz products (Lemma \ref{Key}) imply that Rademacher systems are continuous $S(p)$-uniformizable for every $p \in (2,\infty]$.  It can be similarly shown with a bit more work, specifically by modifying arguments in \citet{drury1972unions} (also based on Riesz products), that 1-Sidon sets are continuous-$S(p)$-uniformizable for every $p \in (2,\infty]$. 
 
 In Part II of the monograph, we will deduce that $F \subset W_{A,n}$ is an $S(p)$-set for all $p \in (2, \frac{2\dim F}{\dim F - 1}]$, \ $n = 2, \ldots$, where $\dim F$ is the \emph{combinatorial dimension} of $F$ \citep{blei1984combinatorial1}.
  
 \begin{question}
 Does $q$-Sidon $\Rightarrow$ $S(q^*)$-set, \ $q \in (1,2)$?
 \end{question}  
 \end{remark} 
 
\

\begin{remark} [$L(p)$-sets] \label{persp3} \ For every $E\subset \Gamma,$ 
\begin{equation}
U_E{\bf{x}} \  \in \  L^{\infty}_{(1)}(\widehat{\Gamma},\mathfrak{m})  \ \ \    \text{for all } \ {\bf{x}} \in l^1(E),
\end{equation}
i.e., every $E \subset \Gamma$ is $L(1)$ (trivially!). For  \ $1 < p \leq 2$,  
\begin{equation} \label{persp1}
U_E{\bf{x}}  \in   L^{\infty}_{(p)}(\widehat{\Gamma},\mathfrak{m})  \ \ \  \text{for all } \ {\bf{x}} \in l^p(E)
\end{equation}
only for finite $E \subset \Gamma$, and otherwise there exist an orthogonal $(l^p,L^{\infty}_{(p)})$-perturbations \ $\mathfrak{p}_E: l^p(E) \rightarrow \mathfrak{S}_{\Gamma \setminus E},$ \  i.e., 
\begin{equation}
U_E{\bf{x}} + \mathfrak{p}_E({\bf{x}}) \in L^{\infty}_{(p)}(\widehat{\Gamma},\mathfrak{m})   \ \ \ \text{for all } \ {\bf{x}} \in l^p(E),
\end{equation}\
if and only if $E \subset \Gamma$  is  $L(p)$.\\

 We will note in Part II that  for every $n = 1, \ldots, $ \ $W_{A,n}$ is $L(p)$ for all $p \in [1,2]$.

\noindent

\begin{question}
Does $\Lambda(2)$ $\Rightarrow$ $L(p)$,\  $p \in (1,2)$?
\end{question}
 \end{remark}

\bibliographystyle{apalike}
\bibliography{blei}
\vspace{.3in}
\end{document}